\documentclass[oneside,a4paper,english, 11pt]{amsart}

\usepackage{adjustbox}
\usepackage{tikz-cd}
\usepackage[latin9]{inputenc}
\usepackage{hyperref}
\usepackage{lscape}
\usepackage{slashbox}
\usepackage{mathrsfs}  
\usepackage{enumerate} 
\usepackage{bm} 

\makeatletter
 \theoremstyle{plain}
\newtheorem{thm}{Theorem}[subsection]
\theoremstyle{plain}
  \newtheorem{prop}[thm]{Proposition}
\theoremstyle{plain}
 \newtheorem{lemma}[thm]{Lemma}
\theoremstyle{plain}

\theoremstyle{plain}
\newtheorem{cor}[thm]{Corollary}
\theoremstyle{definition}
  \newtheorem{defn}[thm]{Definition}
\theoremstyle{definition}
  
 \theoremstyle{definition}
  
  \newtheorem{exam}[thm]{Example}
\theoremstyle{remark}
\newtheorem{rmk}[thm]{Remark}
\numberwithin{equation}{section}

\usepackage{amsmath}
\usepackage{amssymb}
\usepackage{amscd}
\usepackage{amsthm}
\usepackage{array,multirow}
\usepackage[arrow,matrix,tips,all,cmtip,poly,color]{xy}

\usepackage{stmaryrd}
\usepackage{url}
\usepackage{color}
\usepackage{caption}

\pdfpageheight\paperheight
\pdfpagewidth\paperwidth
\newcommand{\Z}{\mathbb{Z}}

\newcommand{\Q}{\mathbb{Q}}
\newcommand{\QQ}{\mathbb{Q}}

\newcommand{\Qp}{\mathbb{Q}_p}

\newcommand{\R}{\mathbb{R}}
\newcommand{\C}{\mathbb{C}}
\newcommand{\bG}{\mathbb{G}}

\newcommand{\F}{\mathbb{F}}

\newcommand{\N}{\mathbb{N}}

\newcommand{\A}{\mathbb{A}}

\newcommand{\fF}{\mathfrak{f}}

\newcommand{\fM}{\mathfrak{M}}

\newcommand{\fp}{\mathfrak{p}}

\newcommand{\fS}{\mathfrak{S}}

\newcommand{\fm}{\mathfrak{m}}

\newcommand{\bA}{\mathbb{A}}

\newcommand{\bT}{\mathbb{T}}

\newcommand{\bV}{\mathbb{V}}

\newcommand{\cA}{\mathcal{A}}

\newcommand{\cC}{\mathcal{C}}

\newcommand{\cE}{\mathcal{E}}

\newcommand{\cG}{\mathcal{G}}

\newcommand{\cI}{\mathcal{I}}

\newcommand{\cJ}{\mathcal{J}}

\newcommand{\cM}{\mathcal{M}}

\newcommand{\cO}{\mathcal{O}}
\newcommand{\cP}{\mathcal{P}}

\newcommand{\cR}{\mathcal{R}}

\newcommand{\cU}{\mathcal{U}}

\newcommand{\cX}{\mathcal{X}}

\newcommand{\cZ}{\mathcal{Z}}

\newcommand{\eps}{\varepsilon}

\newcommand{\phz}{\varphi}

\newcommand{\Zp}{\mathbb{Z}_p}

\newcommand{\Id}{\mathrm{id}}
\newcommand{\Gal}{\mathrm{Gal}}
\newcommand{\Hom}{\mathrm{Hom}}
\newcommand{\Ext}{\mathrm{Ext}}
\newcommand{\Res}{\mathrm{Res}}
\newcommand{\Ind}{\mathrm{Ind}}

\newcommand{\End}{\mathrm{End}}

\newcommand{\GL}{\mathrm{GL}}

\newcommand{\red}{\mathrm{red}}

\newcommand{\Spec}{\mathrm{Spec}\ }

\newcommand{\Frob}{\mathrm{Frob}}

\newcommand{\id}{\mathrm{id}}

\newcommand{\Adm}{\mathrm{Adm}}
\newcommand{\speci}{\mathrm{sp}}
\newcommand{\semis}{\mathrm{ss}}

\newcommand{\Fp}{\F_p}

\newcommand{\un}[1]{\underline{#1}}
\renewcommand{\bf}[1]{\mathbf{#1}}
\newcommand{\Rep}{\mathrm{Rep}}
\newcommand{\Diag}{\mathrm{Diag}}

\newcommand{\tld}[1]{\widetilde{#1}}

\newcommand{\JH}{\mathrm{JH}}
\newcommand{\rig}{\mathrm{rig}}

\newcommand{\supp}{\mathrm{Supp}}

\newcommand{\rbar}{\overline{r}}
\newcommand{\rhobar}{{\overline{\rho}}}
\newcommand{\taubar}{\overline{\tau}}

\newcommand{\Spf}{\mathrm{Spf}}

\newcommand{\rG}{\mathrm{G}}

\newcommand{\defeq}{\stackrel{\textrm{\tiny{def}}}{=}}
\newcommand{\s}{^\times}

\newcommand{\ovl}[1]{\overline{#1}}

\newcommand{\obv}{\mathrm{extr}}
\newcommand{\mord}{\mathrm{mord}}

\newcommand{\orient}{\mathrm{or}}

\newif\iffinalrun
\iffinalrun
  \newcommand{\mar}[1]{}
\else
  \newcommand{\mar}[1]{\marginpar{\raggedright\tiny #1}}

\DeclareMathOperator{\Mod}{Mod}
\DeclareMathOperator{\Coh}{Mod}

\DeclareMathOperator{\Lie}{Lie}

\DeclareMathOperator{\Ad}{Ad}

\DeclareMathOperator{\Mat}{Mat}

\DeclareMathOperator{\Gr}{Gr}
\DeclareMathOperator{\Fl}{Fl}

\DeclareMathOperator{\Conv}{Conv}

\DeclareMathOperator{\Iw}{\cI}

\DeclareMathOperator{\univ}{univ}

\DeclareMathOperator{\pr}{pr}

\newcommand{\ra}{\rightarrow}

\newcommand{\iarrow}{\hookrightarrow}
\newcommand{\into}{\hookrightarrow}

\newcommand{\onto}{\twoheadrightarrow}

\newcommand{\risom}{\buildrel\sim\over\rightarrow} 
\title{Extremal weights and a tameness criterion for mod $p$ Galois representations}
\author{Daniel Le}
\address{Department of Mathematics,
Purdue University,
150 N. University Street, 
West Lafayette, IN 47907-2067}
\email{ledt@purdue.edu}

\author{Bao V.~Le Hung}
\address{Department of Mathematics,
Northwestern University, 
2033 Sheridan Road, 
Evanston, Illinois 60208, USA}
\email{lhvietbao@googlemail.com}

\author{Brandon Levin}
\address{Department of Mathematics,
Rice University, 
6100 Main Street,
Houston, TX 77005}
\email{bwlevin@rice.edu}

\author{Stefano Morra}
\address{Laboratoire d'Analyse, G\'eom\'etrie, Alg\`ebre,
Universit\'e Paris 13 et Paris 8,
99 ave. J-B. Cl\'ement,
93430 Villetaneuse,
France }
\email{morra@math.univ-paris13.fr}

\begin{document}

\begin{abstract}
We study the weight part of Serre's conjecture for generic $n$-dimensional mod $p$ Galois representations.   We first generalize Herzig's conjecture to the case where the field is ramified at $p$ and prove the weight elimination direction of our conjecture.  We then introduce a new class of weights associated to $n$-dimensional local mod $p$ representations which we call \emph{extremal weights}.  Using a ``Levi reduction" property of certain potentially crystalline Galois deformation spaces, we prove the modularity of these weights.  As a consequence, we deduce the weight part of Serre's conjecture for unit groups of some division algebras in generic situations. 

\end{abstract}

\maketitle

\tableofcontents

\clearpage{}%
\section{Introduction}

In \cite{serre-duke}, Serre announced his celebrated modularity conjecture for two-dimensional mod $p$ Galois representations (now a theorem of Khare--Wintenberger \cite{KW_Serre_1}). 
Much of that paper is dedicated to formulating a strong version which predicts the minimal level and weight at which one would find a suitable modular form in terms of the restriction of the mod $p$ Galois representation to decomposition groups. 
There has been enormous progress on generalizing Serre's strong version to other number fields and higher dimensions as well as proving that the strong version follows from the weak version. 
This includes rather complete results in dimension two \cite{coleman-voloch,edixhoven,gross,BDJ,schein,gee-inv,newton13,GLS,GLS15} which have led to a much better understanding of the hypothetical mod $p$ Langlands correspondence \cite{BP,BHHMS2}. 
Still, there has been no complete generalization of the weight part of Serre's conjecture in any dimension larger than two, reflecting difficulties in integral $p$-adic Hodge theory. 

Already in Serre's original paper, the most subtle case occurs when the restriction $\rhobar$ to the decomposition group at $p$ is \emph{not} semisimple: the \emph{tr\`es ramifi\'ee} extensions have fewer modular Serre weights than the \emph{peu ramifi\'ee} ones. 
Still in dimension two but for unramified extensions of $\Q_p$, the set of Serre weights attached to a nonsplit extension by \cite{BDJ} depends in a subtle way on the extension class even in \emph{generic} cases (see also \cite{DDR,CEGM}). 
In higher dimensions, \cite{herzig-duke,GHS} define a conjectural weight set for unramified extensions of $\Q_p$ only when $\rhobar$ is generic and \emph{semisimple}. 
This conjecture was proven in \cite{MLM} (for many definite unitary groups) under a somewhat exotic genericity condition---again only for semisimple $\rhobar$. 
Further, \cite[Conjecture 9.1.5]{MLM} extended (still for unramified extensions of $\Q_p$) the weight part of Serre's conjecture to generic non-semisimple Galois representations using stacks defined by Emerton--Gee \cite{EGstack} and their local models. 
It predicts that $\rhobar$ is modular of a Serre weight $\sigma$ if $\rhobar$ lies on a certain cycle on the Emerton--Gee stack attached to $\sigma$ using known cases of the Breuil--M\'ezard conjecture. 
It seems that this geometric perspective is essential because already in dimension three the set of modular weights cannot be described by linear conditions on extension classes \cite{GL3Wild}. 
At present, this conjecture seems difficult to access for several reasons. 
First, proving the conjecture using the Breuil--M\'ezard perspective would require showing that various Taylor--Wiles patched modules have full support which is a notoriously difficult problem. 
(The relevant deformation rings are not always integral domains when $\rhobar$ is not semisimple.) 
Second, the cycle attached to $\sigma$ depends on cycle multiplicities appearing in local models that are not well understood at present. 
Finally, a classification of possible weight sets seems complicated because the formal definition of the Emerton--Gee stack makes it difficult to understand basic geometric questions such as the intersection patterns of its irreducible components. 
The goal of this paper is to establish upper and lower bounds for the set of modular Serre weights which apply to generic non-semisimple Galois representations in arbitrary dimension  over possibly ramified extensions of $\Qp$.
When applicable, our bounds greatly improve the existing ones in the literature (cf.~Remark \ref{rmk:bounds}). 
They arise from our study of (potentially crystalline) Emerton--Gee stacks.
The new symmetries and geometric structures we observe should play a role in organizing further investigations on Serre weight conjectures.
Our results are new even in the semisimple case: we formulate a generalization of Herzig's  conjecture for ramified $p$-adic fields which we prove is always an upper bound.
We give two concrete applications of our bounds.
Under a mild and explicit genericity condition, we establish the weight part of Serre's conjecture for unit groups of certain division algebras and generalize Gross' automorphic tameness criterion.

\subsection{Results}\label{sec:global:prelim}

Let $p$ be a prime and $n \geq 2$ be an integer. Let $F/F^+$ be a CM extension of a totally real field $F^+\neq \Q$.  
Assume for the sake of exposition that there is a single place $v$ of $F^+$ dividing $p$ which splits in $F$. (Our results apply whenever all the places of $F^+$ dividing $p$ split in $F$.)
Let $G$ be a definite unitary group over $F^+$ split over $F$ which is isomorphic to $U(n)$ at each infinite place and split at $v$. 
A (\emph{global$)$ Serre weight} is an irreducible smooth $\overline{\F}_p$-representation $V$ of $G(\cO_{F^+, v})$,
i.e.~the inflation to $G(\cO_{F^+, v})$ of an irreducible $\overline{\F}_p$-representation of $G(k_v)$, where $k_v$ is the residue field of $F^+$ at $v$.
For a mod $p$ Galois representation $\rbar:G_F \rightarrow \GL_n(\overline{\F}_p)$, let $W(\rbar)$ denote the collection of modular Serre weights for $\rbar$. 
That is, $V \in W(\rbar)$ if the Hecke eigensystem attached to $\rbar$ appears in a space of mod $p$ automorphic forms on $G$ of weight $V$ for some prime to $v$ level. 

Fix a place $\tld{v}$ of $F$ dividing $v$ which identifies $G(k_v)$ with $\GL_n(k_v)$. 
Define $\rbar_v := \rbar|_{\Gal(\overline{F}_{\tld{v}}/F_{\tld{v}})}$. 
The goal of the weight part of Serre's conjecture is to predict $W(\rbar)$ in terms of $\rbar_v$ or more precisely, the restriction of $\rbar_v$ to inertia.

Our global (and local) results include genericity conditions on $\rbar_v$ which will be made precise in the body of the paper.  
We stress that our genericity conditions are completely explicit, unlike those of \cite{MLM}.
We note however that for most results the genericity conditions require $p$ to be at least $O(en^2)$ (where $e$ is the absolute ramification index of $F_{\tld{v}}$). 

Let $K/\Qp$ be a finite extension with residue field $k$.  For any tame $n$-dimensional $\overline{\F}_p$-representation $\overline{\tau}$ of $I_K \subset \Gal(\overline{K}/K)$ which extends to $\Gal(\overline{K}/K)$, one associates a Deligne--Lusztig representation $V(\overline{\tau})$ of $\GL_n(k)$ (generalizing \cite[Proposition 9.2.1]{GHS}) which is defined over a finite extension $E/\Qp$.  It is also a representation of $\GL_n(\cO_K)$ by inflation.      Recall also the operator $\mathcal{R}$ (see \cite[\S 9.2]{GHS}) on the set of irreducible $\overline{\F}_p$-representations of $\GL_n(k)$ (i.e.~the set of Serre weights).  

If $K$ is unramified over $\Qp$ and $\rhobar$ is tame and generic, then Herzig defined the collection $W^?(\rhobar) = \{ \mathcal{R}(\sigma) \mid \sigma \in \JH(\overline{V}(\rhobar|_{I_K}))\}$.     
In the ramified setting, we make the following generalization: 

\begin{defn} \label{intro:weightset}
If $\rhobar$ is tame and generic, we define
\[
W^?(\rhobar)\defeq \left\{ \mathcal{R}(\sigma) \mid \sigma  \in \Big(\JH\Big(\overline{V}(\rhobar|_{I_K}) \otimes \overline{W}(0, 1-e, 2(1-e), \ldots, (n-1)(1-e))\Big)\Big) \right\}.
\]
where $W(0, 1-e, 2(1-e), \ldots, (n-1)(1-e))$ is the irreducible algebraic representation of (parallel) highest weight $(0, 1-e, 2(1-e), \ldots, (n-1)(1-e))$.
\end{defn}

\begin{rmk} 
\begin{enumerate}
\item In \cite{MLM} (see Theorem 4.7.6), in the unramified case, we give a geometric interpretation of Herzig's $W^?(\rhobar)$ in terms of torus fixed points on certain subvarieties of the affine flag variety.   Although we don't directly use this description here because of a lack of local model theory in the ramified case, it motivated Definition \ref{intro:weightset}. 
\item When $n = 2$, Schein gave in \cite{schein} an explicit description of a weight set for tamely ramified $\rhobar$. The two sets agree when $\rhobar$ is sufficiently generic, cf.~\S \ref{subsub:schein}.   
\end{enumerate}
\end{rmk}

We prove the weight elimination direction generalizing \cite{LLL}:
\begin{thm}[``Weight elimination'', cf.~Theorem \ref{thm:WE}] 
\label{intro:we}
Suppose that $\rbar:G_F \rightarrow \GL_n(\overline{\F}_p)$ satisfies standard Taylor--Wiles hypotheses and that $\rbar_v$ is tame and sufficiently generic.  
If $\sigma$ is a sufficiently generic Serre weight, then
\[
\sigma \in W(\rbar) \Longrightarrow \sigma \in  W^?(\rbar_v).
\]
\end{thm}

When $\rhobar$ is not tame, unless $n \leq 3$, we don't have an analogue of $W^?(\rhobar)$.  Historically, certain classes of Serre weights have been identified which are expected to belong to $W^?(\rhobar)$.   For example, Gee--Geraghty proved very generally the modularity of \emph{ordinary} weights, i.e., those weights for which $\rbar_v$ admits ordinary crystalline lifts. 
For tame $\rhobar$ and $K$ unramified, \cite{GHS} introduce a notion of \emph{obvious weight} which roughly speaking are characterized by the property that $\rhobar$ admits an ``obvious" crystalline lift of specified Hodge--Tate weights, namely a sum of inductions of characters.
Building on what we discovered when $n =3$ in \cite{GL3Wild}, we introduce a notion of \emph{extremal} weights $W_{\obv}(\rhobar)$ which encompasses (in generic cases) both these earlier notions and prove the following theorem:

\begin{thm}[``Modularity of extremal weights'', Theorem \ref{thm:globalobv}] \label{intro:thm:globalobv}
Let $\rbar: G_{F} \ra \GL_n(\F)$ be an automorphic representation satisfying standard Taylor--Wiles conditions and such that $\rbar_v$ is sufficiently generic.   If either $\rbar$ is potentially diagonalizably automorphic or $W_{\obv} (\rbar_v) \cap W(\rbar)$ is non-empty, then 
\[
W_{\obv} (\rbar_v) \subset W(\rbar).
\] 
\end{thm}

\noindent When $\rbar$ is tamely ramified at $v$ and $F^+$ is unramified at $p$, Theorem \ref{intro:thm:globalobv} was proven in \cite{LLL} and played a critical role in the theory of local models in \cite{MLM} and its applications to the weight part of Serre's conjecture and the Breuil--M\'ezard conjecture. 
We expect Theorem \ref{intro:thm:globalobv} to have similar applications when $F^+$ is ramified at $p$ and hope to return to this in future work. 

There are two main ingredients in the proof of Theorem \ref{intro:thm:globalobv}: a geometric one (which will be discussed in the next section) and a combinatorial one.
The combinatorial ingredient is a hidden Weyl group symmetry.
When $\rhobar$ is tame and generic, then $W_{\obv}(\rhobar)$ is naturally a torsor for a product of $[k:\Fp]$-copies of the Weyl group $S_n$ of $\GL_n$, as explained in \cite{GHS}.    
Wildly ramified $\rhobar$ have fewer weights in general and fewer extremal weights (see Proposition \ref{prop:tamecrit}), but it turns out that the symmetry can be restored by enhancing an extremal weight with the data of a \emph{specialization}.

A tame inertial $\overline{\F}_p$-type is a continuous tame representation $I_K \ra \GL_n(\overline{\F}_p)$  which admits an extension to $G_K$.   Tame inertial $\overline{\F}_p$-types admit a combinatorial description in terms of fundamental characters of $G_K$ (see \S \ref{subsubsec:TIT}).  To a generic $\rhobar$, we attach a collection of tame inertial $\overline{\F}_p$-types  which we call (extremal) \emph{specializations} (Definition \ref{defn:spec}).   
This notion is somewhat elaborate, relying on the geometry of the Emerton--Gee stack (see e.g.~\S \ref{sub:extr:loc}).
The semisimplification of $\rhobar$ restricted to $I_K$ is a prototypical example of a specialization but there are always others when $\rhobar$ is not tame.   
It is generally expected that the predicted Serre weights of a wildly ramified $\rhobar$ should be a subset of those of $\rhobar^{\semis}$.   What we discover is the same is true for the other specializations of $\rhobar$ as well.   

\begin{thm} [cf.~Theorem \ref{thm:WE}]
\label{intro:wildewe} Suppose $\rbar:G_F \rightarrow \GL_n(\overline{\F}_p)$ satisfies standard Taylor--Wiles hypotheses and that $\rbar_v$ is sufficiently generic.  
Let $\rbar_v^{\speci}$ be a specialization of $\rbar_v$.   
If $\sigma$ is a sufficiently generic Serre weight, then
\[
\sigma \in W(\rbar) \Longrightarrow \sigma \in  W^?(\rbar^{\speci}_v).
\]
\end{thm}

 The proof follows from a purely local result, showing that if $\rbar_v$ admits a tamely potentially crystalline lift of type $(\tau, (n-1, n-2, \ldots, 0))$ then so does any extension of the specialization of $\rbar_v^{\speci}$ to $G_K$, combined with the same weight elimination combinatorics used in the tame case.
 
\begin{rmk} 
Theorems \ref{intro:thm:globalobv} and \ref{intro:wildewe} interact in an interesting way to constrain $W(\rbar)$.
Weight elimination can be used to produce further extremal Serre weights which in turn can be used to produce further specializations, and hence further weight elimination. 
Each iteration sharpens the bounds on $W(\rbar)$. 
\end{rmk}

\begin{rmk}
\label{rmk:bounds}
\begin{enumerate}
\item
When $F^+_v/\Qp$ is unramified the assumption that $\sigma$ is sufficiently generic can be removed (\cite{LLL}).
A similar argument can remove this assumption even in the ramified case, but we don't pursue this here.
\item
Theorem \ref{intro:thm:globalobv} improves on previous lower bounds for the set of modular weights given by the sets of ordinary weights \cite{gee-geraghty} and Fontaine--Laffaille weights \cite{BLGG}. 
For example, if $\rbar_v$ is completely reducible and generic, then $\rbar_v$ has $n!$ ordinary weights, $n^f$ Fontaine--Laffaille weights, and $(n!)^f$ extremal weights where $f$ denotes $[k_v:\F_p]$. 

\item
The only previously known upper bound for $W(\rbar)$ was $W^?(\rbar_v^{\semis})$ (when $F_v^+/\Q_p$ is unramified) \cite{LLL}. 
If $\rbar_v$ has $(n!)^f$ specializations (which is the generic behavior geometrically on the Emerton--Gee stack), Theorem \ref{intro:wildewe} gives a far better upper bound. 
For instance when $n=3$, $F_v^+/\Q_p$ is unramified, and $\rbar_v$ has $6^f$ specializations, Theorem \ref{intro:wildewe} gives an upper bound with a set of size at most $2^f$ rather than $9^f$. 

\item
When $n=2$ and $F_v^+/\Q_p$ is unramified, the lower and upper bounds are equal and agree with the set defined in \cite{BDJ}. 
When $n=3$ and $F_v^+/\Q_p$ is unramified, the lower and upper bounds are close and play a major role in the resolution of the Serre weight conjecture in generic cases \cite{GL3Wild}. 
\end{enumerate}
\end{rmk}

A byproduct of our methods is an automorphic tameness criterion in the spirit of  \cite{gross}.
When $n =2$ and $F = \Q$,  Gross's tameness criterion says that for generic modular $\rbar$, tameness of $\rbar$ at $p$ is equivalent to $W(\rbar)$ having two distinct Serre weights (as opposed to one).  Here we show a similar criterion in terms of the modularity of two extremal weights.  

\begin{thm}[``Automorphic tameness criterion'', Theorem \ref{thm:aut:tameness}]
\label{intro:tamenesscrit}
Let $\sigma_v,\sigma'_{v}\in W_{\obv}(\rbar_v^{\semis})$ be extremal weights of $\rbar_v^{\semis}$ which differ by the longest element $w_0$ under the Weyl group symmetry.   Suppose that $ \sigma_v\in W(\rbar)$ and that $\rbar_v$ is sufficiently generic.  
Then the following are equivalent:
\begin{enumerate}
\item $\sigma'_{v}\in W(\rbar)$; and
\item $\rbar_v$ is tame. 
\end{enumerate}
\end{thm}

\begin{rmk}
In the case where $W(\rbar)$ contains a lowest alcove weight, our methods also give a refined version of the tameness criterion, showing that automorphic information even detects the stratum of $\rbar_v$ in the moduli of Fontaine--Laffaille representations (with respect to a natural partition).
This idea plays a crucial role in \cite{LGC}.
\end{rmk}

Finally, we discuss our results on the weight part of Serre's conjecture for division algebras. 
When $G$ is an anisotropic mod center inner form of $\GL_n$ locally at $v$, Serre weights lift to characteristic zero, and hence the modularity of a Serre weight can be rephrased in terms of the existence of automorphic lifts of specified types.
By local-global compatibility, a necessary condition for the modularity of a generic Serre weight $\chi_v$ is the existence of a lift which is potentially crystalline at $v$ of type $(\tau(\chi_v), (n-1, n-2, \ldots, 0))$ for a certain tame cuspidal type $\tau(\chi_v)$.
We prove the converse under some hypotheses.

\begin{thm}[Serre weights for division algebras, Theorem \ref{thm:divalg}] \label{thm:intro:divalg}
Suppose that $v$ is unramified in $F^+$, that $G$ is an anisotropic mod center inner form of $\GL_n$ locally at $v$, that $\rbar: G_F \ra \GL_n(\ovl{\F}_p)$ satisfies standard Taylor--Wiles hypotheses and is potentially diagonalizably automorphic, and that $\rhobar_v$ is sufficiently generic.
Then $\chi_v \in W(\rbar)$ if and only if $\rhobar_v$ admits a potentially crystalline lift of type $(\tau(\chi_v), (n-1, n-2, \ldots, 0))$. 
\end{thm}

\noindent This generalizes results in \cite{Gee-Savitt} when $n=2$ and \cite{dottoDA} (announced in 2019) when $\rbar$ is tamely ramified at $v$. 
The main difficulty lies in the construction of automorphic lifts. 
One has access to powerful potentially Barsotti--Tate modularity lifting results when $n=2$ \cite{gee-kisin} that are not available in general. 
Via a transfer \`a la Jacquet--Langlands, it suffices to construct automorphic lifts for a group which is quasisplit at places dividing $p$. Generalizing \cite{LLL} in the tame case, we construct the desired automorphic lifts using the modularity of extremal weights. 
In fact, we prove a more general criterion for the existence of automorphic lifts where $\tau(\chi_v)$ is replaced by any sufficiently generic tame inertial type (Corollary \ref{cor:changetype}). 
That the modularity of extremal weights suffices to produce the desired automorphic lifts reduces to combinatorics of the extended affine Weyl group when $\rbar$ is tamely ramified at $v$. 
In general, it requires an analysis of the geometry of local models (\S \ref{sub:extr:loc}). 

\subsection{Geometric methods: Levi reduction for deformation rings}

We now explain how we go about proving the modularity of the extremal weights (Theorem \ref{intro:thm:globalobv}). As mentioned earlier, this notion expands on the notions of ordinary and obvious weights. Gee--Geraghty \cite{gee-geraghty} proved modularity of ordinary weights in considerable generality. Three of the authors proved modularity of obvious weights in the unramified and generic semisimple case (\cite{LLL}). Both of these results rely on producing potentially diagonalizable lifts of some prescribed type. Using these methods, we can access some but not all of the extremal weights. Instead, we adopt the strategy of \cite{LLLM}
and exploit the symmetry of our situation.   

As described above, extremal weights when enhanced with the data of a specialization admit a Weyl group symmetry.  The main point is to show that if two extremal weights $\sigma, \sigma'$ are related by a simple reflection then the modularity of one implies the modularity of the other.  To do this, we show that we can find a sequence of well-chosen tame types $\tau_0, \ldots, \tau_{2e}$ \emph{connecting} $\sigma$ to $\sigma'$ where we can establish good combinatorial behavior of Serre weights and show that the Galois deformation rings are integral domains.     
The following is the main result on deformation rings that we use. 

\begin{thm}[particular case of Theorem \ref{thm:FSM}]
\label{intro:thm:FSM}
Let $\tau_i$ be one of the well-chosen tame inertial types described above (which will be sufficiently generic in our setup). %

Then $R_\rhobar^{(n-1, n-2, \ldots, 0),\tau}$ is either zero or is a normal domain.   Furthermore, if it is nonzero, then either it is formally smooth over $\cO$ or the special fiber is reduced with exactly two irreducible components. 
\end{thm}

We actually prove a more general result for a larger class of deformation rings (Theorem \ref{thm:FSM}). We approach the deformation spaces using the methods for studying Breuil--Kisin modules developed in \cite{LLLM, LLL, MLM}.   This is the first time these methods have been adapted to the ramified setting.    
The key ingredient in our proof of Theorem \ref{thm:FSM} is the fact that the local models (in the sense of \cite{MLM} adapted to the ramified setting) of these Galois deformation spaces have a Levi reduction property: namely, they are formally smooth over similar local models attached to suitable Levi subgroups of $\GL_n$. This turns out to be a general phenomenon whenever the shape of $\rhobar$ relative to the type $\tau$ is suitably ``decomposable'', which may be of independent interest. In the specific case of Theorem \ref{intro:thm:FSM}, the Levi subgroup we reduce to is $\GL_2\times \GL_1^{n-2}$.   Thus, we are able to show essentially that $R_\rhobar^{(n-1, n-2, \ldots, 0),\tau}$ is smooth over the completed local ring of a ramified local model of Pappas--Rapoport from which we deduce the normality and the description of the special fiber. 
We prove a similar Levi reduction property for the Pappas--Zhu local models, which is a key geometric input (Lemma \ref{lemma:2comps}) into the analysis of Serre weight combinatorics for these tame types. 

\begin{rmk} When $K/\Qp$ is unramified, the relevant local model is a product of the Iwahori local models for $\GL_2$.   Concretely, Theorem \ref{intro:thm:FSM} says that $R_\rhobar^{(n-1, n-2, \ldots, 0),\tau}$ will either be power series ring over $\cO$ or will be formally smooth over $\cO[\![x,y]\!]/(xy-p)$.   This observation in the case of $\GL_3$ in \cite{LLLM} was the starting point for this work. 
\end{rmk}

\subsection{Overview}
In \S \ref{sec:preliminaries}, after preliminaries on the affine Weyl group and admissible sets (\S \ref{sec:comb:weyl}), and recollections on Serre weights (\S \ref{sec:SWC}), we formulate a Serre type conjecture on the weights of tame Galois representations over a possibly ramified field (cf.~Definition \ref{defn:SWC:ram}) and obtain our main results on the combinatorics of Serre weights and tame inertial types for the \emph{shapes} we will be interested in (cf.~Propositions \ref{prop:indint}, \ref{prop:indcomb}).

\S \ref{sec:BKM} introduces the notion of \emph{extremal weights} for Galois representations (\S \ref{subsec:extr:wt}). This requires preliminaries on the semicontinuity of shapes for Kisin modules (\S \ref{subsec:SCI}, \ref{subsec:SCII}, in different degrees of generality), the notion of \emph{specializations} for Galois representations (\S \ref{subsec:SPEC}) and the closely related notion of \emph{specialization pairs} (\S \ref{subsec:SPEC:pairs}).
The non-emptiness of the set of the extremal weights is proved in \S \ref{sec:mord} and \S \ref{sub:extr:loc} with different methods. In particular the geometric interpretation of this set in terms of the Emerton--Gee stack in the unramified case is in \S \ref{sec:EGs}, \ref{sub:extr:loc}.

\S \ref{sec:PCDR} calculates the tamely potentially crystalline deformation rings which appear when studying extremal weights of Galois representations.
We first establish structural results on of Breuil--Kisin modules of certain \emph{parabolic shapes} (\S \ref{sub:par:strct}) and then analyze the monodromy condition on them (\S \ref{sub:analysis:MC}, Lemma \ref{lem:parabolic_monodromy}).

In \S \ref{sec:main:mod}, after a number of preliminaries on patching functors and cycles on potentially crystalline deformation rings (\S \ref{sec:patchfunc}, \ref{sec:cycles}), we prove in \S \ref{subsec:WE:MOD} the modularity of extremal weights in an axiomatic setup (Theorem \ref{thm:obv}). 
\S \ref{sec:TW} contains our global applications to automorphic forms on definite unitary groups.

\subsection{Acknowledgements}
The origin of this work dates back to a stay at the Mathematisches Forschungsinstitut Oberwolfach in winter 2019, which provided excellent working condition, and was further carried out during several visits at the University of Arizona and Northwestern University.
We would like to heartily thank these institutions for the outstanding research conditions they provided, and for their support.

We thank Andrea Dotto for explaining the application of the work \cite{LLL} of the first three authors to the weight part of Serre's conjecture for division algebras and for pointing out Remark \ref{rmk:BM}.

D.L. was supported by the National Science Foundation under agreements Nos.~DMS-1128155 and DMS-1703182, an AMS-Simons travel grant, and a start-up grant from Purdue University. B.LH. acknowledges support from the National Science Foundation under grant Nos.~DMS-1128155, DMS-1802037 and the Alfred P. Sloan Foundation. B.L. was supported by National Science Foundation grant DMS-1952556 and the Alfred P. Sloan Foundation. S.M. was supported by the Institut Universitaire de France and the ANR-18-CE40-0026 (CLap CLap).

\subsection{Notation}

For a field $K$, we denote by $\ovl{K}$ a fixed separable closure of $K$ and let $G_K \defeq \Gal(\ovl{K}/K)$.
If $K$ is defined as a subfield of an algebraically closed field, then we set $\ovl{K}$ to be this field.

If $K$ is a nonarchimedean local field, we let $I_K \subset G_K$ denote the inertial subgroup and $W_K \subset G_K$ denote the Weil group.
We fix a prime $p\in\Z_{>0}$.
Let $E \subset \ovl{\Q}_p$ be a subfield which is finite-dimensional over $\Q_p$.
We write $\cO$ to denote its ring of integers, fix an uniformizer $\varpi\in \cO$ and let $\F$ denote the residue field of $E$.
We will assume throughout that $E$ is sufficiently large.

\subsubsection{Reductive groups}
\label{sec:not:RG}
Let $G$ denote a split connected reductive group (over some ring) together with a Borel $B$, a maximal split torus $T \subset B$, and $Z \subset T$ the center of $G$.  
Let $d = \dim G - \dim B$.  
When $G$ is a product of copies of $\GL_n$, we will take $B$ to be upper triangular Borel and $T$ the diagonal torus. 
Let $\Phi^{+} \subset \Phi$ (resp. $\Phi^{\vee, +} \subset \Phi^{\vee}$) denote the subset of positive roots (resp.~positive coroots) in the set of roots (resp.~coroots) for $(G, B, T)$. 
We use the notation $\alpha > 0$ (resp.~$\alpha < 0$) for a positive (resp.~negative) root $\alpha\in \Phi$. 
Let $\Delta$ (resp.~$\Delta^{\vee}$) be the set of simple roots (resp.~coroots).
Let $X^*(T)$ be the group of characters of $T$, and set $X^0(T)$ to be the subgroup consisting of characters $\lambda\in X^*(T)$ such that $\langle\lambda,\alpha^\vee\rangle=0$ for all $\alpha^\vee\in \Delta^{\vee}$.
Let $\Lambda_R \subset X^*(T)$ denote the root lattice for $G$.
Let  $W(G)$ denote the Weyl group of $(G,T)$.  Let $w_0$ denote the longest element of $W(G)$.
We sometimes write $W$ for $W(G)$ when there is no chance for confusion.
Let $W_a$ (resp.~$\tld{W}$) denote the affine Weyl group and extended affine Weyl group 
\[
W_a = \Lambda_R \rtimes W(G), \quad \tld{W} = X^*(T) \rtimes W(G)
\]
for $G$.
We use $t_{\nu} \in \tld{W}$ to denote the image of $\nu \in X^*(T)$. 

The Weyl groups $W(G)$, $\tld{W}$, and $W_a$ act naturally on $X^*(T)$.
If $A$ is any ring, then the above Weyl groups act naturally on $X^*(T)\otimes_{\Z} A$ by extension of scalars.

Let $M$ be a free $\Z$-module of finite rank (e.g. $M=X^*(T)$). 
The duality pairing between $M$ and its $\Z$-linear dual $M^*$ will be denoted by $\langle \ ,\,\rangle$.
If $A$ is any ring, the pairing $\langle \ ,\,\rangle$ extends by $A$-linearity to a pairing between $M\otimes_{\Z}A$ and $M^*\otimes_{\Z}A$, and by an abuse of notation it will be denoted with the same symbol $\langle \ ,\,\rangle$. 

We write $G^\vee = G^\vee_{/\Z}$ for the split connected reductive group over $\Z$  defined by the root datum $(X_*(T),X^*(T), \Phi^\vee,\Phi)$. 
This defines a maximal split torus $T^\vee\subseteq G^\vee$ such that we have canonical identifications $X^*(T^\vee)\cong X_*(T)$ and $X_*(T^\vee)\cong X^*(T)$.

Let $V\defeq X^*(T)\otimes_{\Z}\R\ \setminus\ \big(\bigcup_{(\alpha,n)}H_{\alpha,n}\big)$.
For $(\alpha,k)\in \Phi \times \Z$, we have the root hyperplane $H_{\alpha,k}\defeq \{x \in V \mid \langle\lambda,\alpha^\vee\rangle=k\}$ and the half-hyperplanes $H^{+}_{\alpha, k} = \{ x \in V \mid \langle x, \alpha^\vee\rangle > k \}$  and $H^{-}_{\alpha, n} =  \{ x \in V \mid \langle x, \alpha^\vee\rangle < k \}.$ 
An alcove %
is a connected component of $V \setminus\ \big(\bigcup_{(\alpha,n)}H_{\alpha,n}\big)$. 

We say that an alcove $A$ is \emph{restricted} if $0<\langle\lambda,\alpha^\vee\rangle<1$ for all $\alpha\in \Delta$ and $\lambda\in A$.
We let $A_0$ denote the (dominant) base alcove, i.e.~the set of $\lambda\in X^*(T)\otimes_{\Z}\R$ such that $0<\langle\lambda,\alpha^\vee\rangle<1$ for all $\alpha\in \Phi^+$. 
Let $\cA$ denote the set of alcoves. 
Recall that $\tld{W}$ acts transitively on the set of alcoves, and $\tld{W}\cong\tld{W}_a\rtimes \Omega$ where $\Omega$ is the stabilizer of $A_0$.
We define
\[\tld{W}^+\defeq\{\tld{w}\in \tld{W}:\tld{w}(A_0) \textrm{ is dominant}\}.\]
and
\[\tld{W}^+_1\defeq\{\tld{w}\in \tld{W}^+:\tld{w}(A_0) \textrm{ is restricted}\}.\]
We fix an element $\eta_0\in X^*(T)$ such that $\langle \eta_0,\alpha^\vee\rangle = 1$ for all positive simple roots $\alpha$ and let $\tld{w}_h$ be $w_0 t_{-\eta_0}\in \tld{W}^+_1$. 

When $G = \GL_n$, we fix an isomorphism $X^*(T) \cong \Z^n$ in the standard way, where the standard $i$-th basis element $(0,\ldots, 1,\ldots, 0)$ (with the $1$ in the $i$-th position) of the right-hand side corresponds to extracting the $i$-th diagonal entry of a diagonal matrix.
When $G$ is a product of copies of $\GL_n$ indexed over a set $\cJ$ we take $\eta_0 \in X^*(T)$ to correspond to the element $(n-1, n-2, \ldots, 0)_{j\in\cJ} \in (\Z^n)^{\cJ}$ in the identification above.
In this case, given $j\in\cJ$ we write $\eta_{0,j}\in $ to denote the element which corresponds to the tuple $(n-1,\dots,1,0)$ at $j$ and to the zero tuple elsewhere.

Let $F^+_p$ be a finite \'etale $\Q_p$-algebra.
Then $F^+_p$ is isomorphic to a product $\prod_{S_p} F^+_{v}$ for some finite set $S_p$ where for each $v\in S_p$, $F_{v}^+$ is finite extension of $\Q_p$.
For each $v\in S_p$, let $\cO_{F^+_{v}} \subset F^+_{v}$ be the ring of integers, $k_{v}$ the residue field, $F^+_{v,0} \subset F^+_{v}$ the maximal unramified subextension, $f_{v}$ the unramified degree $[F^+_{v,0}:\Q_p]$, and $e_{v}$ the ramification degree $[F^+_{v}:F^+_{v,0}]$.
Let $\cO_p$ be the product $\prod_{v\in S_p} \cO_{F^+_{v}}$ and $k_p$ the product $\prod_{v\in S_p} k_{v}$. 

In global applications, $S_p$ will be a finite set of places dividing $p$ of a number field $F^+$.
When working locally, $S_p$ will have cardinality one, in which case we drop the subscripts from $f_{v}$, $e_{v}$, and $k_{v}$ and denote the single extension $F^+_{v}$ of $\Q_p$ by $K$.

If $G$ is a split connected reductive group over $\F_p$, with Borel $B$, maximal split torus $T$, and center $Z$, we let $G_0 \defeq \Res_{k_p/\F_p} G_{/k_p}$ with Borel subgroup $B_0 \defeq  \Res_{k_p/\F_p} B_{/k_p}$, maximal torus $T_0 \defeq \Res_{k_p/\F_p} T_{/k_p}$, and $Z_0 = \Res_{k_p/\F_p} Z_{/k_p}$. 
Assume that $\F$ contains the image of any ring homomorphism $k_p \ra \ovl{\F}_p$ and let $\cJ$ be the set of ring homomorphisms $k_p \ra \F$.
Then $\un{G} \defeq (G_0)_{/\F}$ is naturally identified with the split reductive group $G_{/\F}^{\cJ}$. 
We similarly define $\un{B}, \un{T},$ and $\un{Z}$.  
Corresponding to $(\un{G}, \un{B}, \un{T})$, we have the set of positive roots $\un{\Phi}^+ \subset \un{\Phi}$ and the set of positive coroots $\un{\Phi}^{\vee, +}\subset \un{\Phi}^{\vee}$.
The notations $\un{\Lambda}_R$, $\un{W}$, $\un{W}_a$, $\tld{\un{W}}$, $\tld{\un{W}}^+$, $\tld{\un{W}}^+_1$, $\un{\Omega}$ should be clear as should the natural isomorphisms $X^*(\un{T}) = X^*(T)^{\cJ}$ and the like. 
The absolute Frobenius automorphism $\varphi$ on $k_p$ induces an automorphism $\pi$ of the identified groups $X^*(\un{T})$ and $X_*(\un{T}^\vee)$ by the formula $\pi(\lambda)_\sigma = \lambda_{\sigma \circ \varphi^{-1}}$ for all $\lambda\in X^*(\un{T})$ and $\sigma: k_p \ra \F$.
We assume that, in this case, the element $\eta_0\in X^*(\un{T})$ we fixed is $\pi$-invariant.
We similarly define an automorphism $\pi$ of $\un{W}$ and $\tld{\un{W}}$.

\subsubsection{Galois Theory}
\label{sec:not:GT}
We now assume that $S_p$ has cardinality one.
We write $K\defeq F^+_{v}$ and drop the subscripts from $f_{v}$, $e_{v}$, and $k_{v}$.
Let $W(k)$ be ring of Witt vectors which is also ring of integers $\cO_{K_0}$ of $K_0$.
We denote the arithmetic Frobenius automorphism on $W(k)$ by $\phz$, which acts as raising to $p$-th power on the residue field.  
We fix an embedding $\sigma_0$ of $K_0$ into $E$ (equivalently an embedding $k$ into $\F$) and define $\sigma_j = \sigma_0 \circ \phz^{-j}$, which gives an identification between $\cJ=\Hom(k,\F)$ and $\Z/f\Z$.

We normalize Artin's reciprocity map $\mathrm{Art}_{K}: K\s\ra W_{K}^{\mathrm{ab}}$ in such a way that uniformizers are sent to geometric Frobenius elements.

Given an uniformizer $\pi_K\in \cO_K$ and a sequence $\un{\pi}_K\defeq (\pi_{K,m})_{m\in \N}\in \ovl{K}^{\N}$ satisfying $\pi_{K,m+1}^{p}=\pi_{K,m}$, $\pi_{K,0}\defeq \pi_K$ we let $K_\infty$ be $\underset{m\in\N}{\bigcup}K(\pi_{K,m})$.

Given an element $\pi_1 \defeq (-\pi_K)^{\frac{1}{p^{f}-1}}\in \overline{K}$ we have a corresponding character $\omega_{K}:I_K \ra W(k)^{\times}$ which, using our choice of embedding $\sigma_0$ gives a fundamental character of niveau $f$ 
\[
\omega_{f}:= \sigma_0 \circ \omega_{\pi_1}:I_K \ra \cO^{\times}. 
\]

Let $\rho: G_K\rightarrow \GL_n(E)$ be a $p$-adic, de Rham Galois representation. 
For $\sigma: K\iarrow E$, we define $\mathrm{HT}_\sigma(\rho)$ to be the multiset of $\sigma$-labeled Hodge-Tate weights of $\rho$, i.e.~the set of integers $i$ such that $\dim_E\big(\rho\otimes_{\sigma,K}\C_p(-i)\big)^{G_K}\neq 0$ (with the usual notation for Tate twists).
In particular, the cyclotomic character $\eps$ has Hodge--Tate weights 1 for all embedding $\sigma:K\iarrow E$. 
For $\mu=(\mu_j)_j\in X^*(\un{T})$ we say that $\rho$ has Hodge--Tate weighs $\mu$ if
\[
\mathrm{HT}_{\sigma_j}(\rho)=\{\mu_{1,j},\mu_{2,j},\dots,\mu_{n,j}\}.
\]
The \emph{inertial type} of $\rho$ is the isomorphism class of $\mathrm{WD}(\rho)|_{I_K}$, where $\mathrm{WD}(\rho)$ is the Weil--Deligne representation attached to $\rho$ as in \cite{CDT}, Appendix B.1 (in particular, $\rho\mapsto\mathrm{WD}(\rho)$ is \emph{covariant}).
An inertial type is a morphism $\tau: I_K\ra \GL_n(E)$ with open kernel and which extends to the Weil group $W_K$ of $G_K$.
We say that $\rho$ has type $(\mu,\tau)$ if $\rho$ has Hodge--Tate weights $\mu$ and inertial type given by (the isomorphism class of) $\tau$.

\subsubsection{Miscellaneous}
\label{sec:not:mis}

For any ring $S$, we define $\mathrm{Mat}_n(S)$ to be the set of $n\times n$ matrix with entries in $S$. 
If $M\in \mathrm{Mat}_n(S)$ and $A\in \GL_n(S)$ we write
\begin{equation}
\label{def:adj}
\Ad(A)(M)\defeq A\,M\,A^{-1}.
\end{equation}

If $X$ is an ind-scheme defined over $\cO$, we write $X_E\defeq X\times_{\Spec\cO} \Spec E$ and $X_{\F}\defeq X\times_{\Spec \cO}\Spec \F$ to denote its generic and special fiber, respectively.
\clearpage{}%
\clearpage{}%
\section{Preliminaries}
\label{sec:preliminaries}

\subsection{Extended affine Weyl groups}

In this section, we collect some background material on Weyl groups which will be needed throughout the paper.

Recall from \S \ref{sec:not:RG} that $G$ is a split reductive group with split maximal torus $T$ and Borel $B$. 
Let $W \defeq W(G,T)$ be the Weyl group and $V\defeq X^*(T) \otimes \R \cong X_*(T^{\vee}) \otimes \R$ denote the apartment of $(G, T)$ on which $\tld{W}\defeq X^*(T) \rtimes W$ acts.   Let $\cC_0$ denote the dominant Weyl chamber in $V$.  For any $w \in W(G)$, let $\cC_w = w(\cC_0)$.  In particular, denoting the longest element of $W$ by $w_0$, $\cC_{w_0}$ is the anti-dominant Weyl chamber.

Recall from \S \ref{sec:not:RG} that $\cA$ denotes the set of alcoves of $X^*(T) \otimes \R$ and that $A_0 \in \cA$ denotes the dominant base alcove. 
We let $\uparrow$ denote the upper arrow ordering on alcoves as defined in \cite[\S II.6.5]{RAGS}. 
Since $W_a$ acts simply transitively on the set of alcoves, $\tld{w} \mapsto \tld{w}(A_0)$ induces a bijection $W_a \risom \cA$ and thus an upper arrow ordering $\uparrow$ on $W_a$.
The dominant base alcove $A_0$ also defines a set of simple reflections in $W_a$ and thus a Coxeter length function on $W_a$ denoted $\ell(-)$ and a Bruhat order on $W_a$ denoted by $\leq$. 

If $\Omega \subset \tld{W}$ is the stabilizer of the base alcove, then $\tld{W} = W_a \rtimes \Omega$ and so $\tld{W}$ inherits a Bruhat and upper arrow order in the standard way: For $\tld{w}_1, \tld{w}_2\in W_a$ and $\delta\in \Omega$, $\tld{w}_1\delta\leq \tld{w}_2\delta$ (resp.~$\tld{w}_1\delta\uparrow \tld{w}_2\delta$) if and only if $\tld{w}_1\leq \tld{w}_2$ (resp.~$\tld{w}_1\uparrow \tld{w}_2$), and elements in different right $W_a$-cosets are incomparable. 
We extend $\ell(-)$ to $\tld{W}$ by letting $\ell(\tld{w}\delta)\defeq \ell(\tld{w})$ for any $\tld{w}\in W_a$, $\delta\in \Omega$.

\begin{defn}
If $\tld{w}_1,\dots, \tld{w}_m\in\tld{W}$, we say that $\tld{w}_1 \tld{w}_2\cdots \tld{w}_m$ is a \emph{reduced expression} if the inequality $\ell(\tld{w}_1 \tld{w}_2\cdots \tld{w}_m)\leq \sum\limits_{i=1}^m\ell(\tld{w}_i)$ is an equality. 
\end{defn}

Let $(\tld{W}^\vee,\leq)$ be the following partially ordered group: $\tld{W}^\vee$ is identified with $\tld{W}$ as a group, and $\ell(-)$ and $\leq$ are defined with respect to the \emph{antidominant} base alcove.  

\begin{defn} 
\label{affineadjoint} We define a bijection $\tld{w}\mapsto \tld{w}^*$ between $\tld{W}$ and $\tld{W}^\vee$ as follows: for $\tld{w} = t_{\nu}w \in \tld{W}$, with $w\in W$ and $\nu\in X^*(T) = X_*(T^{\vee})$, then $\tld{w}^*\defeq  w^{-1}t_{\nu} \in \tld{W}^\vee$. 
\end{defn}
\noindent This bijection respects notions of length and Bruhat order (see \cite[Lemma 2.1.3]{LLL}). 

We recall some fundamental notions associated to the geometry of $X^*(T)$ and $\tld{W}$.

\begin{defn} \label{defn:conv} Let $\lambda \in X^*(T)$.  
The convex hull of the set $\{w(\lambda)\mid w\in W\}$ is defined to be
\[
\Conv(\lambda) \defeq  \underset{w\in W}{\bigcap} w(\lambda)+\ovl{\cC}_{w w_0}
\]
where $\ovl{\cC}_{w w_0}$ denotes the closure of the Weyl chamber $\cC_{w w_0}$.
\end{defn}

We recall the definition of the admissible set from \cite{KR}:
\begin{defn} \label{defn:adm} For $\lambda \in X^*(T)$, define 
\[
\Adm(\lambda) \defeq  \left\{ \tld{w} \in \tld{W} \mid \tld{w} \leq t_{w(\lambda)} \text{ for some } w \in W \right\}.
\]
\end{defn}

For a positive integer $e$, define the $e$-critical strips to be strips $H^{(1-e, e)}_{\alpha} = \{ x \in V \mid 1-e < \langle x, \alpha^\vee\rangle < e \}$ where $\alpha \in \Phi^+$.

\begin{defn} \label{defn:regular}%
An alcove $A \in \cA$ is \emph{$e$-regular} if $A$ does not lie in any $e$-critical strip.   For any $\tld{w} \in \tld{W}$, we say $\tld{w}$ is \emph{$e$-regular} if $\tld{w}(A_0)$ is $e$-regular.  Define 
\[
\Adm^{\text{$e$-reg}}(\lambda) = \{ \tld{w} \in \Adm(\lambda)\mid \tld{w} \text{ is $e$-regular} \}. 
\]
\end{defn}

\begin{prop}
\label{prop:can:reg} 
If $\tld{w} \in \tld{W}$ is $e$-regular, then there exist $\tld{w}_1$ and $\tld{w}_2 \in \tld{W}_1^+$ and a dominant weight $\nu \in X^*(T)$ such that $\tld{w} = \tld{w}_2^{-1} w_0 t_{\nu + (e-1) \eta_0} \tld{w}_1$.
Moreover, $\tld{w}_1$, $\tld{w}_2$, and $\nu$ as above are unique up to $X^0(T)$. 
Conversely, if $\tld{w}_1$ and $\tld{w}_2$ are elements of $\tld{W}^+$, then $\tld{w}_2^{-1} w_0 t_{(e-1) \eta_0} \tld{w}_1$ is $e$-regular.
\end{prop}

We conclude this section by recalling from \cite[Definition 2.1.10]{MLM} the various notions of genericity for elements of $X^*(T)$.

\begin{defn}
\label{defn:var:gen}
Let $\lambda\in X^*(T)$ be a weight and let $m\geq 0$ be an integer. 
\begin{enumerate}
\item
\label{defn:deep}
We say that $\lambda$ \emph{lies $m$-deep in its $p$-alcove} if for all $\alpha\in \Phi^{+}$, there exist integers $m_\alpha\in \Z$ such that $pm_\alpha+m<\langle \lambda+\eta_0,\alpha^\vee\rangle<p(m_\alpha+1)-m$. 

\item
\label{defn:mugeneric} We say that $\lambda \in X^*(T)$ is \emph{$m$-generic} if 
$m < |\langle \lambda, \alpha^{\vee} \rangle + pk|$
for all $\alpha \in \Phi$ and $k\in \Z$ (or equivalently, $\lambda-\eta_0$ is $m$-deep in its $p$-alcove).
\item
\label{defn:small}
We say that an element $\tld{w}=w t_\nu$ (in either $\tld{W}$ or $\tld{W}^\vee$) is \emph{$m$-small} if $\langle \nu,\alpha^\vee\rangle\leq m$ for all $\alpha\in \Phi$. 
\end{enumerate} 
\end{defn}

\subsection{Combinatorics of the extended affine Weyl group}
\label{sec:comb:weyl}

In this section, we collect a variety of results on the combinatorics of the extended affine Weyl group. 
These will be applied to the analysis of the combinatorics of Serre weights in \S \ref{sec:combserre}. 
The methods are elementary with the exception of a geometric input from Pappas--Zhu local models in the proof of Lemma \ref{lemma:2comps}. 
We begin with results concerning the partial orderings $\leq$ and $\uparrow$. 

\begin{lemma}\label{lemma:minrep}
Suppose that $\tld{x}^+ \in \tld{W}^+$ and $w\in W$. Then $w\tld{x}^+$ is a reduced expression.
\end{lemma}
\begin{proof}
There are galleries in the $1$-direction from $w^{-1}(A_0)$ to $A_0$ and from $A_0$ to $\tld{x}^+(A_0)$.
We conclude that $\ell(w\tld{x}^+) = \ell(w)+\ell(\tld{x}^+)$.
\end{proof}

\begin{lemma}\label{lemma:bruhatup}
Suppose that $\tld{x} \in \tld{W}$ and $\tld{w}^+ \in \tld{W}^+$ and $\tld{x} \leq w_0\tld{w}^+$.
Then $w_0\tld{w}^+ \uparrow w\tld{x}$ for any $w\in W$.
\end{lemma}
\begin{proof}
Since $w_0\tld{w}^+$ is a reduced factorization by Lemma \ref{lemma:minrep}, $\tld{x} \leq w_0\tld{w}^+$ implies that $\tld{x} = s \tld{x}'$ for $s\in W$ and $\tld{x}' \in \tld{W}$ with $\tld{x}' \leq \tld{w}^+$.
Factoring $\tld{x}'$ as the reduced expression $s'\tld{x}^+$ where $s'\in W$ and $\tld{x}^+\in \tld{W}^+$, we have that $\tld{x}^+ \leq \tld{x}'$. 
Replacing $s$ by $ss'$ and $\tld{x}'$ by $\tld{x}^+$, we can thus assume without loss of generality that $\tld{x}' = \tld{x}^+$ is in $\tld{W}^+$.
Wang's theorem (\cite[Theorem 4.3]{Wang} or \cite[Theorem 4.1.1]{LLL}) implies that $\tld{x}^+\uparrow \tld{w}^+$.
Then we have that $w_0\tld{w}^+ \uparrow w_0 \tld{x}^+ \uparrow ws\tld{x}^+ = w\tld{x}$ for any $w\in W$ by \cite[II 6.5(5)]{RAGS}.
\end{proof}

\begin{lemma}\label{lemma:bruhatup1}
If $\tld{x}$ and $\tld{y} \in \tld{W}$ and $\tld{x}\leq \tld{y}$, then $\tld{x}^+ \uparrow \tld{y}^+$ where $\tld{x}^+$ and $\tld{y}^+$ are the unique elements in $W\tld{x} \cap \tld{W}^+$ and $W\tld{y} \cap \tld{W}^+$, respectively.
In particular, we have $\tld{x} \uparrow \tld{y}^+$.
\end{lemma}
\begin{proof}
Let $\tld{y}^+$ be $w\tld{y}$ with $w\in W$.
Since $w_0=(w_0w)w^{-1}$ and $w_0(w\tld{y})$ are reduced expressions (the latter by Lemma \ref{lemma:minrep}, the former by e.g.~\cite[\S 1.8]{humphreys-coxeter}), so is $(w_0w)w^{-1}(w\tld{y})$ and therefore so is $(w_0 w)\tld{y}$.
Since $\tld{x}^+ \leq \tld{x}$ (by Lemma \ref{lemma:minrep}) and $\tld{x} \leq \tld{y}$, 
$w_0w\tld{x}^+ \leq w_0w\tld{y} = w_0 \tld{y}^+$.
Lemma \ref{lemma:bruhatup} implies that $w_0\tld{y}^+ \uparrow w_0\tld{x}^+$ so that $\tld{x}^+ \uparrow \tld{y}^+$.
The last claim follows from \cite[II 6.5(5)]{RAGS}.
\end{proof}

\begin{lemma}\label{lemma:doubleclosure}
If $\tld{w},\tld{w}' \in \tld{W}^+_1$, $\lambda,\nu\in X^*(T)$ with $\lambda$ dominant, then $t_\nu w_0t_\lambda\tld{w} \leq w_0t_\lambda\tld{w}'$ and $t_{-\nu}w_0t_\lambda\tld{w}' \leq w_0t_\lambda\tld{w}$ imply that $\nu \in X^0(T)$ and $\tld{w}' = t_\nu \tld{w}$.
\end{lemma}
\begin{proof}
Suppose that $t_\nu w_0t_\lambda\tld{w} \leq w_0t_\lambda\tld{w}'$ and $t_{-\nu}w_0t_\lambda\tld{w}' \leq w_0t_\lambda\tld{w}$.
Lemma \ref{lemma:bruhatup} implies that $w_0t_\lambda\tld{w}' \uparrow t_\nu w_0t_\lambda\tld{w}$ and $w_0t_\lambda\tld{w} \uparrow t_{-\nu} w_0t_\lambda\tld{w}'$.
Combining these, we have that $w_0t_\lambda\tld{w}' \uparrow t_\nu w_0t_\lambda\tld{w} \uparrow w_0 t_\lambda\tld{w}'$ which implies that $w_0t_\lambda\tld{w}' = t_\nu w_0 t_\lambda\tld{w}$ or equivalently that $\tld{w}' = t_{w_0\nu}\tld{w}$.
This implies that $\tld{w}$ and $\tld{w}'$ have the same image in $W$.
Using that $\tld{w}$ and $\tld{w}'$ are both in $\tld{W}_1^+$, we find that $t_{w_0\nu} = \tld{w}'\tld{w}^{-1} \in X^0(T)$ and in particular $w_0\nu = \nu$.
\end{proof}

We now begin our analysis of certain elements of the admissible set which play an important role in our modularity results. 
For a simple root $\alpha$, let $W_{a,\alpha}$ be the subgroup of $W_a$ generated by $s_\alpha$ and $t_\alpha$.

\begin{lemma}\label{lemma:simpleup}
Let $\alpha$ be a simple root.
Suppose that $\tld{w}_\alpha \tld{w}_1 \uparrow \tld{w}_2 \uparrow \tld{w}_1$ for some $\tld{w}_\alpha \in W_{a,\alpha}$.
Then $\tld{w}_2 \in W_{a,\alpha} \tld{w}_1$.
\end{lemma}
\begin{figure}[h]
    \centering
     \includegraphics[scale=.35]{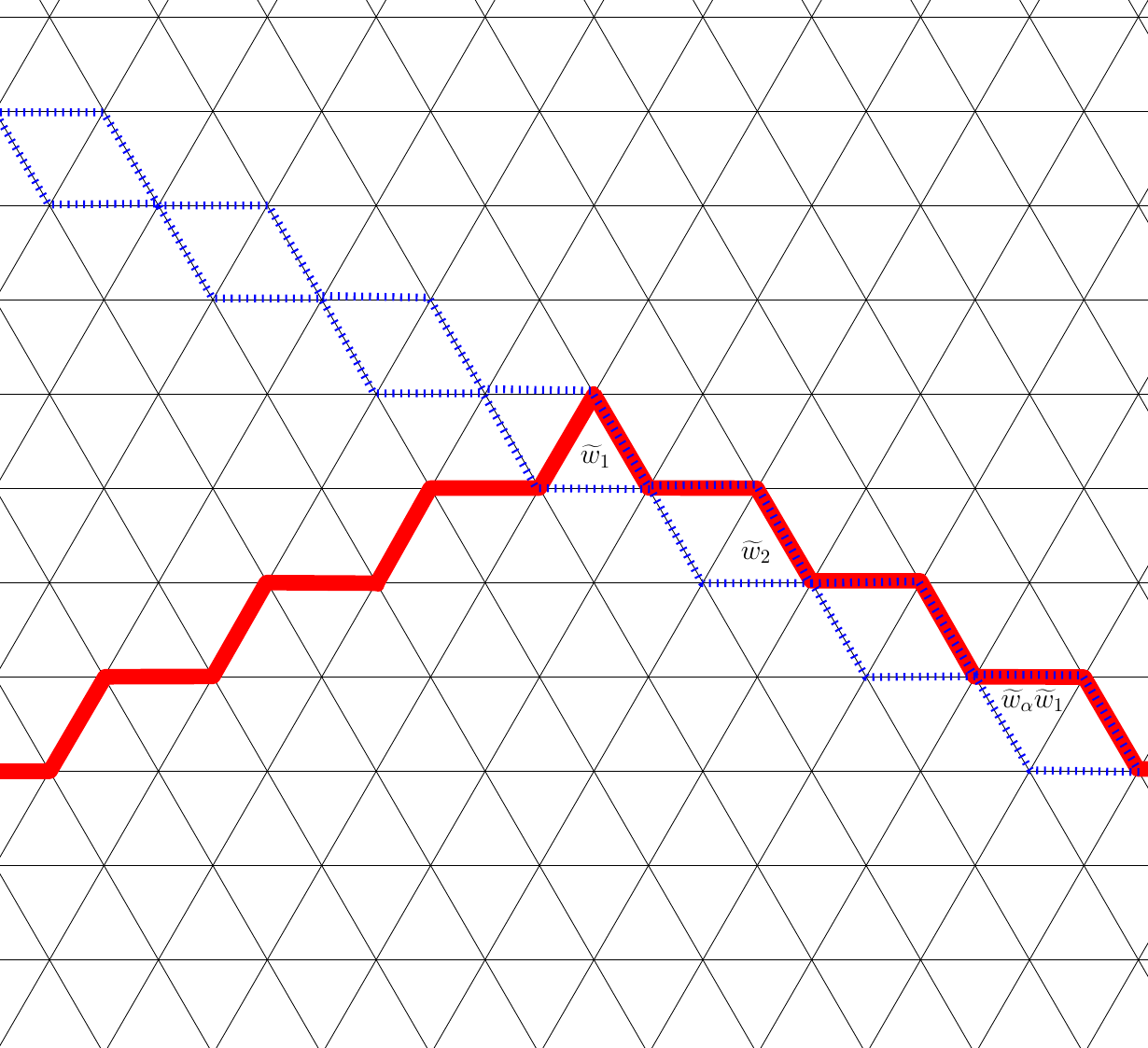} 
        \caption{\small{The alcoves which are below $\tld{w}_1(A_0)$ in the $\uparrow$ order are those below the thickened red lines.
       The alcoves corresponding to $W_{a,\alpha}\tld{w}_1$ are drawn in dotted blue lines.}}
 \label{pic:sandwich}
\end{figure} 
\begin{proof}
Let $x\in A_0$.
Then $\tld{w}_1(x) - \tld{w}_2(x)$ and $\tld{w}_2(x) - \tld{w}_\alpha \tld{w}_1(x)$ are nonnegative linear combinations of positive simple roots.
On the other hand, $\tld{w}_1(x) - \tld{w}_\alpha \tld{w}_1(x)$ is a nonnegative multiple of $\alpha$.
This implies that so is $\tld{w}_1(x) - \tld{w}_2(x)$.

There is a series of hyperplane reflections $(s_i)_{i=1}^m$ such that 
\[
\tld{w}_2 \uparrow s_1 \tld{w}_2 \uparrow s_2 s_1 \tld{w}_2 \uparrow \cdots \uparrow s_m \cdots s_2 s_1 \tld{w}_2 = \tld{w}_1.
\]
If the corresponding positive roots are $(\alpha_i)_{i=1}^m$, then $\tld{w}_1(x) - \tld{w}_2(x)$ is a positive linear combination of the roots in $\{\alpha_i\}_{i=1}^m$.
The above paragraph implies that $\alpha_i = \alpha$ for all $i$.
\end{proof}

Let $e$ be a positive integer.
Recall that the $e\eta_0$-admissible set $\Adm(e\eta_0) \subset \tld{W}$ is the subset of elements $\tld{w}$ such that $\tld{w} \leq t_{w(e\eta_0)}$ for some $w\in W$.

\begin{prop}\label{prop:listshapes}
The set $w^{-1} W_{a,\alpha} t_{e\eta_0} w \cap \Adm(e\eta_0)$ consists of elements
\[
t_{w^{-1}(e\eta_0-k\alpha)} \qquad \textrm{for }0\leq k \leq e
\]
and
\[
\tld{w}^{-1}s_\alpha t_{e\eta_0-(k+1)\alpha}\tld{w} \qquad \textrm{for }0\leq k \leq e-1,
\]
where $\tld{w}\in \tld{W}^+_1$ is an element (unique up to $X^0(T)$) with image $w$ in $W$.
\end{prop}
\begin{proof}
It is easy to check that the listed elements lie in $w^{-1} W_{a,\alpha} t_{e\eta_0} w$.
Furthermore, they are all less than or equal to either $t_{w^{-1}(e\eta_0)}$ or $t_{(s_\alpha w)^{-1}(e\eta_0)}$.
Indeed, set $\tld{z}_k\defeq t_{(e-1)\eta_0 - k\alpha}$ and $\tld{z}'_k\defeq t_{(e-1)\eta_0 - (e-1-k)\alpha}s_\alpha$.
Then
\begin{align}
\label{eq:el1}
t_{w^{-1}(e\eta_0-k\alpha)} &= (\tld{w}_h\tld{w})^{-1}w_0 (\tld{z}_k \tld{w}) \quad\ \ = (\tld{w}_h\tld{w})^{-1}(w_0 s_\alpha) (\tld{z}'_k \tld{w})&\\
\label{eq:el2}
\tld{w}^{-1}s_\alpha t_{e\eta_0-(k+1)\alpha}\tld{w} &= (\tld{w}_h\tld{w})^{-1}(w_0s_\alpha) (\tld{z}_k \tld{w})= (\tld{w}_h\tld{w})^{-1}w_0 (\tld{z}'_k \tld{w})&
\end{align}
where $0\leq k\leq e$ for the elements in (\ref{eq:el1}) and $0\leq k\leq e-1$ for the elements in (\ref{eq:el2}).
Both $\tld{z}_k,\, \tld{z}'_k \uparrow t_{(e-1)\eta_0}$ and, if $k\neq e$, one is them is in $\tld{W}^+$.
Wang's theorem implies that, for $0\leq k\leq e-1$, one among $\tld{z}_k$, $\tld{z}'_k$ is less than or equal to $t_{(e-1)\eta_0}$.
This implies that for $0\leq k\leq e-1$ the elements (\ref{eq:el1}), (\ref{eq:el2}), with the exception of $t_{(s_\alpha w)^{-1}(e\eta_0)}$, are less than or equal to $t_{w^{-1}(e\eta_0)}$ in the Bruhat ordering.
The exceptional element is less than or equal to itself.

We claim that any element in $w^{-1} W_{a,\alpha} t_{e\eta_0} w$ of length at most that of $t_{w^{-1}(e\eta_0)}$ is one of the listed elements.
This would provide the reverse inclusion.
For each positive root $\beta$ and $\tld{w}\in \tld{W}$, let 
\[
n_\beta(\tld{w})=
\begin{cases}
\lfloor \langle \tld{w}(x),w^{-1}(\beta^\vee) \rangle \rfloor \qquad & \textrm{ if } w(\beta) > 0 \\ 
\lfloor \langle \tld{w}(x),w^{-1}(\beta^\vee) \rangle \rfloor + 1 \qquad & \textrm{ if } w(\beta) < 0 
\end{cases}
\] 
for any $x\in A_0$. 
Let $m_\beta(\tld{w})$ be $|n_\beta(\tld{w})|$.
Then $\ell(\tld{w})$ is the sum $\sum_{\beta>0} m_\beta(\tld{w})$
(\cite[Proposition 1.23]{iwahori-matsumoto}, see also \cite[\S 1.3]{he-nie}).
Let $d(\tld{w})$ be the sum $m_\alpha(\tld{w}) + \sum_{\beta>0,\, \beta\neq \alpha} n_\beta(\tld{w})$.
The function $d(-)$ has three favorable properties: $d(\tld{w}_1) \leq \ell(\tld{w}_1)$ for all $\tld{w}_1 \in \tld{W}$, $\ell(\tld{w}_1) = d(\tld{w}_1)$ if $w\tld{w}_1 \in \tld{W}^+$ (in particular for $\tld{w}_1 = t_{w^{-1}(e\eta_0)}$), and as we shall see next, $d(\tld{w}_1)-m_\alpha(\tld{w}_1) = \sum_{\beta>0,\, \beta\neq \alpha} n_\beta(\tld{w}_1)$ is the same for all $\tld{w}_1 \in w^{-1} W_{a,\alpha} t_{e\eta_0} w$. 

Fix $x$ as above such that $\langle x,w^{-1}(\alpha^\vee)\rangle = \pm\frac{1}{2}$.
Then for each $\tld{w}_1 \in w^{-1} W_{a,\alpha} t_{e\eta_0} w$, $\tld{w}_1(x) = x+w^{-1}(e\eta_0-\frac{k}{2}\alpha)$ for a some $k \in \Z$.
Moreover, the map $\tld{w}_1\mapsto k$ defines a bijection $w^{-1} W_{a,\alpha} t_{e\eta_0} w \ra \Z$.
We claim that 
\begin{equation} \label{eqn:notalpha}
\sum_{\beta > 0,\, \beta \neq \alpha} \lfloor \langle t_{w^{-1}(e\eta_0)}(x),w^{-1}(\beta^\vee) \rangle \rfloor = \sum_{\beta > 0,\, \beta \neq \alpha}\lfloor \langle \tld{w}_1(x),w^{-1}(\beta^\vee) \rangle \rfloor.
\end{equation}
Assuming (\ref{eqn:notalpha}) for the moment, we obtain 
\begin{align*}
\ell(t_{w^{-1}(e\eta_0)}) - d(\tld{w}_1) &= d(t_{w^{-1}(e\eta_0)}) - d(\tld{w}_1) \\
&= m_\alpha(t_{w^{-1}(e\eta_0)}) - m_\alpha(\tld{w}_1) + \sum_{\beta > 0,\, \beta \neq \alpha} (\lfloor \langle t_{w^{-1}(e\eta_0)}(x),w^{-1}(\beta^\vee) \rangle \rfloor - \lfloor \langle \tld{w}_1(x),w^{-1}(\beta^\vee) \rangle \rfloor) \\
&= m_\alpha(t_{w^{-1}(e\eta_0)}) - m_\alpha(\tld{w}_1)\\
&= |e|-|e-k|.
\end{align*}
If $\ell(\tld{w}_1) \leq \ell(t_{w^{-1}(e\eta_0)})$, then since $d(\tld{w}_1) \leq \ell(\tld{w}_1)$, $|e|\geq |e-k|$ so that $0 \leq k \leq 2e$.
These $2e+1$ values for $k$ correspond to the $2e+1$ listed elements.
(See Figure \ref{pic:strip} for the case of $\GL_3$ and $e=3$.)

It suffices to justify (\ref{eqn:notalpha}).
We need to show that 
\[
\sum_{\beta > 0,\, \beta \neq \alpha}\lfloor \langle x,w^{-1}(\beta^\vee)\rangle \rfloor = \sum_{\beta > 0,\, \beta \neq \alpha} \lfloor \langle x-\frac{k}{2}w^{-1}(\alpha),w^{-1}(\beta^\vee)\rangle \rfloor,
\]
or equivalently, letting $y = w(x)$, that
\[
\sum_{\beta > 0,\, \beta \neq \alpha}\lfloor \langle y,\beta^\vee\rangle \rfloor = \sum_{\beta > 0,\, \beta \neq \alpha} \lfloor \langle y-\frac{k}{2}\alpha,\beta^\vee\rangle \rfloor.
\]
We can ignore roots $\beta$ such that $\langle \alpha,\beta^\vee \rangle = 0$.
The remaining positive roots come in pairs $(\beta_-,\beta_+)$ where $\langle \alpha,\beta_-^\vee \rangle < 0$ and $\beta_+ = s_\alpha(\beta_-)$.
Fix such a pair.
The fact that $\langle \alpha, \beta_-^\vee\rangle + \langle \alpha, \beta_+^\vee\rangle = 0$ implies that 
\[
\langle y,\beta_+^\vee\rangle + \langle y,\beta_-^\vee\rangle =  \langle y-\frac{k}{2}\alpha,\beta_+^\vee\rangle +\langle y-\frac{k}{2}\alpha,\beta_-^\vee\rangle.
\]
It suffices to show that 
\begin{equation}\label{eqn:fracpart}
\{\langle y,\beta_+^\vee\rangle \} + \{\langle y,\beta_-^\vee\rangle\} = \{\langle y-\frac{k}{2}\alpha,\beta_+^\vee\rangle \}+\{\langle y-\frac{k}{2}\alpha,\beta_-^\vee\rangle\}
\end{equation}
where $\{ r \}$ denotes the fractional part $r - \lfloor r \rfloor$ of $r\in \R$. 
If $\langle \alpha,\beta^\vee_+\rangle$, and therefore $\langle \alpha,\beta^\vee_-\rangle$, is even, then (\ref{eqn:fracpart}) is clear. 
Now suppose that $\langle \alpha,\beta^\vee_+\rangle$ is odd.
Recall that we chose $x$ so that $\langle y,\alpha^\vee\rangle = \pm\frac{1}{2}$.
Then $\beta_+ = \beta_-+\langle \beta_+,\alpha^\vee\rangle \alpha$ implies that $\{\langle y, \beta_-^\vee \rangle\} = \{\langle y, \beta_+^\vee\rangle+\frac{1}{2}\}$.
We see that the terms of each side of (\ref{eqn:fracpart}) are the same (resp.~permuted) when $k$ is even (resp.~odd).
\end{proof}

\begin{figure}[h]
    \centering
     \includegraphics[scale=.55]{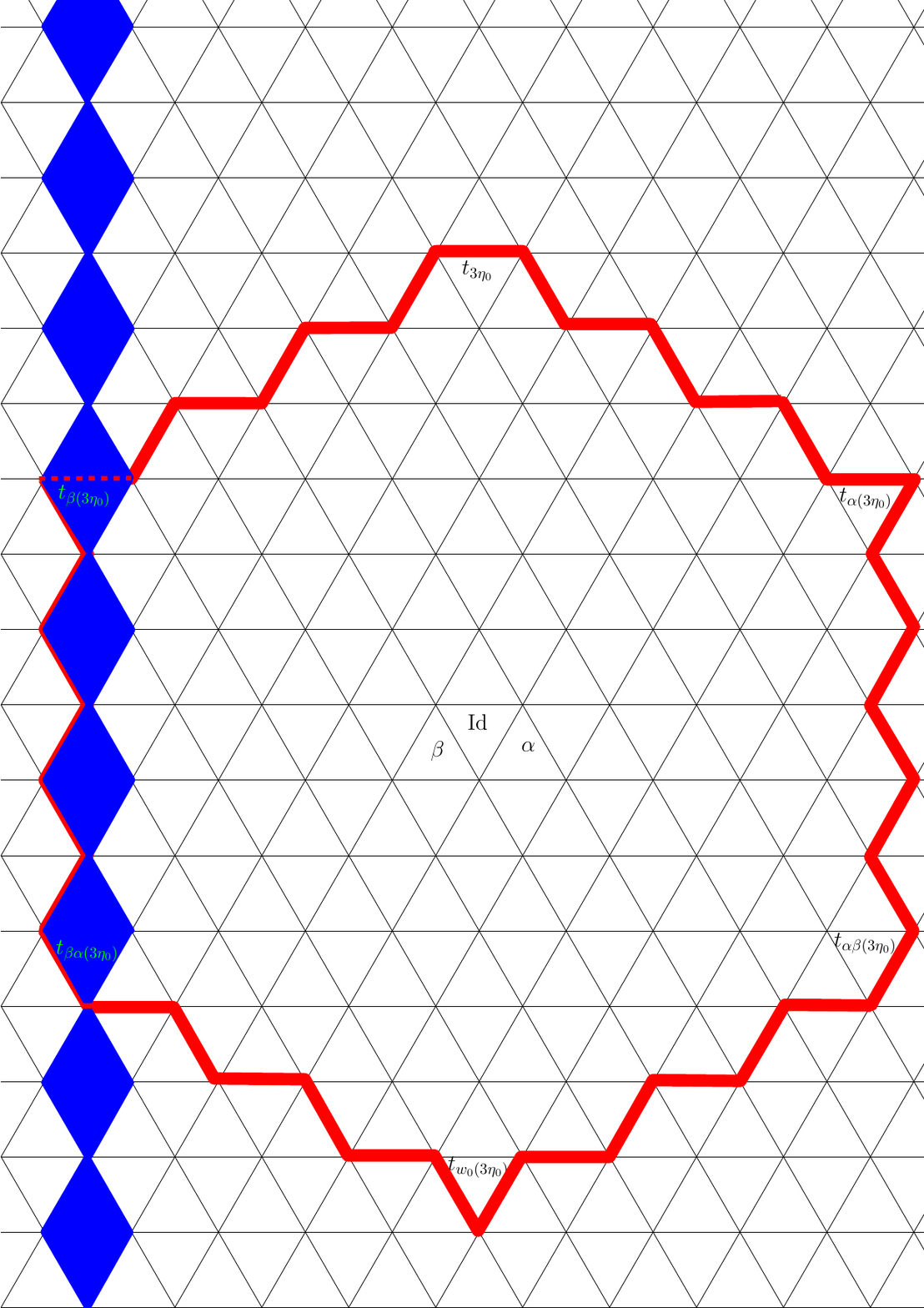} 
        \caption{\small{The $G=\GL_3$ $3\eta_0$-admissible set is in red.
        The set $\beta^{-1} W_{a,\alpha}t_{e\eta_0}\beta$ is in blue.}}
 \label{pic:strip}
\end{figure}

\begin{lemma}\label{lemma:2comps}
Let $\alpha$ be a simple root and $w$ be an element of $W$.
If $\tld{w} \in w^{-1} W_{a,\alpha} t_{e\eta_0} w \cap \Adm(e\eta_0)$ and $\tld{w} \leq t_{\sigma^{-1}(e\eta_0)}$ for some $\sigma \in W$, then $\sigma \in \{w,s_\alpha w\}$.
\end{lemma}
\begin{proof}
Suppose that $\tld{w}$ is as in the statement.
If $\tld{w} = t_{w^{-1}(e\eta_0)}$ or $t_{(s_\alpha w)^{-1}(e\eta_0)}$ and $\tld{w} \leq t_{\sigma^{-1}(e\eta_0)}$, then $\tld{w} = t_{\sigma^{-1}(e\eta_0)}$ since $\ell(\tld{w}) = \ell(t_{\sigma^{-1}(e\eta_0)})$ and the conclusion follows. 
Otherwise, $\tld{w} \leq t_{w^{-1}(e\eta_0)},t_{(s_\alpha w)^{-1}(e\eta_0)}$ by Proposition \ref{prop:listshapes} and the first part of its proof (applying Proposition \ref{prop:listshapes} with $w$ both taken to be $w$ and $s_\alpha w$ here). 
But by Corollary \ref{cor:at_most_two}, there are at most two $\sigma \in W$ with $\tld{w} \leq t_{\sigma^{-1}(e\eta_0)}$ (the reader can check that the proof of Corollary \ref{cor:at_most_two} only involves geometric properties of Pappas--Zhu local models, and does not make use of any of the results of this section). 
The conclusion follows. 
\end{proof}

\begin{prop}\label{prop:ends}
Let $\alpha$ be a simple root and $w$ be an element of $W$.
Suppose that 
\begin{enumerate}
\item \label{item:corridor} $\tld{x} \in w^{-1} W_{a,\alpha} t_{e\eta_0} w \cap \Adm(e\eta_0)$; 
\item \label{item:up} $\tld{w}_2\in \tld{W}^+$ and $\tld{w}_\lambda\in \tld{W}^+_1$ such that $\tld{w}_2 \uparrow \tld{w}_h\tld{w}_\lambda$; and 
\item \label{item:weightineq} $\tld{w}_2\tld{x} \leq w_0 t_{(e-1)\eta_0} \tld{w}_\lambda$.
\end{enumerate}
Then $\tld{w}_2$ equals $\tld{w}_h\tld{w}_\lambda$ and the image of $\tld{w}_\lambda$ in $W$ is in the set $\{w,s_\alpha w\}$.
Moreover, we can take $\tld{w}_\lambda$ as above to have image $w \in W$ \emph{(}resp.~$s_\alpha w \in W$\emph{)} if and only if $\tld{x} \neq t_{(s_\alpha w)^{-1}(e\eta_0)}$ \emph{(}resp.~$\tld{x} \neq t_{w^{-1}(e\eta_0)}$\emph{)}.
\end{prop}
\begin{proof}
There exists a dominant weight $\omega$ (unique up to $X^0(T)$) such that $t_{-\omega}\tld{w}_2\in \tld{W}^+_1$.
Then $t_{-w_0(\omega)}\tld{w}_\lambda \in \tld{W}^+$ and item (\ref{item:up}) and \cite[Proposition 4.1.2]{LLL} give us $t_{-w_0(\omega)}\tld{w}_\lambda \uparrow t_{-w_0(\omega)}\tld{w}_h^{-1} \tld{w}_2 = \tld{w}_h^{-1} t_{-\omega}\tld{w}_2$.
Then Wang's theorem implies that $t_{-w_0(\omega)}\tld{w}_\lambda \leq \tld{w}_h^{-1} t_{-\omega}\tld{w}_2$, and so by \cite[Lemma 4.1.9]{LLL} we have
\begin{align*}
\tld{x}\leq \tld{w}_2^{-1}w_0 t_{(e-1)\eta_0} \tld{w}_\lambda &= (t_{-\omega}\tld{w}_2)^{-1} w_0 t_{(e-1)\eta_0} t_{-w_0(\omega)}\tld{w}_\lambda \\ 
&\leq (t_{-\omega}\tld{w}_2)^{-1} w_0 t_{(e-1)\eta_0} \tld{w}_h^{-1} t_{-\omega}\tld{w}_2 = t_{(w_0w_2)^{-1}(e\eta_0)}
\end{align*}
where $w_2\in W$ is the image of $\tld{w}_2$.
Lemma \ref{lemma:2comps} implies that $w_0w_2 \in \{w,s_\alpha w\}$.

Suppose without loss of generality that $w_0w_2 = w$.
Let $\tld{w}_2$ be $t_\omega \tld{w}_h \tld{w}$ where $\tld{w}\in \tld{W}_1^+$ has image $w\in W$ and $\omega \in X^*(T)$ is dominant ($\omega$ in the last paragraph can be chosen to coincide with $\omega$ here).
By (\ref{item:corridor}), we let $\tld{x}$ be $(\tld{w}_h \tld{w})^{-1}w_0 \tld{w}_\alpha t_{(e-1)\eta_0} \tld{w}$ for some $\tld{w}_\alpha \in W_{a,\alpha}$.
Then (\ref{item:weightineq}) becomes $t_\omega w_0 \tld{w}_\alpha t_{(e-1)\eta_0} \tld{w} \leq w_0 t_{(e-1)\eta_0}\tld{w}_\lambda$ which implies by Lemma \ref{lemma:bruhatup1} that 
\[
t_{w_0(\omega)}\tld{w}_\alpha t_{(e-1)\eta_0} \tld{w} \uparrow t_{(e-1)\eta_0}\tld{w}_\lambda,
\]
which upon multiplying by $t_{-(e-1)\eta_0}$ and using item (\ref{item:up}) and \cite[Proposition 4.1.2]{LLL} gives
\[
\tld{w}_\alpha' \tld{w}_h^{-1}\tld{w}_2 \uparrow \tld{w}_\lambda \uparrow \tld{w}_h^{-1}\tld{w}_2
\]
for some $\tld{w}'_\alpha \in W_{a,\alpha}$ (using that $W_{a,\alpha}$ is stable under conjugation by $X^*(\un{T})$).
Then $\tld{w}_\lambda \in W_{a,\alpha} \tld{w}_h^{-1}\tld{w}_2$ by Lemma \ref{lemma:simpleup}, or equivalently $\tld{w}_2 \in W_{a,-w_0(\alpha)} \tld{w}_h\tld{w}_\lambda$. 

That $\tld{w}_2\in \tld{W}^+$, $\tld{w}_2 \uparrow \tld{w}_h \tld{w}_\lambda$, and $\tld{w}_h\tld{w}_\lambda \in \tld{W}^+_1$ imply respectively that 
\[
0 \leq \lfloor\langle \tld{w}_2(x),-w_0(\alpha^\vee)\rangle\rfloor \leq \lfloor\langle \tld{w}_h\tld{w}_\lambda(x),-w_0(\alpha^\vee) \rangle \rfloor = 0
\]
for any $x\in A_0$, which forces equalities throughout.
Combined with the fact that $\tld{w}_2 \in W_{a,-w_0(\alpha)} \tld{w}_h\tld{w}_\lambda$, we see that $\tld{w}_2 = \tld{w}_h \tld{w}_\lambda$.
In particular, the image of $\tld{w}_\lambda$ in $W$ is $w_0w_2$.

The final part follows from the first part of the proof of Proposition \ref{prop:listshapes}.
\end{proof}

\begin{prop}\label{prop:upset}
Let $\alpha$ be a simple root and $w$ be an element of $W$.
If $\tld{w}_1 \in w^{-1} W_{a,\alpha} t_{e\eta_0} w$ and $\tld{w}_1\leq \tld{w}_2 \leq t_{w^{-1}(e\eta_0)}$, then $\tld{w}_2 \in w^{-1} W_{a,\alpha} t_{e\eta_0} w$.
\end{prop}
\begin{proof}
By Lemma \ref{lemma:bruhatup1}, the inequalities $\tld{w}_1\leq \tld{w}_2 \leq t_{w^{-1}(e\eta_0)}$ imply that $w\tld{w}_1 \uparrow s\tld{w}_2 \uparrow wt_{w^{-1}(e\eta_0)}$ where $s \in W$ is the unique element such that $s\tld{w}_2\in \tld{W}^+$.
Since $w\tld{w}_1 \in W_{a,\alpha} wt_{w^{-1}(e\eta_0)}$, we deduce from Lemma \ref{lemma:simpleup} that $s\tld{w}_2 \in W_{a,\alpha} wt_{w^{-1}(e\eta_0)}$ or equivalently, that $\tld{w}_2 \in s^{-1}W_{a,\alpha} wt_{w^{-1}(e\eta_0)}$.

We now narrow the possibilities of $s$.
Since $\tld{w}_1 \leq \tld{w}_2 \leq t_{s^{-1}(e\eta_0)}$, where the final inequality follows from \cite[Corollary 4.4]{HH}, $s \in \{w,s_\alpha w\}$ by Lemma \ref{lemma:2comps}.
Combining with the above paragraph, $\tld{w}_ 2 \in w^{-1}W_{a,\alpha}t_{e\eta_0} w$.
\end{proof}

\subsection{The weight part of a Serre-type conjecture for tame representations}
\label{sec:SWC}

The aim of this section, and of the following one, is to recollect the necessary notions to formulate the weight part for Serre conjectures, and to pursue a combinatorial study of the set of conjectural modular weights in terms of the geometry of the affine Weyl group.

\subsubsection{Serre weights}
\label{subsub:SW}
Recall from \ref{sec:not:RG} that $G$ is a split group defined over $\Fp$, $k_p$ is a finite \'etale $\F_p$-algebra, $G_0 = \Res_{k_p/\F_p} G_{/k_p}$ and $\F$ contains the image of any ring homomorphism $k_p \ra \ovl{\F}_p$ so that $\un{G} \defeq (G_0)_{/\F}\cong G_{/\F}^{\Hom(k_p,\F)}$. 
Let $\rG$ be $G_0(\Fp)$. 
A \emph{Serre weight} (of $\rG$) is an absolutely irreducible $\F$-representation of $\rG$.

Let $\lambda\in X^*(\un{T})$ be a dominant character.
We write $W(\lambda)_{/\F}$ for the $\un{G}$-module $\Ind_{\un{B}}^{\un{G}} w_0 \lambda$.
Let $F(\lambda)$ denote the (irreducible) socle of the $\rG$-restriction of $W(\lambda)_{/\F}(\F)$. 

We define:
\[
X_1(\un{T}) \defeq \left\{\lambda\in X^*(\un{T}),0\leq \langle \lambda,\alpha^\vee\rangle\leq p-1\text{ for all }\alpha\in \un{\Delta}\right\}
\]
which we call the set of $p$-restricted weights. 
Then the map $\lambda \mapsto F(\lambda)$ defines a bijection  from $X_1(\un{T})/(p-\pi)X^0(\un{T})$ to the set of isomorphism classes of Serre weights of $\rG$ (see \cite[Lemma 9.2.4]{GHS}).
We say that $\lambda\in X_1(\un{T})$ is \emph{regular $p$-restricted} if $\langle \lambda,\alpha^\vee\rangle < p-1$ for all $\alpha\in \un{\Delta}$ and say a Serre weight $F(\lambda)$ is \emph{regular} if $\lambda$ is. 
Similarly we say that $F(\lambda)$ is $m$-deep if $\lambda$ is $m$-deep.

To handle the combinatorics of Serre weights it is convenient to introduce the notion of $p$-alcoves and the $p$-dot action on them.
A $p$-alcove is a connected component of $X^*(\un{T})\otimes_{\Z}\R\ \setminus\ \big(\bigcup_{(\alpha,pn)}(H_{\alpha,pn}-\eta_0)\big)$ and we say that a $p$-alcove $\un{C}$ is
\emph{dominant} (resp.~\emph{$p$-restricted}) if $0 <  \langle\lambda + \eta_0,\alpha^\vee\rangle$ (resp.~if $0 <  \langle\lambda + \eta_0,\alpha^\vee\rangle<p$) for all $\alpha\in \un{\Delta}$ and $\lambda\in \un{C}$.
We write $\un{C}_0$ for the dominant base $p$-alcove, i.e.~the alcove characterized by $\lambda\in \un{C}_0$ if and only if $0 <\langle\lambda + \eta_0,\alpha^\vee\rangle<p$ for all $\alpha\in \un{\Phi}^+$.
The $p$-\emph{dot action} of $\tld{\un{W}}$ on $X^*(\un{T})\otimes_{\Z}\R$ is defined by $\tld{w}\cdot \lambda\defeq w(\lambda+\eta_0+p\nu)-\eta_0$ for $\tld{w}=wt_\nu\in \tld{\un{W}}$ and $\lambda\in X^*(\un{T})\otimes_{\Z}\R$.
Then we have
\[\tld{\un{W}}^+=\{\tld{w}\in \tld{\un{W}}:\tld{w}\cdot \un{C}_0 \textrm{ is dominant}\}\]
and
\[\tld{\un{W}}^+_1=\{\tld{w}\in \tld{\un{W}}^+:\tld{w}\cdot \un{C}_0 \textrm{ is } p\textrm{-restricted}\}\]
and $\un{\Omega}$ is the stabilizer of $\un{C}_0$ for the dot action.

We have an equivalence relation on $\tld{\un{W}} \times X^*(\un{T})$  defined by $(\tld{w},\omega) \sim (t_\nu \tld{w},\omega-\nu)$ for all $\nu \in X^0(\un{T})$ (\cite[\S 2.2]{MLM}). 
For $(\tld{w}_1,\omega-\eta_0)\in \tld{\un{W}}^+_1\times (X^*(\un{T})\cap \un{C}_0)/\sim$, we define the Serre weight $F_{(\tld{w}_1,\omega)}\defeq F(\pi^{-1}(\tld{w}_1)\cdot (\omega-\eta_0))$ (this only depends on the equivalence class of $(\tld{w}_1,\omega)$).
We call the equivalence class of $(\tld{w}_1,\omega)$ a \emph{lowest alcove presentation} for the Serre weight $F_{(\tld{w}_1,\omega)}$ and note that $F_{(\tld{w}_1,\omega)}$ is $m$-deep if and only if $\omega-\eta_0$ is $m$-deep in alcove $\un{C}_0$.
(We often implicitly choose a representative for a lowest alcove presentation to make \emph{a priori} sense of an expression, though it is \emph{a posteriori} independent of this choice.)

\subsubsection{Deligne--Lusztig representations}
To a \emph{good} pair $(s,\mu)\in \un{W}\times X^*(\un{T})$ we attach a Deligne--Lusztig representation $R_s(\mu)$ of $\rG$ defined over $E$  (see \cite[\S 2.2]{LLL} and \cite[Proposition 9.2.1, 9.2.2]{GHS}, where the representation $R_s(\mu)$ is there denoted $R(s,\mu)$). 
We call $(s,\mu-\eta_0)$ a \emph{lowest alcove presentation} for $R_s(\mu)$ and  say that $R_s(\mu)$ is \emph{$N$-generic} if $\mu-\eta_0$ is $N$-deep in alcove $\un{C}_0$ for $N\geq 0$
If $\mu-\eta_0$ is $1$-deep in $\un{C}_0$ then $R_s(\mu)$ is an irreducible representation.
We say that a Deligne--Lusztig representation $R$ is $N$-generic if there exists an isomorphism $R\cong R_s(\mu)$ where $R_s(\mu)$ is \emph{$N$-generic}.

\subsubsection{Tame inertial types}
\label{subsubsec:TIT}
An inertial type (for $K$, over $E$) is the $\GL_n(E)$-conjugacy class of an homomorphism  $\tau:I_{K}\ra\GL_n(E)$ with open kernel and which extends to the Weil group of $G_K$.
An inertial type is \emph{tame} if one (equivalently, any) homomorphism in the conjugacy class factors through the tame quotient of $I_{K}$. 

Let $s\in \un{W}$ and $\mu \in  X^*(\un{T})\cap \un{C}_0$.
Associated to this data we have an integer $r$ (the order of the element $s_0s_1\cdots s_{f-2}s_{f-1}\in W$), $n$-tuples $\bf{a}^{\prime (j')}\in\Z^n$ for $0\leq j'\leq fr-1$, and a tame inertial type $\tau(s,\mu+\eta_0)\defeq \sum_{i=1}^{n}(\omega_{fr})^{\mathbf{a}_i^{\prime(0)}}$. 
(See \cite[Example 2.4.1, equations (5.2), (5.1)]{MLM} for the explicit construction of the $n$-tuples $\bf{a}^{\prime (j')}\in\Z^n$.)
We say that $\tau(s,\mu+\eta_0)$ is a \emph{principal series type} if $r=1$.

If $N \geq 0$ and $\mu$ is $N$-deep in alcove $\un{C}_0$, the pair $(s,\mu)$ is said to be \emph{an $N$-generic lowest alcove presentation} for the tame inertial type $\tau(s,\mu+\eta_0)$. 
We say that a tame inertial type is $N$-generic if it admits an $N$-generic lowest alcove presentation. 
(Different pairs $(s,\mu)$ can give rise to isomorphic tame inertial types, see \cite[Proposition 2.2.15]{LLL}.) 

If $(s,\mu)$ is a lowest alcove presentation of $\tau$, let $\tld{w}(\tau) \defeq t_{\mu+\eta_0}s\in\tld{\un{W}}$. (In particular, when writing $\tld{w}(\tau)$ we use an implicit lowest alcove presentation for $\tau$).

Inertial $\F$-types are defined similarly with $E$ replaced by $\F$. 
Tame inertial $\F$-types have analogous notions of lowest alcove presentations and genericity.
If $\taubar$ is a tame inertial $\F$-type we write $[\taubar]$ to denote the tame inertial type over $E$ obtained from $\taubar$ using the Teichm\"uller section $\F^\times \into \cO^\times$.

Assume that $\mu$ is $1$-deep in $\un{C}_0$.
For each $0\leq j'\leq fr-1$ we define $s'_{\orient,j'}$ to be the (necessarily unique) element of $W$ such that $(s'_{\orient,j'})^{-1}(\bf{a}^{\prime\,(j')})\in\Z^n$ is dominant.
(In the terminology of \cite{LLLM}, cf.~Definition 2.6 of \emph{loc.~cit.}, the $fr$-tuple $(s'_{\orient,j'})_{0\leq j'\leq fr-1}$ is the \emph{orientation} of $(\bf{a}^{\prime\,(j')})_{0\leq j'\leq fr-1}$.)
We will need the observation that $(s'_{\mathrm{or},  j})^{-1} (\bf{a}^{\prime\, (j)})$ equals $s_{j}^{-1}\big(\mu_{j} + \eta_{0,j}\big)$ modulo $p$ for all $0\leq j\leq f-1$.

\subsubsection{Inertial local Langlands}
\label{subsub:ILL}
Let $K/\Qp$ be a finite extension with ring of integers $\cO_K$ and residue field $k$. 
Let $\tau$ be a tame inertial type for $K$. 
By \cite[Theorem 3.7]{CEGGPS} there exists an irreducible smooth representation $\sigma(\tau)$ of $\GL_n(\cO_K)$ over $E$ satisfying results towards the inertial local Langlands correspondence.

Let $k_p = k$, so that $G_0(\Fp)\cong\GL_n(k)$. 
When $\tau\cong\tau(s,\mu)$ with $\mu-\eta\in \un{C}_0$ a $1$-deep character, the representation $\sigma(\tau)$ can and will be taken to be (the inflation to $\GL_n(\cO_K)$ of) $R_s(\mu)$ (see \cite[Theorem 2.5.3]{MLM}, \cite[Corollary 2.3.5]{LLL}).

The following definition will play a key role in our generalization of Herzig's Serre weight conjecture.
\begin{defn}
\label{defn:SWC:ram}
Let $\cR$ denote the bijection on regular Serre weights given by $F(\lambda)\mapsto F(\tld{w}_h \cdot \lambda)$.
If $\taubar:I_K\ra\GL_n(\F)$ is a $1$-generic tame inertial $\F$-type for $I_K$ we define
\[
W^?(\taubar)\defeq\cR\Big(\JH\Big(\ovl{\sigma([\taubar])}\otimes W((1-e)w_0\eta_0)\Big)\Big).
\]
\end{defn}

\subsubsection{Compatibilities of lowest alcove presentations}
Recall the canonical isomorphism $\tld{\un{W}}/\un{W}_a\stackrel{\sim}{\ra}X^*(\un{Z})$.
Let $\zeta\in X^*(\un{Z})$.
We say that an element $\tld{w}\in\tld{\un{W}}$ is $\zeta$-compatible if it corresponds to $\zeta$ via the isomorphism $\tld{\un{W}}/\un{W}_a\stackrel{\sim}{\ra}X^*(\un{Z})$.

A lowest alcove presentation $(s,\mu)$ for a tame inertial type $\tau$ for $K$ over $E$ or a Deligne--Lusztig representation $R$ is $\zeta$-compatible if $t_{\mu+\eta_0}s\in\tld{\un{W}}$ is $\zeta$-compatible.
A lowest alcove presentation $(s,\mu)$ of a tame inertial $\F$-type is $\zeta$-compatible if $t_{\mu-(e-1)\eta_0}s\in\tld{\un{W}}$ is $\zeta$-compatible.
(If $\tau$ is a tame inertial type and $\taubar$ is the tame inertial $\F$-type obtained by reduction, the same lowest alcove presentation of $\tau$ and $\taubar$ are compatible with elements of $X^*(\un{Z})$ that differ by $\eta_0|_{\un{Z}}$.)
A lowest alcove presentation $(\tld{w}_1,\omega)$ for Serre weight is $\zeta$-compatible if the element $t_{\omega-\eta_0}\tld{w}_1\in \tld{\un{W}}$ is $\zeta$-compatible.
Finally, lowest alcove presentations (of possibly different types of objects) are compatible if they are all $\zeta$-compatible for some $\zeta \in X^*(\un{Z})$. 

\subsubsection{$L$-parameters}
\label{subsub:Lp}
Recall from \S \ref{sec:not:RG} the finite \'etale $\Qp$-algebra $F_p^+$.
We adapt the constructions of tame inertial types and the inertial local Langlands above to arbitrary $S_p$.
We assume that $E$ contains the image of any homomorphism $F_p^+\rightarrow \ovl{\Q}_p$.
Let 
\[
\un{G}^\vee\defeq\prod_{F_p^+\ra E}G^{\vee}_{/\cO}
\]
be the dual group of $\Res_{F^+_p/\Qp}(G_{/F^+_p})$ and ${}^L\un{G}\defeq \un{G}^\vee\rtimes\Gal(E/\Qp)$  the Langlands dual group of $\Res_{F^+_p/\Qp}(G_{/F^+_p})$ (where $\Gal(E/\Qp)$ acts on the set $\{F_p^+\ra E\}$ by post-composition).
An $L$-parameter (over $E$) is a $\un{G}^\vee(E)$-conjugacy class of an $L$-homomorphism, i.e.~of a continuous homomorphism $\rho:G_{\Qp}\ra{}^L\un{G}(E)$ which is compatible with the projection to $\Gal(E/\Qp)$.
An inertial $L$-parameter is a $\un{G}^\vee(E)$-conjugacy class of an homomorphism $\tau:I_{\Qp}\ra\un{G}^\vee(E)$ with open kernel, and which admits an extension to an $L$-homomorphism $G_{\Qp}\ra{}^L\un{G}(E)$.
An inertial $L$-parameter is \emph{tame} if some (equivalently, any) representative in its equivalence class factors through the tame quotient of $I_{\Qp}$.

The argument of \cite[Lemmas 9.4.1, 9.4.5]{GHS} carries over in our setting and we have a bijection between $L$-parameters (resp.~tame inertial $L$-parameters) and collections of the form $(\rho_v)_{v\in S_p}$ (resp.~of the form $(\tau_v)_{v\in S_p}$) where for all $v\in S_p$ the element $\rho_v:G_{F^+_v}\ra\GL_n(E)$ is a continuous Galois representation (resp.~the element $\tau_v:I_{F^+_v}\ra\GL_n(E)$ is a tame inertial type for $F^+_v$).
(This bijection depends on a choice of isomorphisms $\ovl{F^+_v}\stackrel{\sim}{\ra}\ovl{\Q}_p$ for all $v\in S_p$.)
We have similar notions for $L$-parameters (resp.~inertial $L$-parameters) over $\F$.

In this setting, given a tame inertial $L$-parameter $\tau$ corresponding to the collection of tame inertial types $(\tau_v)_{v\in S_p}$, we let $\sigma(\tau)$ be the irreducible smooth $E$-valued representation of $\GL_n(\cO_p)$ given by $\otimes_{v\in S_p} \sigma(\tau_v)$, where $\sigma(\tau_v)$ is the smooth representation corresponding to $\tau_v$ via the inertial local Langlands correspondence appearing in \S \ref{subsub:ILL}.

\subsection{Combinatorics of Serre weights}
\label{sec:combserre}
In this section, we apply the results of \S \ref{sec:comb:weyl} on extended affine Weyl groups to analyze the combinatorics of the Serre weight sets defined in \S \ref{sec:SWC}. 
We assume for simplicity that $F_p^+=K$, but the results herein do not require this.
Given tame inertial types $\tau$ and $\rhobar^{\speci}$ over $E$ and $\F$, respectively, with fixed compatible lowest alcove presentations, we define $\tld{w}(\rhobar^{\speci},\tau)$ to be $\tld{w}(\tau)^{-1}\tld{w}(\rhobar^{\speci}) \in t_{e\eta_0}\un{W}_a$. 

The following two results follow readily from \cite[Proposition 2.3.7]{MLM}.

\begin{prop}\label{prop:JH}
Suppose we fix a $2h_{\eta_0}$-generic lowest alcove presentation of a tame inertial type $\tau$. 
The map 
\begin{equation}
\label{eqn:JH}
(\tld{w}_\lambda,\tld{w}_2) \longmapsto F_{(\tld{w}_\lambda, \tld{w}(\tau)\tld{w}_2^{-1}(0))}
\end{equation}
induces a bijection between 
\begin{itemize}
\item the set of pairs $(\tld{w}_\lambda,\tld{w}_2)$, modulo the diagonal action of $X^0(T)$, with $\tld{w}_\lambda \in \tld{W}^+_1$ and $\tld{w}_2\in \tld{W}^+$ such that $\tld{w}_\lambda \uparrow \tld{w}_h^{-1} \tld{w}_2$; and
\item elements of $\JH(\ovl{\sigma}(\tau))$.
\end{itemize}
Moreover, these lowest alcove presentations of Serre weights are compatible with that of $\tau$.
Finally, the weight corresponding to $(\tld{w}_\lambda,\tld{w}_h\tld{w}_\lambda)$ appears as a Jordan--H\"older factor with multiplicity one.
\end{prop}

\begin{prop}\label{prop:W?}
Suppose we fix an $(\max\{2,e\}h_{\eta_0})$-generic lowest alcove presentation of a tame inertial type $\rhobar^\speci$ over $\F$.
The map 
\[
(\tld{w}_\lambda,\tld{w}_2)\longmapsto 
F_{(\tld{w}_\lambda, \tld{w}^*(\rhobar^\speci)\tld{w}_2^{-1}(0))}.
\]
induces a bijection between
\begin{itemize}
\item pairs $(\tld{w}_\lambda,\tld{w}_2)$ with $\tld{w}_\lambda \in \tld{W}^+_1$ and $\tld{w}_2\in \tld{W}^+$, up to the diagonal $X^0(\un{T})$-action, such that $\tld{w}_2 \uparrow t_{(e-1)\eta_0}\tld{w}_\lambda$; and
\item elements of $W^?(\rhobar^\speci)$.
\end{itemize}
Moreover, these lowest alcove presentations of Serre weights are compatible with that of $\rhobar^\speci$.
\end{prop}

The following definition is central to this paper. 

\begin{defn}[Extremal weights]
\label{defn:obv}
Suppose we fix a $(\max\{2,e\}h_{\eta_0})$-generic lowest alcove presentation for a tame inertial type $\rhobar^\speci$ over $\F$.
Let $w$ be an element of $W$ and $\tld{w} \in \tld{W}_1^+$ be an element (unique up to $X^0(\un{T})$) whose image in $W$ is $w$.
The weight
\[
F_{(\tld{w},\tld{w}(\rhobar^\speci)\tld{w}^{-1}(-(e-1)\eta_0))}\in W^?(\rhobar^\speci)
\]
is called the \emph{extremal weight of $\rhobar^\speci$ corresponding to} $w$. 
Let $W_\obv(\rhobar^\speci)$ be the set of all extremal weights of $\rhobar^\speci$. 
(While the extremal weight corresponding to $w$ depends on the choice of lowest alcove presentation, the set of all extremal weights does not.) 
\end{defn}

\begin{rmk}
If $\rhobar: G_K \ra \GL_n(\F)$ is semisimple and $2h_{\eta_0}$-generic, and $K$ is unramified, the notion of obvious weight for $\rhobar$ corresponding to $w$ (\cite[Definition 2.6.3]{MLM}) and of extremal weight for $\rhobar$ corresponding to $w$ coincide, and the set $W_\obv(\rhobar|_I)$ is the set defined in \cite[Definition 7.1.3]{GHS}.
\end{rmk}

The following combinatorial result relates the set $W^?(\rhobar^\speci)$ and the admissible set and is key to weight elimination. 

\begin{prop} \label{prop:combWE}
Suppose we fix an $(\max\{2,e\}h_{\eta_0})$-generic lowest alcove presentation for $\tld{w}(\rhobar^\speci)$.
Let $(\tld{w}_\lambda,\omega)$ be a compatible lowest alcove presentation of a $3h_{\eta_0}$-deep Serre weight $\sigma$. 
Then $\omega = \tld{w}(\rhobar) \tld{w}^{-1}(0)$ for a unique $\tld{w} \in \tld{W}^+$. 
Let $\tau$ be the tame inertial type over $E$ with $\tld{w}(\rhobar^\speci,\tau) = (\tld{w}_h\tld{w}_\lambda)^{-1} w_0 \tld{w}$ for some (necessarily compatible) lowest alcove presentation. 
Then 
\begin{enumerate}
\item $\sigma \in \JH(\ovl{\sigma}(\tau))$; and
\item $(\tld{w}_h\tld{w}_\lambda)^{-1} w_0 \tld{w} \in \Adm(e\eta_0)$ implies that $\sigma \in W^?(\rhobar^\speci)$.
\end{enumerate}
\end{prop}
\begin{proof}
By definition of $\tau$, we have that $\tld{w}(\tau)(\tld{w}_h\tld{w}_\lambda)^{-1}(0) = \tld{w}(\rhobar^\speci) \tld{w}^{-1}(0)$.
Note that the lowest alcove presentation of $\tau$ is $2h_{\eta_0}$-generic by the depth assumption on $\sigma$.
Then $\sigma$ corresponds to the pair $(\tld{w}_\lambda,\tld{w}_h\tld{w}_\lambda)$ in \eqref{eqn:JH}.

Suppose that $(\tld{w}_h\tld{w}_\lambda)^{-1} w_0  \tld{w} \in \Adm(e\eta_0)$. 
If we let $w_\lambda \in W$ be the image of $\tld{w}_\lambda$, then we claim that $w_\lambda (\tld{w}_h\tld{w}_\lambda)^{-1} w_0 \tld{w} \in \tld{W}^+$.
Indeed since $\tld{w}_\lambda \in \tld{W}^+_1$, $w_0 w_\lambda (\tld{w}_h\tld{w}_\lambda)^{-1}$ is an antidominant translation so that $w_0 w_\lambda (\tld{w}_h\tld{w}_\lambda)^{-1} w_0 \tld{w}\cdot C_0$ is in the antidominant Weyl chamber.
By \cite[Corollary 4.4]{HH}, we have that 
\[
(\tld{w}_h\tld{w}_\lambda)^{-1} w_0 \tld{w} \leq t_{w_\lambda^{-1} (e\eta_0)} = (\tld{w}_h\tld{w}_\lambda)^{-1} w_0 t_{(e-1)\eta_0} \tld{w}_\lambda.
\]
Since these expressions are reduced by \cite[Lemma 4.9]{LLL}, we conclude that $\tld{w} \leq t_{(e-1)\eta_0} \tld{w}_\lambda$ which implies that $\tld{w}\uparrow t_{(e-1)\eta_0} \tld{w}_\lambda$.
We conclude from Proposition \ref{prop:W?} that $\sigma \in W^?(\rhobar^\speci)$.
\end{proof}

Denote by $W^?(\rhobar^\speci,\tau)$ the intersection $W^?(\rhobar^\speci) \cap \JH(\ovl{\sigma}(\tau))$.  

\begin{prop}\label{prop:intersect}
Suppose we fix compatible $(\max\{2,e\}h_{\eta_0})$-generic and $2h_{\eta_0}$-generic lowest alcove presentations of tame inertial types $\rhobar^{\speci}$ and $\tau$ over $\F$ and $E$, respectively.
Then the set $W^?(\rhobar^\speci,\tau)$ is exactly the set of weights in \eqref{eqn:JH} such that $\tld{w}_2 \tld{w}(\rhobar^\speci,\tau) \leq w_0 t_{(e-1)\eta_0} \tld{w}_\lambda$.
\end{prop}
\begin{proof}
Consider an element $\sigma\in \JH(\ovl{\sigma}(\tau))$.
Let $\tld{w}_\lambda$ and $\tld{w}_2$ be as in Proposition \ref{prop:JH}.
By Proposition \ref{prop:W?} and uniqueness of compatible lowest alcove presentations (see \cite[Lemma 2.2.3]{MLM}), $\sigma \in W^?(\rhobar^\speci)$ if and only if there exist $\tld{s}_2 \in \tld{W}^+$ with $\tld{s}_2\uparrow t_{(e-1)\eta_0}\tld{w}_\lambda$ such that
\[
\tld{w}(\tau) \tld{w}_2^{-1}(0) = \tld{w}(\rhobar^\speci)\tld{s}_2^{-1}(0),
\]
or equivalently, $\tld{w}_2\tld{w}(\rhobar^\speci,\tau)\in W \tld{s}_2$. 

We now show that there exists $\tld{s}_2 \in \tld{W}^+$ with $\tld{s}_2\uparrow t_{(e-1)\eta_0}\tld{w}_\lambda$ such that $\tld{w}_2\tld{w}(\rhobar^\speci,\tau)\in W \tld{s}_2$ if and only if $\tld{w}_2\tld{w}(\rhobar^\speci,\tau) \leq w_0 t_{(e-1)\eta_0} \tld{w}_\lambda$.
First suppose that such an $\tld{s}_2$ exists.
This implies that 
\[
\tld{w}_2\tld{w}(\rhobar^\speci,\tau)\leq w_0 \tld{s}_2 \leq w_0 t_{(e-1)\eta_0} \tld{w}_\lambda,
\]
where the second inequality follows from the fact that $\tld{s}_2\leq t_{(e-1)\eta_0}\tld{w}_\lambda$ by Wang's theorem.
Conversely, if $\tld{w}_2\tld{w}(\rhobar^\speci,\tau)\leq w_0 t_{(e-1)\eta_0} \tld{w}_\lambda$, then using that $w_0 (t_{(e-1)\eta_0} \tld{w}_\lambda)$ is a reduced factorization, $\tld{w}_2\tld{w}(\rhobar^\speci,\tau) = w\tld{s}_2$ for some $\tld{s}_2 \in \tld{W}^+$ with $\tld{s}_2\leq t_{(e-1)\eta_0}\tld{w}_\lambda$ (or equivalently $\tld{s}_2\uparrow t_{(e-1)\eta_0}\tld{w}_\lambda$ by Wang's theorem) and $w\in W$.
\end{proof}

\begin{cor}\label{cor:adm}
Suppose that tame inertial types $\rhobar^{\speci}$ and $\tau$ over $\F$ and $E$ have compatible $(\max\{2,e\}h_{\eta_0})$-generic and $2h_{\eta_0}$-generic lowest alcove presentations, respectively.
If $W^?(\rhobar^\speci,\tau)$ is nonempty, then $\tld{w}(\rhobar^\speci,\tau) \in \Adm(e\eta_0)$.
\end{cor}
\begin{proof}
As in the statement of the corollary, we fix compatible lowest alcove presentations for $\rhobar^\speci$ and $\tau$, respectively. 
If $W^?(\rhobar^\speci,\tau)$ is nonempty, by Proposition \ref{prop:intersect} we have that $\tld{w}_2 \tld{w}(\rhobar^\speci,\tau) \leq w_0 t_{(e-1)\eta_0} \tld{w}_\lambda$ for some $\tld{w}_2\in \tld{W}^+$ and $\tld{w}_\lambda \in \tld{W}^+_1$ with $\tld{w}_\lambda \uparrow \tld{w}_h^{-1}\tld{w}_2$.
Then $\tld{w}_2 \uparrow \tld{w}_h \tld{w}_\lambda$ by \cite[Proposition 4.1.2]{LLL} so that $\tld{w}_2 \leq \tld{w}_h \tld{w}_\lambda$.
Since $\tld{w}_2^{-1}w_0 t_{(e-1)\eta_0}\tld{w}_\lambda$ and $(\tld{w}_h \tld{w}_\lambda)^{-1}w_0 t_{(e-1)\eta_0}\tld{w}_\lambda$ are reduced expressions by \cite[Lemma 4.1.9]{LLL}, we have that 
\[
\tld{w}(\rhobar^\speci,\tau) \leq \tld{w}_2^{-1}w_0 t_{(e-1)\eta_0}\tld{w}_\lambda \leq \tld{w}_\lambda^{-1}\tld{w}_h^{-1}w_0 t_{(e-1)\eta_0}\tld{w}_\lambda = t_{w_\lambda^{-1}(e\eta_0)}.
\]
\end{proof}

We now establish some results which will be used to prove modularity of certain Serre weights. 

\begin{prop}\label{prop:indint}
Suppose that $\tld{w}(\rhobar^\speci,\tau) \in w^{-1} W_{a,\alpha} t_{e\eta_0} w \cap \Adm(e\eta_0)$ for compatible $(\max\{2,e\}h_{\eta_0})$-generic and $2h_{\eta_0}$-generic lowest alcove presentations for tame inertial types $\rhobar^{\speci}$ and $\tau$ over $\F$ and $E$, respectively.

Then $W^?(\rhobar^\speci,\tau)$ equals
\[
\begin{cases}
\{F_{(\tld{w},\tld{w}(\tau)\tld{w}^{-1}\tld{w}_h^{-1}(0))}\} & \mbox{if } \tld{w}(\rhobar^\speci,\tau) = t_{w^{-1}(e\eta_0)}\\
\{F_{(\tld{s_\alpha w},\tld{w}(\tau)\tld{s_\alpha w}^{-1}\tld{w}_h^{-1}(0))}\} & \mbox{if }\tld{w}(\rhobar^\speci,\tau) = t_{(s_\alpha w)^{-1}(e\eta_0)}\\
\{F_{(\tld{w},\tld{w}(\tau)\tld{w}^{-1}\tld{w}_h^{-1}(0))}, F_{(\tld{s_\alpha w},\tld{w}(\tau)\tld{s_\alpha w}^{-1}\tld{w}_h^{-1}(0))}\} & \mbox{otherwise}. 
\end{cases}
\]
Moreover, each weight appears as a Jordan--H\"older factor of $\ovl{\sigma}(\tau)$ with multiplicity one.
\end{prop}
\begin{proof}
Suppose that a weight $\sigma$ of the form (\ref{eqn:JH}) is in $\JH(\ovl{\sigma}(\tau))$.
Then by Proposition \ref{prop:intersect}, $\sigma \in W^?(\rhobar^\speci)$ if and only if $\tld{w}_2 \tld{w}(\rhobar^\speci,\tau) \leq w_0 t_{(e-1)\eta_0} \tld{w}_\lambda$.
By Proposition \ref{prop:ends}, $\tld{w}_h^{-1}\tld{w}_2$ and $\tld{w}_\lambda$ are both either $\tld{w}$ or $\tld{s_\alpha w}$.
The last part of Proposition \ref{prop:ends} implies the inclusion of $W^?(\rhobar^\speci,\tau)$ in the casewise defined sets.

On the other hand, using Proposition \ref{prop:listshapes} and
\begin{equation}\label{eqn:factor1}
t_{w^{-1}(e\eta_0 - k\alpha)} = \tld{w}^{-1}\tld{w}_h^{-1} w_0 t_{(e-1)\eta_0-k\alpha}\tld{w} = \tld{s_\alpha w}^{-1}\tld{w}_h^{-1} w_0 t_{(e-1)\eta_0 - (e-k)\alpha} \tld{s_\alpha w}
\end{equation}
and
\begin{equation}\label{eqn:factor2}
\tld{w}^{-1} s_\alpha t_{e\eta_0 - (k+1)\alpha} \tld{w} = \tld{w}^{-1}\tld{w}_h^{-1} w_0 s_\alpha t_{(e-1)\eta_0-k\alpha}\tld{w} = \tld{s_\alpha w}^{-1}\tld{w}_h^{-1} w_0 s_\alpha t_{(e-1)\eta_0 - (e-k-1)\alpha} \tld{s_\alpha w}.
\end{equation}
we have that $\tld{w}_h \tld{w}\tld{w}(\rhobar^\speci,\tau)$ is either $w_0 t_{(e-1)\eta_0 - k\alpha}\tld{w}$ or $w_0s_\alpha t_{(e-1)\eta_0-k\alpha} \tld{w}$, which is less than or equal to $w_0 t_{(e-1)\eta_0} \tld{w}$ if $k \neq e$.
This implies that $F_{(\tld{w},\tld{w}(\tau)\tld{w}^{-1}\tld{w}_h^{-1}(0))} \in W^?(\rhobar^\speci,\tau)$ if $\tld{w}(\rhobar^\speci,\tau) \neq t_{(s_\alpha w)^{-1}(e\eta_0)}$.
Similarly, $\tld{w}_h \tld{s_\alpha w}\tld{w}(\rhobar^\speci,\tau)$ is either $w_0 t_{(e-1)\eta_0 - (e-k)\alpha}\tld{s_\alpha w}$ or $w_0s_\alpha t_{(e-1)\eta_0-(e-k-1)\alpha} \tld{s_\alpha w}$, so that 
$F_{(\tld{s_\alpha w},\tld{w}(\tau)\tld{s_\alpha w}^{-1}\tld{w}_h^{-1}(0))} \in W^?(\rhobar^\speci,\tau)$ if $\tld{w}(\rhobar^\speci,\tau) \neq t_{w^{-1}(e\eta_0)}$.
This gives the reverse inclusion.
The multiplicity statement follows from that of Proposition \ref{prop:JH}.
\end{proof}

\begin{prop}\label{prop:indcomb}
Suppose we fix an $(e+2)h_{\eta_0}$-generic lowest alcove presentation of a tame inertial type $\rhobar^\speci$ over $\F$.
For $0\leq k \leq e$, let $\tau_{2k}$ be the tame inertial type over $E$ with compatible lowest alcove presentation such that $\tld{w}(\tau) = \tld{w}(\rhobar^\speci) t_{w^{-1}(k\alpha-e\eta_0)}$.
For $0\leq k\leq e-1$, let $\tau_{2k+1}$ be the tame inertial type over $E$ with compatible lowest alcove presentation such that $\tld{w}(\tau) = \tld{w}(\rhobar^\speci) \tld{w}^{-1} t_{k\alpha-e\eta_0}s_\alpha\tld{w}$.

For $0\leq k \leq e-1$, let $\sigma_{2k}$ be 
\[
F_{(\tld{w},\tld{w}(\rhobar^\speci)\tld{w}^{-1}(k\alpha-(e-1)\eta_0))},
\]
and let $\sigma_{2k+1}$ be
\[
F_{(\tld{s_\alpha w},\tld{w}(\rhobar^\speci)\tld{s_\alpha w}^{-1}((e-k-1)\alpha-(e-1)\eta_0))}.
\]
Then for $0\leq m \leq 2e$, $W^?(\rhobar^\speci,\tau_m) = \{\sigma_{m-1},\sigma_m\}$ (where $\sigma_\ell$ should be omitted for $\ell = -1$ or $2e$).
Moreover, $\sigma_{m-1}$ and $\sigma_m$ appear as Jordan--H\"older factors of $\ovl{\sigma}(\tau_m)$ with multiplicity one.
\end{prop}
\begin{proof}
This follows from Propositions \ref{prop:listshapes} and \ref{prop:indint} using (\ref{eqn:factor1}) and (\ref{eqn:factor2}) noting that $\tau_m$ is $2h_{\eta_0}$-generic for all $m$. 
\end{proof}

\subsubsection{The case of $\GL_2$: a comparison with Schein's recipe}
\label{subsub:schein}

\cite{schein} explicitly describes a set of Serre weights for a semisimple $\rhobar: G_K\rightarrow \GL_2(\F)$ with $K$ possibly ramified over $\Qp$ in terms of a ``reflection operation'' $\cR^\delta$ similar to $\cR$ above. 
We compare this description in generic cases with the set $W^?(\rhobar)$ defined in \S \ref{sec:SWC}.

Assume $e\leq p-1$ and let $\rhobar: G_K\rightarrow \GL_2(\F)$ be semisimple.
In \cite[Conjecture 1]{schein}, a set of Serre weights is associated to $\rhobar$, in terms of a reflection operation denoted as $\cR^\delta_{\fp}$ in \emph{loc.~cit}.
The superscript $\delta$ is an element in $\{0,\dots,e-1\}^\cJ$ and leads to the notion of $\delta$-regular weight:
\begin{defn}[\cite{schein}]
A Serre weight $F(\lambda)$ is \emph{$\delta$-regular} if $p-1-\langle\lambda,\alpha_j^\vee\rangle\in \{1,\dots,p\}+(2\delta_j-e+1)$ for all $j\in \cJ$.
(Note that this definition does not depend of the lift of $\lambda\in X_1(\un{T})/(p-\pi)X_0(\un{T})$.)
\end{defn}
A direct computation shows that if $\lambda\in \un{C}_0$ is $(e-1)$-deep then $\lambda$ is $\delta$-regular for any $\delta\in\{0,\dots,e-1\}^{\cJ}$, and moreover $\lambda+\nu\in\un{C}_0$ for any weight $\nu$ appearing in $\big(W((1-e)w_0(\eta_0))_{/\F}\big)|_{\un{T}}$.

Let now $\lambda\in\un{C}_0$ be $(e-1)$-deep.
A direct computation using the definition of $\cR^\delta_{\fp}$ for $\delta$-regular weights yields:
\[
\cR^\delta_{\fp}(F(\lambda))=F\big(\tld{w}_h\cdot (\lambda-(e-1)\eta_0+\sum_{j\in \cJ}\delta_j\alpha_j)\big)
\]
 and hence 
\[
\bigcup_{\delta\in \{0,\dots,e-1\}^{\cJ}}\cR^\delta_{\fp}(F(\lambda))=\cR\big(\JH(F(\lambda)\otimes W((1-e)w_0\eta_0))\big)
\]
by the translation principle (cf.~ \cite[Proposition 3.3]{LMS}).
From \cite[Proposition 2.15]{DL} (or Proposition \ref{prop:JH} above when $e\geq2$), noting that for an $e$-generic Deligne--Lusztig representation $R$ all the Serre weights $F\in \JH(R)$ are $e-1$-deep, we deduce
\begin{prop}
Let $\rhobar:G_K\rightarrow \GL_2(\F)$ be semisimple and $e$-generic.
Then the set of weights $W^?_{\fp}(\rhobar)$ defined in \cite[Conjecture 1(1)]{schein} coincides with the set $W^?(\rhobar)$  of Definition \ref{defn:SWC:ram} above.
\end{prop}

\clearpage{}%
\clearpage{}%
\section{Breuil--Kisin modules}
\label{sec:BKM}

\subsection{Moduli of Breuil--Kisin modules and local models}

In this section, we introduce background on Breuil--Kisin modules with tame descent.   We closely follow  \cite[\S 5]{MLM} making the necessary modifications to allow $K/\Qp$ to be ramified.  We will generally admit proofs as the generalizations are straightforward.   Throughout this section, we take $G=\GL_n$.   

Let  $K/\Qp$ be finite.
We let $K_0$ be the maximal unramified subextension of $K$, with $f \defeq [K_0:\mathbb{Q}_p]$ and $e \defeq [K:K_0]$. 
Let $k$ denote the residue field of $K$, of cardinality $p^f$ and which coincides with the residue field of $K_0$. 
Let $W(k)$ be ring of Witt vectors of $k$, which is also the ring of integers of $K_0$.   
We denote the arithmetic Frobenius automorphism on $W(k)$ by $\phz$, which acts as raising to $p$-th power on the residue field.
 
We fix a uniformizer $\pi_K\in \ovl{K}$ of $K$. 
Let $E(v) \in W(k)[v]$ be the minimal polynomial for $\pi_K$ over $K_0$, of degree $e$.

Let $\cJ_K = \Hom(K, E)$ and $\cJ = \Hom(k, \F) = \Hom(K_0, E)$.  
Recall that we have fixed an embedding $\sigma_0:K_0\into E$, hence an identification $\cJ\stackrel{\sim}{\rightarrow} \Z/f\Z$ given by $\sigma_j\defeq \sigma_0\circ\phz^{-j}\mapsto j$.

Let $\tau$ be a tame inertial type having a $1$-generic lowest alcove presentation $(s,\mu)\in W^\cJ\times X^*(T)^\cJ$, which we now fix throughout this section.
Recall from \cite[Example 2.4.1]{MLM} that we have a combinatorial data attached to $(s,\mu)$, in particular the element $s_\tau\in W$ (when $K=\Qp$, this is the \emph{niveau} of $\tau$).

Let $r$ be the order of $s_\tau$. 
We write $K'$ for the subfield of $\ovl{K}$ which is unramified of degree $r$ over $K$, $k'$ for its residue field, and $K'_0$ denote maximal unramified subextension of $K'$.  Set $\cJ_{K'} = \Hom(K', E)$, $\cJ'\defeq \Hom(k', \F)$.  
Let $f'\defeq fr$, $e' \defeq p^{f'}-1$ and fix an embedding $\sigma'_0:K'_0 \iarrow E$ which extends $\sigma_0:K_0 \iarrow E$, so that the identification $\cJ'\cong \Z/f' \Z$ given by $\sigma_{j'} \defeq \sigma'_0\circ \phz^{-j'} \mapsto j'$ induces the natural surjection $\Z/f' \Z \onto \Z/f \Z$ when considering the restriction of embedding from $K_0'$ to $K_0$.

We fix an $e'$-root $\pi_{K'}\in \ovl{K}$ of $\pi_K$ and set $L' \defeq K'(\pi_{K'})$.
Let $\Delta' \defeq \Gal(L'/K') \subset \Delta \defeq \Gal(L'/K)$. 
We set $\omega_{K'}(g) \defeq   \frac{g(\pi_{K'})}{\pi_{K'}}$ for $g \in \Delta'$; then $\omega_{K'}$ does not depend on the choice of $\pi'$.  
Composing with $\sigma'_j\in \cJ'$, we get a corresponding character $\omega_{K', \sigma'_j} :\Delta'\ra\cO^\times$ which will also be seen as a character of $I_{K'}=I_K$.
For $j'=0$ we set $\omega_{f'}\defeq \omega_{K', \sigma'_j}$.

Let $R$ be an $\cO$-algebra. Let $\fS_{L'} \defeq W(k')[\![u']\!]$ and $\fS_{L', R} \defeq (W(k') \otimes_{\Zp} R)[\![u']\!]$. 
As usual, $\varphi:\fS_{L', R} \ra \fS_{L', R}$ acts as Frobenius on $W(k')$, trivially on $R$, and sends $u'$ to $(u')^{p}$.  Note that for any $\fS_{L', R}$-module $\fM$, we have the standard $R[\![u']\!]$-linear decomposition $\fM\cong\bigoplus_{j' \in \cJ'}\fM^{(j')}$, induced by the maps $W(k')\otimes_{\Zp}R\rightarrow R$ defined by $x\otimes r\mapsto \sigma_{j'}(x)r$ for $j' \in \cJ'$.

We endow $\fS_{L', R}$ with an action of $\Delta$ as follows: for any $g$ in $\Delta'$, $g(u') = \omega_{K'}(g) u'$ and $g$ acts trivially on the coefficients; if $\sigma^f \in\Gal(L'/K)$ is the lift of the $p^f$-Frobenius on $W(k')$ which fixes $\pi_{K'}$, then $\sigma^f$ is a generator for $\Gal(K'/K)$, acting in natural way on $W(k')$ and trivially on both $u'$ and $R$. 
Set $v = (u')^{e'}$, and define
$
\fS_R \defeq (\fS_{L', R})^{\Delta = 1} = (W(k) \otimes_{\Zp} R)[\![v]\!].
$
Note that $E(v) = E((u')^{e'})$ is the minimal polynomial for $\pi_{K'}$ over $K_0$.    

We will make use of the following group schemes over $\cO$. 
For $j\in\cJ$ and for any $\cO$-algebra $R$, define  
\begin{align*}
L \cG^{(j)}(R)&\defeq \{A \in \GL_n(R[v]^{\wedge_{E_j}}[\frac{1}{E_j}])\};\\
L^+ \cG^{(j)}(R) &\defeq \{ A \in \GL_n(R[v]^{\wedge_{E_j}}), \text{ is upper triangular modulo $v$} \}
\end{align*}
where $E_j=\sigma_j(E(v))\in \cO[v]$, and $\wedge_{E_j}$ stands for the $E_j$-adic completion. In particular if $R$ is $p$-adically complete, this is the same as the $v$-adic completion of $R[v]$.

\subsubsection{Breuil--Kisin modules with tame descent} 
Let $R$ be a $p$-adically complete Noetherian $\cO$-algebra.
For any positive integer $h$, let $Y^{[0, h], \tau}(R)$ denote the groupoid of Breuil--Kisin modules of rank $n$ over $\fS_{L', R}$, height in $[0,h]$ and descent data of type $\tau$ (cf.~\cite[\S 3]{CL}, \cite[Definitions 5.1.1 and 5.1.3]{MLM}):
\begin{defn}
\label{defn:KM}
An object of $Y^{[0, h], \tau}(R)$ consists of 
\begin{itemize}
\item a finitely generated projective $\fS_{L', R}$-module $\fM$ which is locally free of rank $n$;
\item an injective $\fS_{L', R}$-linear map $\phi_\fM:\phz^*(\fM)\ra\fM$ whose cokernel is killed by $E(v)^h$;
\item a semilinear action of  $\Delta$ on $\fM$ which commutes with $\phi_{\fM}$, and such that Zariski locally on $R$, for each $j' \in \cJ'$, 
\[
\fM^{(j')} \mod u' \cong \tau^{\vee} \otimes_{\cO} R 
\]  
as $\Delta'$-representations. 
\end{itemize}
Morphisms are $\fS_{L', R}$-linear maps respecting all the above structures.  
\end{defn}

We will often omit the additional data and just write $\fM \in Y^{[0, h], \tau}(R)$ in what follows.
It is known that $ Y^{[0, h], \tau}$ is a $p$-adic formal algebraic stack over $\Spf \, \cO$ (see, for example, \cite[Theorem 4.7]{CL}).

Recall that an eigenbasis of $\fM \in Y^{[0,h], \tau}(R)$ is a collection of bases $\beta^{(j')}$ for each $\fM^{(j')}$ for $j' \in \cJ'$ compatible with the descent datum (see \cite[Definition 5.1.6]{MLM} for details).  
Given the lowest alcove presentation $(s, \mu)$ of $\tau$, and element $\fM \in Y^{[0, h], \tau}(R)$ and an eigenbasis $\beta$ of $\fM$, equation (5.4) in \cite{MLM} defines the matrix $A^{(j')}_{\fM,\beta} \in \Mat_n(\fS_R)$ for each $j' \in \cJ'$. 
We refer the reader to \emph{loc. cit.} for details rather than recall the excessive notation needed to make a precise definition.   We will recall the properties we need as we go along. 

First, the matrix  $A^{(j')}_{\fM, \beta}$ only depends on $j' \mod f$. Abusing notation, we occasionally write $A^{(j)}_{\fM, \beta}$ for $j \in \cJ$ with the obvious meaning.  Because $\tau$ is $1$-generic, the height condition is equivalent to $A^{(j)}_{\fM,\beta}$ and $(E_j)^h (A^{(j)}_{\fM, \beta})^{-1}$ both lying in $\Mat_n(R[\![v]\!])$ and being upper triangular modulo $v$, for all $j \in \cJ$. 

\begin{defn} \label{defn:smuconj}  \begin{enumerate}
\item  For integers $a \leq b$, define 
\[
L^{[a,b]}\cG^{(j)}(R) := \{ A \in L\cG^{(j)}(R) \mid E_j^{-a} A, E_j^b A^{-1} \in \Mat_n(R[\![v]\!]) \text{ and upper triangular mod } v\}.
\]
\item Given a pair $(s,\mu)\in W^\cJ\times X^*(T)^\cJ$, we define the $(s,\mu)$-twisted $\phz$-conjugation action of $\prod_{j \in \cJ} L^+\cG^{(j)}(R) $ on $\prod_{j \in \cJ} L^{[a,b]}\cG^{(j)}(R)$ by 
\begin{equation}
\label{eq:twisted:cnj}
(I^{(j)}) \cdot (A^{(j)}) =I^{(j)}A^{(j)}\big(\Ad(s^{-1}_jv^{\mu_j+\eta_{0,j}})\big(\phz(I^{(j-1)})^{-1}\big)\big).
\end{equation}
\end{enumerate}
\end{defn}
\begin{rmk}
\begin{enumerate}
\item
\label{rmk:cng:basis}
The change of basis formula in \cite[Proposition 5.1.8]{MLM} can be summarized as follows.  
For the fixed lowest alcove presentation $(s, \mu)\in W^\cJ\times X^*(T)^\cJ$ of $\tau$,
 the set of eigenbases of $\fM$ is a torsor for $\prod_{j \in \cJ} L^+\cG^{(j)}(R)$,
  and given two eigenbases $\beta$ and $\beta'$ differing by $(I^{(j)})_{j\in\cJ} \in \prod_{j \in \cJ} L^+\cG^{(j)}(R)$, the collections $(A^{(j)}_{\fM, \beta})$ and $(A^{(j)}_{\fM, \beta'})$ differ by $(s, \mu)$-twisted $\phz$-conjugation by $(I^{(j)})_{j\in\cJ}$. 
\item
Since eigenbases exist locally, we have the presentation 
\[
Y^{[0,h],\tau}_{\F}\cong \left[\prod_{j \in \cJ} L^{[0,h]}\cG^{(j)}/_{(s,\mu),\phz}\prod_{j \in \cJ}L^+\cG^{(j)}\right]
\] 
where the quotient is with respect to the twisted $\phz$-conjugation \eqref{eq:twisted:cnj}.
\item
 Let $\tld{w}^*(\tau) = s^{-1} t_{\mu + \eta}$.  A key observation which we use frequently is that $(s,\mu)$-twisted conjugation $\prod_{j \in \cJ} L^{[a,b]}\cG^{(j)}$  is the same as usual $\phz$-conjugation on the right translation  $\prod_{j \in \cJ} L^{[a,b]}\cG^{(j)} \tld{w}^*(\tau)$.    
\end{enumerate}
\end{rmk}

We now recall some useful results mod $p$.   
We write $\cI\defeq L^+\cG^{(j)}_{\F}$, which is the usual {(upper triangular)} Iwahori group scheme over Noetherian $\F$-algebras, in particular it is independent of the choice $j\in\cJ$. We also write $\cI_1 \subset \cI$ for pro-$v$ Iwahori consisting of upper unipotent matrices mod $v$.  
\begin{rmk}
In what follows, many of our proofs refer to results in \cite{MLM}, which is in the setting where $K$ is unramified.
However, just as in \emph{loc.~cit}., the proofs proceed by reduction to statements about characteristic $p$ loop groups, to which the results apply as written, thanks to the fact that $E_j \equiv v^{e} \mod p$ so that $L^{[a,b]}\cG^{(j)}_{\F}=L^{[ea,eb]}\GL_{n,\F}$.
For this reason, the requisite genericities in the ramified setting are usually scaled by $e$.
\end{rmk}

\begin{lemma}[Lemma 5.2.2 \cite{MLM}] \label{lem:Istraight} %
 \label{lem:iotawelldef}Let $R$ be an $\F$-algebra and $(A^{(j)}_1)_{j \in \cJ},  (A^{(j)}_2)_{j \in \cJ} \in  L^{[a,b]}\cG^{(j)}(R)$.  Let  $\tld{z} =s^{-1}t_{\mu+\eta}\in \tld{W}^{\vee,\cJ}$ where $\mu$ is $(e(b-a)+1)$-deep in $\un{C}_0$ and $s \in W^\cJ$.  Then, there is a bijection between the following:
\begin{enumerate}
\item
\label{lem:eq:mat:1}
Tuples $(I^{(j)})_{j \in \cJ} \in \cI_1(R)^{\cJ}$ such $A^{(j)}_2 \tld{z}_j = I^{(j)} A^{(j)}_1 \tld{z}_j \phz(I^{(j-1)})^{-1}$ for all $j \in \cJ$;
\item
\label{lem:eq:mat:2}
Tuples $(X_j)_{j \in \cJ} \in \cI_1(R)^{\cJ}$ such that 
$A^{(j)}_2=X_jA^{(j)}_1$ for all $j \in \cJ$.  
\end{enumerate} 
\end{lemma} 

\begin{rmk} \label{rmk:quotient} As in \cite[Corollary 5.2.3]{MLM}, Lemma \ref{lem:Istraight} gives a presentation of $Y^{[a,b], \tau}_{\F}$ as quotient of $\prod_{j \in \cJ} \cI_1 \backslash  L^{[a,b]}\cG^{(j)}_{\F}$ by $(s, \mu)$-twisted conjugation by the torus $T^{\vee, \cJ}_{\F}$ when $\mu$ is $(e(a-b) +1)$-deep.   
\end{rmk}

\begin{defn}
\label{defn:shape} Let $\F'/\F$ be finite extension. The \emph{shape} of a Breuil--Kisin module $\fM \in Y^{[0,h], \tau}(\F')$  with respect to $\tau$ is the element $\tld{z} = (\widetilde{z}_j)_{j \in \cJ} \in \widetilde{W}^{\vee, \cJ}$ such that for any eigenbasis $\beta$ and any $j\in\cJ$, the matrix $A^{(j)}_{\fM,\beta}$ lies in $\Iw(\F') \widetilde{z}_j \Iw(\F')$.
\end{defn}

\begin{prop} \label{prop:closurerelations}
For each $\tld{z} \in  \widetilde{W}^{\vee, \cJ}$ such that $\tld{z}_j \in L^{[0,h]}\cG^{(j)}(\F)$ for $j \in \cJ$, there is a locally closed substack  $Y^{[0,h], \tau}_{\F, \tld{z}} \defeq [ \prod_{j \in \cJ}\cI\tld{z}_j \cI /_{(s,\mu), \phz} \prod_{j \in \cJ}\cI] \subset Y^{[0,h], \tau}_{\F}$ whose $\F'$-points are the Breuil--Kisin modules of shape $\tld{z}$. %
 The closure of $Y^{[0,h], \tau}_{\F, \tld{z}}$ is contained in the union of the strata $Y^{[0,h], \tau}_{\F, \tld{z}'}$ such that $\tld{z}' \leq \tld{z}$ in the Bruhat order.  
\end{prop} 
\begin{proof}
We have maps
\[
\xymatrix{
Y^{[0,h],\tau}_{\F}\ar[rd] & &  \prod_{j \in \cJ} L^{[0,h]}\cG^{(j)}_{\F} / \cI \ar[ld] \\
& \prod_{j \in \cJ} [  \cI \backslash  L^{[0,h]}\cG^{(j)}_{\F} / \cI] &
}
\]
where the right arrow is an $\cI^{\cJ}$-torsor (cf.~\cite[Proposition 5.4]{CL}).
The substack $Y^{[0,h],\tau}_{\F,\tld{z}}$ is the preimage of $\prod_{j \in \cJ} [  \cI \backslash \cI \tld{z}_j \cI / \cI] $, so the stated closure relation follows from those on $\prod_{j \in \cJ} [  \cI \backslash  L\cG^{(j)}_{\F} / \cI] $, which in turn follow from those of $\prod_{j\in\cJ}\cI$-orbits in $\prod_{j \in \cJ} L\cG^{(j)}_{\F} / \cI$.
\end{proof}

We define the cocharacter $\eta\defeq((n-1,\dots,1,0),\cdots,(n-1,\dots,1,0))\in X_*(T^{\vee})^{\Hom(K, E)}$.
There is a closed $p$-adic formal algebraic stack $Y^{\leq \eta, \tau} \subset Y^{[0,n-1], \tau}$ defined in \cite[Theorem 5.3]{CL} \cite[\S 5.3]{MLM}. 
We recall the following result: 
\begin{prop}
\label{prop:CL:specialfiber} 
The special fiber of $Y^{\leq \eta, \tau}_{\F}$  satisfies
\[
Y^{\leq \eta, \tau}_{\F, \mathrm{red}} \subset \bigcup_{\tld{z} \in \Adm^{\vee}(e \eta_0)} Y^{\leq \eta, \tau}_{\F, \tld{z}}.
\]
\end{prop}  
\begin{proof} In the principal series case, this follows from \cite[Theorems 2.15 and 5.3]{CL}.  When $\tau$ is $n$-generic and $K/\Qp$ is unramified, this follows from \cite[Corollary 5.2.3 and Proposition 5.4.7]{MLM}. In general, we can reduce to the principal series case following the strategy in the proof of \cite[Theorem 3.2.20]{LLL}:  Let $\tau' = \tau|_{I_{K'}}$ where $K'$ is an unramified extension such that $\tau'$ is principal series.  Let $\cJ' = \Hom(k' , \F)$.   There is a natural map  $Y^{\leq \eta, \tau}_{\F, \mathrm{red}} \rightarrow Y^{\leq \eta, \tau'}_{\F, \mathrm{red}}$ and any strata $Y^{[0,h],\tau}_{\F,\tld{z}} \subset Y^{\leq \eta, \tau}_{\F, \mathrm{red}}$ clearly maps to $Y^{[0,h],\tau}_{\F,\tld{z}'}$ where $\tld{z}'_{j'} = \tld{z}_j$ where $j'$ restricts to $j$.  Thus, $\tld{z} \in \Adm^{\vee}(e \eta_0)$ by the principal series case. 
\end{proof}
\begin{rmk} In fact, the special fiber of $Y^{\leq \eta, \tau}$ is reduced and the inclusion in Proposition \ref{prop:CL:specialfiber} is an equality.  This is shown in the principal series case in \cite{CL}. 
(See also Remark \ref{rmk:twisted:LM} and the discussion preceding it.)
\end{rmk}

\subsubsection{\'Etale $\phz$-modules}

Let  $\cO_{\cE,K}$ (resp.~$\cO_{\cE,L'}$) be the $p$-adic completion of $(W(k)[\![v]\!])[1/v]$ (resp.~of $(W(k')[\![u']\!])[1/u']$). 
It is endowed with a continuous Frobenius morphism $\phz$ extending the Frobenius on $W(k)$ (resp.~on $W(k')$) and such that $\phz(v)=v^p$ (resp.~$\phz(u')=(u')^p$).
Let $R$ be a $p$-adically complete Noetherian $\cO$-algebra.
We then have the groupoid
$\Phi\text{-}\Mod^{\text{\'et}, n}_K(R)$ (resp.~ $\Phi\text{-}\Mod^{\text{\'et}, n}_{dd,L'}(R)$) of \'etale $(\phz,\cO_{\cE,K}\widehat{\otimes}_{\Zp}R)$-modules (resp.~\'etale $(\phz,\cO_{\cE,L'}\widehat{\otimes}_{\Zp}R)$-modules with descent data from $L'$ to $K$).
Its objects are rank $n$ projective modules $\cM$ over $\cO_{\cE,K}\widehat{\otimes}_{\Zp}R$ (resp.~$\cO_{\cE,L'}\widehat{\otimes}_{\Zp}R)$), endowed with a Frobenius semilinear endomorphism $\phi_{\cM}:\cM\ra\cM$ (resp.~a Frobenius semilinear endomorphism $\phi_{\cM}:\cM\ra\cM$ commuting with the descent data) inducing an isomorphism on the pull-back: $\Id\otimes_{\phz}\phi_{\cM}:\phz^*(\cM)\stackrel{\sim}{\longrightarrow}\cM$. It is known that $\Phi\text{-}\Mod^{\text{\'et}, n}_K(R)$ and  $\Phi\text{-}\Mod^{\text{\'et}, n}_{dd,L'}(R)$
form fppf stacks over $\Spf\,\cO$
(see \cite[\S 3.1]{EGstack}, \cite[\S 5.2]{EGschemetheoretic}, \cite[\S 3.1]{CEGS} where they are denoted $\mathcal{R}_{n},  \mathcal{R}^{dd}_{n, L'}$ respectively). 
We use $\Phi\text{-}\Mod^{\text{\'et}}_K(R)$ (resp.~$\Phi\text{-}\Mod^{\text{\'et}}_{dd,L'}(R)$) to denote the category of \'etale $(\phz,\cO_{\cE,K}\widehat{\otimes}_{\Zp}R)$-modules (resp.~$\cO_{\cE,L'}\widehat{\otimes}_{\Zp}R)$-modules with descent from $L'$ to $K$) of arbitrary finite rank.  

Given $\fM\in Y^{[0,h],\tau}(R)$, $\fM \otimes_{\fS_{L',R}} (\cO_{\cE,L'}\widehat{\otimes}_{\Zp}R)$
is naturally an object of $\Phi\text{-}\Mod^{\text{\'et},n}_{dd,L'}(R)$, and we define an \'etale $\phz$-module $\cM \in \Phi\text{-}\Mod^{\text{\'et},n}_{K}(R)$ by 
\[
\cM \defeq (\fM \otimes_{\fS_{L',R}} (\cO_{\cE,L'}\widehat{\otimes}_{\Zp}R))^{\Delta=1}
\]
with the induced Frobenius.  
This construction defines a morphism of stacks $\eps_\tau: Y^{[0,h],\tau}\ra\Phi\text{-}\Mod^{\text{\'et},n}_{K}$ which is representable by algebraic spaces, proper, and of finite presentation (see \cite[Proposition 5.4.1]{MLM}, which carries through in our ramified setting). 
Note that $\eps_\tau$ is independent of any presentation of $\tau$.

For any $(\cM,\phi_\cM)\in \Phi\text{-}\Mod^{\text{\'et}}_{K}(R)$, we decompose $\cM=\oplus_{j \in \cJ} \cM^{(j)}$ over the embeddings $\sigma_j: W(k)\ra\cO$, with induced maps $\phi_\cM^{(j)}:\cM^{(j-1)}\ra\cM^{(j)}$. 
We can define the map $\eps_{\tau}$ explicitly in some cases: 
\begin{prop}$($\cite[Proposition 5.4.2]{MLM}$)$ 
\label{prop:expeps}  
Let $\fM \in Y^{[0, h], \tau}(R)$ and set $\cM = \eps_{\tau}(\fM)$.  
Let $(s, \mu)$ be the fixed lowest alcove presentation of $\tau$. 
 If $\beta$ is an eigenbasis of $\fM$, then there exists a basis $\fF$ \emph{(}constructed from $\beta$\emph{)} for $\cM$ such that the matrix of $\phi_{\cM}^{(j)}$ with respect to $\fF$ is given by 
\[
A^{(j)}_{\fM, \beta} s^{-1}_j v^{\mu_j + \eta_{0,j}}.
\]
{(Note that this is $A^{(j)}_{\fM, \beta} \tld{w}^*(\tau)_j$.)}
\end{prop}

Finally, we recall that in generic situations the map $\eps_{\tau}$ does not lose information:
\begin{prop} $($\cite[Proposition 5.4.3]{MLM}$)$ \label{prop:BK_to_phi_mono} Assume $\tau$ is $(eh+1)$-generic. Then the map $\eps_\tau: Y^{[0,h],\tau}\to \Phi\text{-}\Mod^{\text{\emph{\'et}},n}_{K}$ is a closed immersion.
\end{prop}
\begin{proof}
As in the proof of \cite[Proposition 5.4.3]{MLM} this reduces to prove that whenever we have a relation
\[
I^{(j)}A_1^{(j)}=A_2^{(j)}\Ad(s_j^{-1}v^{\mu_j+\eta_j})(\phz(I^{(j-1)}))
\]
where $A^{(j)}_1,\, A^{(j)}_2\in L^{[0,h]}\cG^{(j)}(R)$, and $I^{(j)}\in L\cG^{(j)}(R)$,
then $I^{(j)}$ are in $L^+\cG^{(j)}(R)$.
Just as in \emph{loc.~cit}., this reduces to checking the statement and its infinitesimal version over $\F$.
In turn, these two statements are exactly \cite[Lemmas 5.4.4, 5.4.5]{MLM}, with the adjustment that $L^{[a,b]}\cG^{(j)}_{\F}=L^{[ea,eb]}\GL_{n,\F}$.
\end{proof}

We briefly recall the relations between Breuil--Kisin modules and Galois representations.
Recall from \ref{sec:not:GT} the extension $K_{\infty}/K$, and let $G_{K_{\infty}} \subset G_K$ denote the absolute Galois group of $K_{\infty}$.  
{Let $R$ be a complete Noetherian local $\cO$-algebra with residue field $\F$.}
We have an anti-equivalence of categories established by the exact functor %
\begin{align*}
\bV^*_K:\Phi\text{-}\Mod^{\text{\'et},n}_{K}(R)&\ra\Rep^n_R(G_{K_\infty})
\end{align*}
defined through the theory of fields of norms (cf.~\cite[\S 2.3 and \S 6.1]{LLLM} for details) and therefore a functor $T^*_{dd}: Y^{[0,h],\tau}(R)\rightarrow \Rep^n_R(G_{K_\infty})$ defined as the composite of $\eps_{\tau}$ followed by $\bV^*_K$.

We finally recall from \cite[\S 5.5]{MLM} the notion of \emph{shape} of an $n$-dimensional $\F$-representation of $G_K$ (or $G_{K_{\infty}}$) with respect to $\tau$.
\begin{defn} \label{defn:shaperhobar} Assume that $\tau$ is $(e(n-1)+1)$-generic. %
Let $\rhobar$ be an $n$-dimensional $\F$-representation of $G_K$ or $G_{K_{\infty}}$.  
If there exists $\ovl{\fM} \in Y_\F^{\leq\eta, \tau}(\F)
$ such that $T^*_{dd}(\overline{\fM}) \cong \rhobar|_{G_{K_{\infty}}}$ then we say that $\rhobar$ is \emph{$\tau$-admissible}, and we define $\tld{w}(\rhobar, \tau)\in \Adm(e \eta_0)$ to be the shape of $\overline{\fM}$ with respect to $\tau$ (Definition \ref{defn:shape}).  
This is well-defined by {Propositions \ref{prop:CL:specialfiber} and} \ref{prop:BK_to_phi_mono}.
\end{defn}

\begin{prop} $($\cite[Proposition 5.5.7]{MLM}$)$ \label{prop:sskisin}  
Assume that the fixed lowest alcove presentation $(s,\mu)$ of $\tau$ is $(e(n-1)+1)$-generic.
Let $\rhobar$ be a semisimple representation of $G_{K}$ over $\F$.  
Then $\rhobar$ is $\tau$-admissible if and only if $\rhobar|_{I_K}$ admits a lowest alcove presentation $(w, \nu)$ compatible with the lowest alcove presentation of $\tau$ such that $s^{-1} t_{\nu - \mu} w  \in \Adm(e\eta_0)$.   
Furthermore, if $\rhobar$ is $\tau$-admissible then $\tld{w}(\rhobar, \tau) = s^{-1} t_{\nu - \mu} w $.   
\end{prop}
\begin{prop} \label{prop:modpform}  Assume that the fixed lowest alcove presentation $(s,\mu)$ of $\tau$ is $(eh+1)$-generic.   Let $\fM \in Y^{[0,h], \tau}(\F)$ with shape $\tld{z} \in \tld{W}^{\cJ}$.  
Then, there exists an eigenbasis $\beta$ for $\fM$, unique up to scaling by $T(\F)^{\cJ}$, such that 
\[
A^{(j)}_{\fM, \beta} \in T(\F) \tld{z}_j N_{\tld{z}_j} (\F)
\]
where $N_{\tld{z}_j}$ is unipotent subgroup scheme of $\cI$ defined in \cite[Definition 4.2.9]{MLM}.  
\end{prop}%
\begin{proof}
This follows from \cite[Proposition 4.2.13 and Corollary 5.2.3]{MLM}.   
\end{proof}

\subsection{Mod $p$ monodromy}  
\label{sec:Modp:mon}

Let $\Fl := \cI \backslash L\cG_{\F}$ denote the affine flag variety over $\F$ where $L \cG_{\F} = L \cG^{(j)}_{\F}$ for any $j$ denotes the usual loop group.  
Given $\tld{w}\in \tld{W}^\vee$, we write $S^\circ_\F(\tld{w})$ for the affine open Schubert cell associated to $\tld{w}$.
Let $\bf{a} \in (\cO^n)^{\cJ}$.
{We have a} closed subfunctor $L \cG_{\F}^{\nabla _{\bf{a}}}\subset L \cG_{\F}^{\cJ}$ {defined on an $\F$-algebra $R$} by 
 \begin{equation} \label{eq:nablaa}
L \cG_{\F}^{\nabla _{\bf{a}}}(R)\defeq \left\{ (A^{(j)})\in L \cG_{\F}^{\cJ} (R) \mid v \frac{dA^{(j)}}{dv} (A^{(j)})^{-1}  + A^{(j)} \Diag(\bf{a}_j) (A^{(j)})^{-1} \in \frac{1}{v^e} \Lie \Iw (R) \text{ for all } j \in \cJ \right\}.
\end{equation}
This condition defines a closed sub-ind-scheme $\Fl^{\nabla_{\bf{a}}}_{\cJ} \subset \Fl^{\cJ}$.  
For any subset $S \subset L \cG_{\F}^{\cJ}(R)$, 
we set $S^{\nabla_{\bf{a}}} := S \cap L \cG_{\F}^{\nabla _{\bf{a}}}(R)$; similarly for any subscheme $X \subset \Fl^{\cJ}$, set $X^{\nabla_{\bf{a}}} \defeq X \cap \Fl^{\nabla_{\bf{a}}}_{\cJ}$.   

Following \cite[Definition 4.2.2]{MLM}, given an integer $m\geq 0$, we say that an element $(a_1,\dots,a_n)\in R^n$ is  \emph{$m$-generic} if $a_i-a_k-\ell\in R^\times$ for all $\ell\in\{-m,-m+1,\dots,m-1,m\}$ for all $i\neq k$.

\begin{prop} \label{prop:monodromySchubert} Let $h$ be a positive integer.   Let $\tld{w} \in \tld{W}^{\cJ}$ and $\bf{a}=(\bf{a}_j)_{j\in \cJ} \in (\cO^n)^{\cJ}$.  Assume that $\tld{w}$ is $e$-regular and $h$-small $($see Definitions \ref{defn:regular} and \ref{defn:var:gen}$($\ref{defn:small}$))$ and that $\bf{a}_j \mod \varpi \in \F^n$ is $h$-generic for all $j \in \cJ$. Then the intersection $S^\circ_{\F}(\tld{w}^*) \cap \Fl^{\nabla_{\bf{a}}}$ is an affine space of dimension  $[K:\Qp] \dim (B \backslash \GL_n)_{\F}$.  
\end{prop} 
\begin{proof}
This is a direct generalization of  \cite[Theorem 4.2.4]{MLM} to the ramified setting.  We only briefly outline the proof.  It suffices to consider the case when $\# \cJ = 1$.    By \cite[Proposition 4.2.13]{MLM}, there is an isomorphism $\tld{w}^* N_{\tld{w}^*} \cong S^\circ(\tld{w}^*)$ where $N_{\tld{w}^*}$ is a unipotent subgroup scheme of $\Iw$ isomorphic to an affine space of dimension $\ell(\tld{w})$.    
As $\tld{w}$ is $e$-regular, for each $\alpha$ in the support of $N_{\tld{w}^*}$, we have $N_{\tld{w}^*, \alpha} = v^{\delta_{\alpha < 0}} f_{\alpha}$ where $f_{\alpha}$ is a polynomial of degree at least $e -1$ (cf.~\cite[Corollary 4.2.5]{MLM}, and note that, more precisely, the degree is $\lfloor \langle \tld{w}(x),-\alpha^\vee\rangle\rfloor-\delta_{\alpha<0}$, which is at least  $e -1$ by the $e$-regularity condition). 
Condition (\ref{eq:nablaa}) does not impose any constraint on the coefficients of degree $\deg(f_\alpha),\deg(f_\alpha)-1,\dots,\deg(f_\alpha)-(e-1)$ of $f_\alpha$, while the coefficients of degree strictly smaller than $\deg(f_\alpha)-(e-1)$ are solved in terms of the coefficients of the polynomials $f_{\alpha'}$ with $\alpha'\!<_{\cC}\!\alpha$ for a partial order $<_{\cC}$ on $\Phi$ determined by $\tld{w}$ (cf.~equation (4.6) in \emph{loc.~cit.}~).
Hence, $(\tld{w}^* N_{\tld{w}^*})^{\nabla_{\bf{a}}}$ is an affine space of dimension $e \dim (B \backslash \GL_n)_{\F}$.
\end{proof}

 Let $\tld{z} = s^{-1} t_{\mu} \in \tld{W}^{\vee, \cJ}$ acting by right translation on $\Fl^{\cJ}$.   
Let $\bf{a} \in (\Z^n)^{\cJ}$ and assume that $\bf{a}_j \equiv s^{-1}_j(\mu_j)$ mod $p$ for all $j \in \cJ$.
An easy calculation shows that:
\begin{equation*} \label{eq:translation}
L \cG_{\F}^{\cJ}  \tld{z} \cap L \cG_{\F}^{\nabla _0} = L \cG_{\F}^{\nabla _{\bf{a}}} \tld{z},   \quad \Fl^{\cJ} \tld{z} \cap \Fl^{\nabla_0}_{\cJ} =   \Fl^{\nabla_{\bf{a}}}_{\cJ} \tld{z} 
\end{equation*}

We can now state the main result of the section which is the ramified analogue of \cite[Proposition 4.3.4]{MLM}:
\begin{prop} \label{prop:nosematch}  Let $\tld{w}, \tld{w}' \in \tld{W}^{\cJ}$ be $h$-small, $e$-regular elements such that $\tld{w}' \leq \tld{w}$.  Write $\tld{w}^* = (\tld{w}')^* \tld{z}'$ and assume this is a reduced expression for $\tld{w}^*$.    Let $\tld{z} \in \tld{W}^{\vee, \cJ}$ be $2h$-generic.   Then
$$
(\cI(\F) (\tld{w}'_j)^* \cI(\F) \tld{z}'_j \tld{z}_j)^{\nabla_0} = (\cI(\F) \tld{w}^*_j \cI(\F)  \tld{z}_j)^{\nabla_0}
$$
 for all $j \in \cJ$. 
\end{prop}
\begin{proof}
Again, the proof is very similar to the proof of \cite[Proposition 4.3.4]{MLM}, and we refer the reader to \emph{loc.~cit}.~for further detail.

Since   $(\tld{w}')^* \tld{z} = \tld{w}^*$ is a reduced expression, there is an inclusion of the left side in the right side.    Since both sides are invariant under $\cI(\F)$, we can descend to $\Fl^{\nabla_0}_{\cJ}$ and reduce to showing 
\[
(S^{\circ} (\tld{w}')^*  \tld{z}' \tld{z})^{\nabla_0}  = (S^{\circ}(\tld{w}^*)  \tld{z})^{\nabla_0}.
\]   By the assumptions, both $\tld{z}$ and $\tld{z}' \tld{z}$ are $h$-generic and so by Proposition \ref{prop:monodromySchubert}, both sides are affine spaces of the same dimension and so inclusion implies equality. 
\end{proof}

\begin{defn}  \label{defn:modpmono} Assume that the lowest alcove presentation $(s,\mu)$ of $\tau$ is $(eh +1)$-generic. 
We say that $\fM \in Y^{[0, h], \tau}(\F)$ satisfies the \emph{mod $p$ monodromy condition}
 if for any choice of eigenbasis $\beta$ of $\fM$, the collection $(A^{(j)}_{\fM, \beta} \tld{w}^*(\tau)_j)$ is in $ L \cG_{\F}^{\nabla _0}(\F)$.  %
\end{defn}

\subsection{Semicontinuity I}
\label{subsec:SCI}

We fix a tame inertial type $\tau$ with a $1$-generic lowest alcove presentation $(s, \mu)$, as defined in \S \ref{sec:SWC}.  In this section, we show a semicontinuity result for the shape of a mod $p$ Kisin module of type $\tau$ with respect to the shape of its semisimplification.   
This is preliminary to a more general semicontinuity result (Theorem \ref{thm:semicont} in section \ref{subsec:SCII}).

\begin{prop}    
\label{prop:SC:I}
Let $\fM \in Y^{[0, h], \tau}_{\F} (\F')$ and set $\rhobar := T^*_{dd}(\fM)$ for any finite extension $\F'/\F$.    There exists $\fM_0 \in  Y^{[0, h], \tau}_{\F}(\F')$ such that 
\[
T^*_{dd}(\fM_0) = \rhobar^{\mathrm{ss}}.
\]
Furthermore, the shape of $\fM_0$ with respect to $\tau$ is less than or equal to the shape of $\fM$ with respect to $\tau$ in the Bruhat order on $\tld{W}^{\vee, \cJ}$. 
\end{prop} 
\begin{proof}
By the closure relations for the stratum of the stack $Y_\F^{[0,h], \tau}$ (Proposition \ref{prop:closurerelations}), it suffices to construct a map $\mathbb{A}^1_{\F} \ra  Y_\F^{[0,h],\tau}$, sending $x$ to $\fM_{x}$, 
such that
\begin{enumerate}
	\item\label{secondo:sc1} $T^*_{dd}(\fM_{0})\cong \rhobar^{\semis}$; and
	\item\label{primo:sc1} for all $x \in \overline{\F}, x \neq 0$,  $T^*_{dd}(\fM_{x})\cong \rhobar$.
\end{enumerate}

The construction of the map proceeds as in the proof of \cite[Proposition 5.5.9]{MLM}.
Let $\alpha$ be the eigenbasis for $\fM$ constructed in \emph{loc. cit.} adapted to the filtration $(\cM_i)$ on the \'etale $\varphi$-module $\fM[1/u']$.  
Define the matrix $C^{(j)}\in G(\F'(\!(u')\!))$ by the condition
\[
\phi_{\fM}^{(j)}(\phz^*(\alpha^{(j)})) = \alpha^{(j+1)} C^{(j)}.
\]
By construction, $C^{(j)}$ lies in a parabolic subgroup $P(\F'(\!(u')\!)) \subset G(\F'(\!(u')\!))$ corresponding to the filtration $(\cM_i)$.   
Let $L$ denote the corresponding Levi subgroup which contains the diagonal torus $T$.   
Choose a dominant cocharacter $\lambda$ such that $L$ is the centralizer of $\lambda$.  

For $x \neq 0$, define $\fM_x$ to be the free Breuil--Kisin module of rank $n$ with basis $\alpha_x$ such that $\Delta$ acts on $\alpha_x$ in the same way it acts on $\alpha$ and such that the Frobenius acts by $C^{(j)}_x = \lambda(x) C^{(j)} \lambda(x)^{-1}$ (with respect to $\alpha_x$).  
Observe that the limit of $C^{(j)}_x$ as $x \ra 0$ exists and lies in the Levi subgroup $L(\F'(\!(u')\!))$.  
Thus, we can extend this to a family over $\bA^1_{\F}$.   It is easy to check property (\ref{secondo:sc1}).   For property (\ref{primo:sc1}), we note that for any $x \neq 0$,  $C^{(j)}_x$ is  the matrix for Frobenius with respect to the basis $(\alpha^{(j)} \cdot \lambda(x))$ and so $\fM_x$ is isomorphic to $\fM_1$.  
\end{proof}

\begin{cor}
\label{prop:weak:sc}
Assume that $\tau$ is $(e(n-1)+1)$-generic.   If $\rhobar$ is $\tau$-admissible, then $\rhobar^{\mathrm{ss}}$ is $\tau$-admissible and for all $j \in \cJ$, 
\[
\tld{w}(\rhobar^{\semis},\tau)_j \leq \tld{w}(\rhobar, \tau)_j.
\]
\end{cor}

\subsection{Specializations} 
\label{subsec:SPEC}

Throughout this section we consider a continuous Galois representation $\rhobar: G_K\ra\GL_n(\F)$.
We say that $\rhobar$ is $N$-generic if the tame inertial $\F$-type $\rhobar^{\semis}|_{I_K}$ is $N$-generic (see \S \ref{subsubsec:TIT}). \emph{All lowest alcove presentations for tame inertial types (over $\F$ or over $E$) will always be compatible with a given lowest alcove presentation for $\rhobar^{\semis}$.}

If $\rhobar^{\speci}$ is a tame inertial $\F$-type for $K$ and $\tau$ is an inertial type over $E$ with compatible lowest alcove presentation, then recall the combinatorially defined shape $\tld{w}(\rhobar^{\speci}, \tau)  =  \tld{w}(\tau)^{-1}\tld{w}(\rhobar^{\speci}) \in \tld{W}^{\cJ}$ defined in \S \ref{sec:combserre}.

\begin{defn}\label{defn:spec}
Let $\rhobar:G_K \ra \GL_n(\F)$ be a continuous Galois representation.   
\begin{enumerate}
\item A tame inertial $\F$-type $\rhobar^\speci$ for $K$ over $\F$ is a \emph{specialization} of $\rhobar$ if there exists an $(e(n-1)+1)$-generic tame inertial type $\tau$ such that $\rhobar$ is $\tau$-admissible and $\tld{w}(\rhobar,\tau) = \tld{w}(\rhobar^\speci,\tau)$.
We say that $\tau$ \emph{exhibits} the specialization.
\item 
A specialization $\rhobar^\speci$ is called an \emph{extremal specialization} of $\rhobar$ if there exists a $\tau$ exhibiting the specialization such that $\tld{w}(\rhobar^\speci,\tau) = t_{w^{-1}(e\eta_0)}$ for some $w \in W^{\cJ}$ and if the unique $\fM \in Y^{\leq \eta, \tau}(\F)$ such that $T^*_{dd}(\fM) \cong \rhobar|_{G_{K_{\infty}}}$ satisfies the mod $p$ monodromy condition (Definition \ref{defn:modpmono}).
\end{enumerate}
\end{defn}

\begin{rmk} \label{rmk:liftsandmodpmono}  By a version of \cite[Proposition 7.4.1]{MLM} in the ramified setting (based on the analysis of the monodromy condition in characteristic 0, cf.~Proposition \ref{prop:monodromy_control} below), if $\tau$ is $(e+1)(n-1)$-generic 
and $\rhobar$ admits a potentially crystalline lift of type $(\tau, \eta)$ then the unique $\fM \in Y^{\leq \eta, \tau}(\F)$ such that $T^*_{dd}(\fM) \cong \rhobar|_{G_{K_{\infty}}}$ satisfies the mod $p$ monodromy condition.  Thus, the technical condition in Definition \ref{defn:spec}(2) could be replaced by the existence of a potentially crystalline lift.
\end{rmk}

\begin{rmk}  Using the methods of \cite{MLM}, it can be shown under suitable genericity conditions that all specialization are extremal when $K/\Qp$ is unramified.    It is likely that the same is true in the ramified case but we do not attempt to prove it here. 
\end{rmk}

Let $S(\rhobar)$ (resp.~$S_{\mathrm{ext}}(\rhobar)$) denote  be the set of specializations (resp.~extremal specializations) of $\rhobar$.  

\begin{rmk} \label{lem:Sfinite}
The sets $S(\rhobar)$ and $S_{\mathrm{ext}}(\rhobar)$ are finite because the set of $(e(n-1)+1)$-generic types $\tau$ for which $\rhobar^{\semis}$ is $\tau$-admissible is finite by Proposition \ref{prop:sskisin} and the set of types for which $\rhobar$ is $\tau$-admissible is a subset of this set by Corollary \ref{prop:weak:sc}. In Theorem \ref{thm:extension}, we show that $\rhobar^{\semis}|_{I_K}$ is an extremal specialization of $\rhobar$ so $S_{\mathrm{ext}}(\rhobar)$ is also non-empty.
\end{rmk}

\begin{exam} 
We have the following examples when $K=\Qp$ and $n=2,3$.
\begin{enumerate}
\item Let  $1 < a < p-1$ and assume that $\rhobar_{I_{\Qp}}$ is of the form $\begin{pmatrix}
\omega^a & \ast \\
0 & 1
\end{pmatrix}$, where $\ast\neq 0$.
Then we have two specializations, given by $\omega^{a}\oplus 1$ and $\omega_2^{a}\oplus\omega_2^{pa}$.
A type which exhibits the specialization $\omega_2^{a}\oplus\omega_2^{pa}$ is $(\omega_2^{a+1} \oplus \omega_2^{p(a+1)})\otimes \omega^{-1}$.
\item (\cite[Theorem 4.2.5]{GL3Wild}) Assume that $(a+1,b+1,c+1)\in \Z^3$ is $6$-deep in alcove $C_0$ and that $\rhobar|_{I_K}$ is of the form 
\[
\begin{pmatrix}
\omega^{a} & *_1 & * \\
0 & \omega^b & *_2\\
0 & 0 & \omega^c 
\end{pmatrix}
\]
where $\ast_1$, $\ast_2$ denote non-split extensions.
Then $\rhobar$ has up to $6$ specializations, namely $\omega^{a}\oplus\omega^{b}\oplus\omega^{c}$, $\omega_2^{a+pb}\oplus\omega_2^{b+pa}\oplus\omega^{c}$, $\omega_2^{a+pc}\oplus\omega^{b}\oplus\omega_2^{c+pa}$, $\omega^{a}\oplus\omega_2^{b+pc}\oplus\omega_2^{c+pb}$, $\omega_3^{a+pb+p^2c}\oplus \omega_3^{b+pc+p^2a}\oplus \omega_3^{c+pa+p^2b}$ and $\omega_3^{a+pc+p^2b}\oplus \omega_3^{b+pa+p^2c}\oplus \omega_3^{c+pb+p^2a}$.
A type which exhibits the specialization $\rhobar^{\semis}$ is $\omega_2^{(a-1)+p(c-1)}\oplus\omega^{b-1}\oplus\omega_2^{(c-1)+p(a-1)}$.
We have $4$ specializations precisely when $\rhobar$ has either a potentially crystalline lift of type $\omega_{a-1}\oplus\omega_2^{b+p(c-2)}\oplus \omega_2^{(c-2)+pb}$ (in which case the specializations $\omega_2^{a+pc}\oplus\omega^{b}\oplus\omega_2^{c+pa}$ and $\omega_3^{a+pc+p^2b}\oplus \omega_3^{b+pa+p^2c}\oplus \omega_3^{c+pb+p^2a}$ do not appear) or of type $\omega_2^{a+p(b-2)}\oplus \omega_2^{(b-2)+pa}\oplus\omega^{c-1}$ (in which case the specializations $\omega_2^{a+pc}\oplus\omega^{b}\oplus\omega_2^{c+pa}$ and $\omega_3^{a+pb+p^2c}\oplus \omega_3^{b+pc+p^2a}\oplus \omega_3^{c+pa+p^2b}$ do not appear).
\end{enumerate}
\end{exam}

\subsection{Semicontinuity II} \label{subsec:SCII}

The following theorem generalizes Proposition \ref{prop:SC:I}.

\begin{thm}
\label{semi:cnt} \label{thm:semicont}
Let $\rhobar:G_K\ra\GL_n(\F)$ be a $3e(n-1)$-generic continuous Galois representation.
Assume that $\rhobar$ specializes to a tame inertial $\F$-type $\rhobar^{\speci}$ for $K$ and that $\rhobar$ is $\tau$-admissible.  
For each $j\in \cJ$, we have the inequality
\[
\tld{w}(\rhobar^{\speci},\tau)_j  \leq \tld{w}(\rhobar, \tau)_j .
\]
\end{thm}

We begin by stating two combinatorial lemmas which will be needed in the proof of Theorem \ref{thm:semicont}.

\begin{lemma} \label{lem:gap} Let $\tau$ and $\tau'$ be $(e(n-1)+1)$-generic tame inertial types over $E$.  Assume there exists a $\rhobar$ which is both $\tau$ and $\tau'$-admissible.  Then, for any choice of lowest alcove presentation of $\rhobar^{\semis}$, $\tau$ and $\tau'$ admit lowest alcove presentations $(s, \mu)$ and $(s', \mu')$, compatible with that of $\rhobar^{\semis}$,
such that
\[
|\mu_{j,i}-\mu'_{j,i}|\leq e(n-1).
\]
\end{lemma} 
\begin{proof}
Since $\rhobar$ is both $\tau$ and $\tau'$-admissible, the same is true for $\rhobar^{\semis}$ by Proposition \ref{prop:weak:sc}. Fixing a lowest alcove presentation of $\rhobar^{\semis}$, $\tau$ and $\tau'$ admit compatible presentation $(s, \mu)$ and $(s', \mu')$ respectively
and we have $e\eta_0$-admissible elements $\tld{w}(\rhobar^{\semis}, \tau) = t_\nu w$ and $\tld{w}(\rhobar^{\semis}, \tau') =t_{\nu'} w'$. Since $\nu$ and $\nu'$ are in the convex hull of $\un{W}e\eta_0$ (cf.~\cite[Theorem 3.3]{HC}), 
\begin{equation}\label{eqn:nu}
0 \leq \nu_{j,i}, \nu'_{j, i} \leq e(n-1).
\end{equation}

By Proposition \ref{prop:sskisin}, $\rhobar^{\semis}|_{I_K}$ has lowest alcove presentation $(sw, \mu + s(\nu)) = (s'w', \mu' + s'(\nu'))$. %
The result now follows from this equation and (\ref{eqn:nu}).
\end{proof}

\begin{lemma} \label{lem1}  
Let $\mu, \mu'$ be dominant cocharacters which are $2e(n-1)$-deep in alcove $\un{C}_0$, and assume that for all $j \in \cJ$, $1\leq i\leq n$, 
\[
|\mu_{j,i}-\mu'_{j,i}|\leq e(n-1).
\]
Let $(B_j), (B_j') \in \Mat_n(\F[\![v]\!])^{\cJ}$ such that for all $j$, $v^{e(n-1)} B_j^{-1} \in \Mat_n(\F[\![v]\!])$. Assume that for all $j \in \cJ$ there exists $C_j \in \GL_n(\F(\!(v)\!))$ such that 
\begin{equation}
\label{eq:comparison}
C_{j} B_j' v^{\mu'_j + \eta_{0,j}}= B_j v^{\mu_j + \eta_{0,j}} \phz(C_{j-1}).
\end{equation}
Then $C_j \in \Iw(\F)$ for all $j \in \cJ$.    
\end{lemma}
\begin{proof} 
The technique is similar to the proof of \cite[Theorem 3.2]{LLLM} and \cite[Lemma 5.4.5]{MLM}. 
We first show that for all $j\in\cJ$ we have $C_j \in \Mat_n(\F[\![v]\!])$. 
For all $j\in\cJ$, write $C_j =v^{-k_{j}}C^{+}_j$ with  $k_{j}\in\Z$, $C^{+}_j \in\Mat_n(\F[\![v]\!])$ and $C^{+}_j \not\equiv 0$ modulo $v$.
Rearranging equation (\ref{eq:comparison}), we can write:
\begin{equation}
\label{iso-phi:1}
v^{-pk_{j-1}}\Ad(v^{\mu_j + \eta_{0,j}})\phz(C^{+}_{j-1})= v^{-k_{j}} B_j^{-1} C_{j}^+ B_j'v^{\mu'_j-\mu_j}.
\end{equation}
Since the RHS of (\ref{iso-phi:1}) becomes integral after multiplying by $v^{k_{j} +e(n-1)+ \max_{1\leq i\leq n}|\mu'_{j,i}-\mu_{j,i}|}$, we get that 
\[
k_{j}+p-1>k_{j}+  e(n-1)+\max_{1\leq i\leq n}|\mu'_{j,i}-\mu_{j,i}|+\max_{\alpha\in \Phi}| \langle\mu_j,\alpha^\vee \rangle |\geq pk_{j-1}. 
\]
This shows that if $k=\max_{1\leq j \leq f} k_j$, then $(p-1)k<p-1$, hence $k_j\leq 0$ for all $j\in\cJ$. Thus $C_j\in  \Mat_n(\F[\![v]\!])$, and comparing determinants we see that $C_j\in \GL_n(\F[\![v]\!])$, for all $j\in\cJ$.
(In particular $k_j=0$ for all $j\in\cJ$.)
Finally, we show that $C_j\in \Iw(\F)$. If this were not the case, then for some $\alpha\in \Phi^{-}$ the entry corresponding to $\alpha$ in $\Ad(v^{\mu_j+\eta_{0,j}})\phz(C^{+}_{j-1})$ will have $v$-adic valuation $\langle \mu_j+\eta_{0,j},\alpha^\vee \rangle$.
Comparing the $\alpha$ entry in the equation
\[\Ad(v^{\mu_j + \eta_{0, j}})\phz(C^{+}_{j-1})= B_j^{-1} C_{j}^+ B_j'v^{\mu'_j-\mu_j}\]
then shows that 
\[\langle \mu_j+\eta_{0,j},\alpha^\vee \rangle\geq -e(n-1)- \max_{1\leq i\leq n}|\mu'_{j,i}-\mu_{j,i}|\]
which contradicts the deepness assumption on $\mu$.
\end{proof}

\begin{proof}[Proof of Theorem \ref{semi:cnt}]
Let $\tau'$ be a type which exhibits the specialization to $\rhobar^{\speci}$.   Let $(s, \mu)$ and $(s', \mu')$ be lowest alcove presentations of $\tau$ and $\tau'$ respectively compatible with a fixed choice of $3e(n-1)$-generic lowest alcove presentation of $\rhobar^{\semis}$.  Note that $\mu$ and $\mu'$ are $2e(n-1)$-deep and satisfy the conclusion of Lemma \ref{lem:gap}.
Let $\tld{w} = \tld{w}(\rhobar, \tau)$ and $\tld{w}' = \tld{w}(\rhobar, \tau')$.  

The strategy is similar to the proof of Proposition \ref{prop:weak:sc}. We will construct a morphism
\begin{align}
\label{eq:family:2}
&\bA^1_{\F} \ra  Y_\F^{\leq\eta,\tau}\\
&x\mapsto \fM_{x} \nonumber
\end{align}
which satisfies the following properties:
\begin{enumerate}
	\item\label{it:sc:1} for all $x \neq 0$, the Breuil--Kisin module $\fM_x$ has shape $\tld{w}(\rhobar, \tau)$; 
	\item\label{it:sc:2} $T^*_{dd}(\fM_1)\cong \rhobar|_{G_{K_\infty}}$; and
	\item\label{it:sc:3} $T^*_{dd}(\fM_0)|_{I_{K_\infty}}\cong \rhobar^{\speci}$.
\end{enumerate}

Let $\fM' \in Y_\F^{\leq\eta,\tau'}(\F)$ be the unique Breuil--Kisin module satisfying $T^*_{dd}(\fM')\cong \rhobar|_{G_{K_\infty}}$.
By Proposition \ref{prop:modpform}, there is an eigenbasis $\beta'$ for $\fM'$ such that
\[
A_{\fM',\beta'}^{(j)} = D_j (\tld{w}'_j)^* U_j
\]
where $D_j \in T(\F)$ and $U_j \in N_{(\tld{w}'_j)^*}(\F) \subset \Iw(\F)$ is defined in \cite[Definition 4.2.9]{MLM}. 
Since $U_j$ is unipotent (\cite[Corollary 4.2.16]{MLM}),  there exists $s_j\in W$ such that 
\begin{equation}
\Ad(s_j(\eta_{0,j})(x))\cdot U_j \in 1+x\Mat_n(\F[x][\![v]\!]) \label{eq3}.
\end{equation}

We define a map $\kappa':(\mathbb{G}_m)_{\F} \ra  Y_\F^{\leq \eta,\tau'}$ by specifying  Breuil--Kisin module $\tld{\fM}'$ over $\F[x^{\pm 1}]$ of type $\tau'$ and eigenbasis $\tld{\beta}'$ such that 
\[
A^{(j)}_{\tld{\fM}', \tld{\beta}'} = D_j (\tld{w}'_j)^* \Ad(s_j(\eta_{0,j})(x))\cdot U_j
\]
for all $x \neq 0$. By \eqref{eq3}, this map extends to a map $\kappa':\bA^1_{\F} \ra  Y_\F^{\leq \eta,\tau'}$.

The map $\kappa'$ gives rise to a family $\tld{\cM} \defeq \eps_{\tau'} (\tld{\fM}')$ of \'etale $\phi$-modules over $K$ parametrized by $\bA^1_{\F}$.  %
Over $\mathbb{G}_m$, by Proposition \ref{prop:expeps}, $\tld{\cM}$ admits a basis $\fF$ such that
\[
Q'_{j} \defeq \Mat_{\fF}(\phi^{(j)}_{\tld{\cM}})= D_j (\tld{w}'_j)^*\,\big(\Ad(s_j(\eta_{0,j})(x))\cdot U_j \big)\, (s'_j)^{-1} v^{\mu'_j + \eta_{0,j}}.
\]
For $x\in \mathbb{G}_m$, we write $Q'_{j,x}=\Mat_{\fF_x}(\phi^{(j)}_{\tld{\cM}_x})$ in what follows. 
By construction, $\bV^*_K(\tld{\cM}_1)\cong\rhobar|_{G_{K_\infty}}$ and $\bV^*_K(\tld{\cM}_0)|_{I_{K_{\infty}}} \cong \rhobar^{\speci}$.

By assumption, $\cM_1$ is the \'etale $\phz$-module over $K$ associated to the unique $\fM \in Y_\F^{\leq\eta, \tau}(\F)$ satisfying $T^*_{dd}(\fM)\cong\rhobar|_{G_{K_\infty}}$. 
Choose an eigenbasis $\beta$ for $\fM$.  
By Proposition \ref{prop:expeps},  there exists $(C_j) \in  \GL_n(\F(\!(v)\!))^{\cJ}$ such that for all $j \in \cJ$
\begin{equation}
\label{iso-phi}
 C^{(j+1)} Q'_{j, 1} = A^{(j)}_{\fM, \beta} s^{-1}_j v^{\mu_j + \eta_{0,j}}  \phz(C^{(j)}).
\end{equation}
Applying Lemma \ref{lem1} with $B'_{j} = Q'_{j,1} v^{-\mu'_j - \eta_{0,j}}$ and $B_j = A^{(j)}_{\fM, \beta} s_{j}^{-1}$, we conclude that $C^{(j)} \in \Iw(\F)$ for all $j \in \cJ$.  Hence, by changing the eigenbasis of $\fM$ if necessary, we can arrange that $Q'_{j, 1} =  A^{(j)}_{\fM, \beta} s^{-1}_j v^{\mu_j + \eta_{0,j}}$.  

We now construct a map $\kappa:(\mathbb{G}_m)_{\F} \to Y_\F^{\leq\eta,\tau}$ by specifying a Breuil--Kisin module $\tld{\fM}$ over $\F[x^{\pm 1}]$ of type $\tau$ with eigenbasis $\tld{\beta}$ such that 
\[
A_{\tld{\fM}, \tld{\beta}}^{(j)} = Q'_j   v^{-\mu_j - \eta_{0,j}} s_j  =  D_j (\tld{w}'_j)^*\,\big(\Ad(s_j(\eta_{0,j})(x))\cdot U_j \big)\, (s'_j)^{-1} v^{\mu'_j - \mu_j} s_j .
\]

To see that $\kappa$ is well-defined, observe that 
\[
A^{(j)}_{\tld{\fM}_x, \tld{\beta}_x} = t_1 A^{(j)}_{\tld{\fM}_1, \tld{\beta}_1} t_2 = t_1 A^{(j)}_{\fM, \beta} t_2
\]
for suitable (constant) diagonal matrices $t_1, t_2\in T(\F')$ depending on $x\in(\F')^\times$.  
This also shows that $\kappa$ satisfies property (\ref{it:sc:1}). 
The map $\kappa$ satisfies property (\ref{it:sc:2}) by construction

The construction of $\kappa$ shows that the fiber $\tld{\cM}_x$ of the family $\tld{\cM}$ of \'etale $\phi$-modules over $\bA^1_{\F}$ comes from a point of $Y_\F^{\leq\eta,\tau}$ for each $x\neq 0$. 
Since this is a closed condition and the map $Y_\F^{\leq\eta,\tau}\ra \Phi\text{-}\Mod^{\text{\'et},n}_{K}$ is proper (Proposition \ref{prop:BK_to_phi_mono}), it follows that $\kappa$ extends to a map $\kappa:\bA^1_{\F} \to Y_\F^{\leq\eta,\tau}$, and property (\ref{it:sc:3}) holds for this extension.

\end{proof}

The proof of Theorem \ref{thm:semicont} has the following useful consequence. 
\begin{cor}\label{cor:isom}
Suppose that $\tau$ and $\tau'$ are $2e(n-1)$ generic tame types with compatible lowest alcove presentations.
Assume that $\eps_\tau(\fM)\cong\eps_{\tau'}(\fM')$ for objects $\fM\in Y^{\leq \eta,\tau}(\F)$, $\fM'\in Y^{\leq \eta,\tau'}(\F)$.   Then $\fM$ and $\fM'$ admit eigenbases $\beta$ and $\beta'$ respectively such that  %
\[
 A_{\fM,\beta}^{(j)} \tld{w}^*(\tau) =  A_{\fM',\beta'}^{(j)} \tld{w}^*(\tau') 
\]
for all $j\in\cJ$.
\end{cor}

\subsection{Specialization pairs}
\label{subsec:SPEC:pairs}

In this subsection, we enhance the notion of specialization of $\rhobar$ to a pair of specialization and a Serre weight.  The pairs exhibit a nice combinatorial structure indexed by the Weyl group (see Definition \ref{defn:thetarhobar}).

\begin{lemma}\label{lemma:import}
Suppose that $\rhobar^\speci$ is an extremal specialization of $\rhobar$ and that $\tau$ is a tame inertial type exhibiting this specialization and such that $\tld{w}(\rhobar,\tau) = t_{w^{-1}(e\eta_0)}$ for some $w\in W^{\cJ}$. %
Let $\tld{w} \in \tld{W}_1^{+,\cJ}$ be an element (unique up to $X^0(\un{T})$) whose image in $W^{\cJ}$ is $w$.
Let $\tau_g$ be the tame inertial type with lowest alcove presentation compatible with $\tau$ such that $\tld{w}(\rhobar^\speci,\tau_g)$ is the unique element in $\Omega^{\cJ} w_0 t_{(e-1)\eta_0}\tld{w}\cap t_{e\eta_0}W^{\cJ}_a$. 
Assume that $\tau_g$ is $(e(n-1)+1)$-generic.
Then $\rhobar$ is $\tau_g$-admissible and $\tld{w}(\rhobar,\tau_g) = \tld{w}(\rhobar^\speci,\tau_g).$ %
\end{lemma}
\begin{proof}
Note that $t_{w^{-1}(e \eta_0)} =\tld{w}(\tau)^{-1} \tld{w}(\rhobar^\speci, \tau) = \tld{w}(\rhobar^{\speci}) $ and we can write $t_{w^{-1}(e \eta_0)} = \tld{w}_2^{-1} w_0 t_{(e-1) \eta_0} \tld{w}$ where $\tld{w}_2 \in \tld{W}^{+, \cJ}_1$.   Let $\delta \in \Omega^{\cJ}$ such that $\delta \tld{w}_2 \in W^{\cJ}_a$.    We define $\tau_g$ to be the unique tame inertial type such that $\tld{w}(\tau_g) = \tld{w}(\tau) (\delta \tld{w}_2)^{-1}$.  
By definition $\tau_g$ is endowed with a compatible lowest alcove presentation which is $(e(n-1) + 1)$-generic and $\tld{w}(\rhobar^\speci,\tau_g)$ is as desired.   
It remains to show that $\tld{w}(\rhobar,\tau_g) = \tld{w}(\rhobar^\speci,\tau_g).$

Let $\fM \in Y^{\eta, \tau}(\F)$ be unique Breuil--Kisin module such that $T^*_{dd}(\fM) \cong \rhobar|_{G_{K_{\infty}}}$.    By assumption, $\fM$ has shape $t_{w^{-1}(e \eta_0)}$ and satisfies the mod $p$ monodromy condition (Definition \ref{defn:modpmono}).    Hence, for any choice of eigenbasis $\beta$, we have that $\iota_{\tau}(\fM)$ is the \'etale $\phi$-module with partial Frobenii given by
$
A^{(j)}_{\fM, \beta} \tld{w}^{*}(\tau)_j
$
for $j \in \cJ$, where 
$$
A^{(j)}_{\fM, \beta} \tld{w}^{*}(\tau)_j\in (\cI(\F)   t_{w^{-1}(e\eta_{0,j})} \cI(\F)  \tld{w}^{*}(\tau)_j)^{\nabla_0}.
$$
Applying Proposition \ref{prop:nosematch} with $\tld{w} = t_{w^{-1}(e\eta_{0,j})},  \tld{w}' = \delta w_0 t_{(e-1) \eta_{0,j}} \tld{w}_{1,j}$ and $\tld{z} = \tld{w}^*(\tau)_j$, we have for all $j \in \cJ$: 
$$
A^{(j)}_{\fM, \beta} \tld{w}^{*}(\tau)_j \in \cI(\F) (\delta w_0 t_{(e-1) \eta_{0,j}} \tld{w}_1)^* \cI(\F) \tld{w}^*(\tau_g)_j
$$
(note that $\tld{w} = t_{w^{-1}(e\eta_{0,j})},  \tld{w}' = \delta w_0 t_{(e-1) \eta_{0,j}} \tld{w}_{1,j}$ are $e$-regular and $n-1$-small by Proposition \ref{prop:can:reg}).
Hence there exists $\fM' \in  Y^{\eta, \tau_g}(\F)$ such that $\eps_{\tau_g}(\fM') \cong \eps_{\tau}(\fM)$ and such that $\fM'$ has shape $\delta w_0 t_{(e-1) \eta_0} \tld{w}_1 = \tld{w}(\rhobar^\speci,\tau_g)$.   
\end{proof}

\begin{defn}[Specialization pairs]
Suppose that $\rhobar$, $\rhobar^\speci$, $\tau$, and $w$ are as in Lemma \ref{lemma:import}.
Let $\tld{w}\in (\tld{W}^+_1)^{\cJ}$ be the unique (up to $X^0(\un{T})$) element whose projection in $W^{\cJ}$ is $w$.
Let $\sigma$ be the Serre weight 
\begin{equation} \label{eq:weight}
F(\pi^{-1}(\tld{w})\cdot(\tld{w}(\tau)\tld{w}^{-1}\tld{w}_h^{-1}(0)-\eta_0))) = F_{(\tld{w}, \tld{w}(\tau)\tld{w}^{-1}\tld{w}_h^{-1}(0))}.
\end{equation} 
Then we say that $\rhobar$ \emph{specializes to the pair} $(\sigma,\rhobar^\speci)$ and that $\tau$ exhibits this specialization.
Let $SP(\rhobar)$ be the set of pairs to which $\rhobar$ specializes. %
\end{defn}
Note that if $\rhobar^\speci$ is $\max\{2,e\}(n-1)$-generic and $\tau$ is $2(n-1)$-generic then $\sigma$ is the unique element in $W^?(\rhobar^\speci,\tau)$ by Proposition \ref{prop:indint}, and is the extremal weight corresponding to $w$ (see Definition \ref{defn:obv}).

If $\rhobar$ is $2e(n-1)+1$-generic, we have a natural map $SP(\rhobar) \ra S_{\mathrm{ext}}(\rhobar)$ which is surjective and hence the set $SP(\rhobar)$ is finite.
If $\zeta\in X^*(\un{T})$ and some $\rhobar^\speci \in S_{\mathrm{ext}}(\rhobar)$ has a $\zeta$-compatible lowest alcove presentation, then every element of $S_{\mathrm{ext}}(\rhobar)$ has a $\zeta$-compatible lowest alcove presentation.

\begin{defn} \label{defn:thetarhobar}
Assume that $\rhobar$ is $(2e(n-1)+1)$-generic.
Let $\zeta \in X^*(\un{T})$ and suppose that some (equivalently any) element of $S_{\mathrm{ext}}(\rhobar)$ has a $\zeta$-compatible lowest alcove presentation. 
We define a map $$\theta_\rhobar^\zeta: SP(\rhobar) \ra W^{\cJ}$$ as follows:
If $(\sigma,\rhobar^\speci)$ is in $SP(\rhobar)$, $(w_{\rhobar^\speci},\mu_{\rhobar^\speci})$ is a $\zeta$-compatible lowest alcove presentation of $\rhobar^\speci$, and $\sigma$ is the extremal weight corresponding to $w$, %
we set $\theta_\rhobar^\zeta(\sigma,\rhobar^\speci) = w_{\rhobar^{\speci}} w^{-1}$. 
\end{defn}

\begin{prop}\label{prop:inj}
Assume that $\rhobar$ is $3e(n-1)$-generic.
The map $\theta_\rhobar^\zeta$ is injective.  $($Later in \S \ref{subsec:WE:MOD}, we show that the map is bijective.$)$ 
\end{prop}
\begin{proof}
Suppose that $\theta_\rhobar^\zeta(\sigma,\rhobar^\speci) = \theta_\rhobar^\zeta(\sigma',\rhobar^{\prime,\speci})$.
Let $\tau$ and $\tau'$ be as in Lemma \ref{lemma:import} exhibiting these specializations with $\tld{w}(\rhobar,\tau) = t_{w^{-1}(e\eta_0)}$ and $\tld{w}(\rhobar,\tau) = t_{(w')^{-1}(e\eta_0)}$, and let $\tau_g$ and $\tau'_g$ also be as in Lemma \ref{lemma:import}.
Then $\sigma$ and $\sigma'$ are the extremal weights of $\rhobar^\speci$ and $\rhobar^{\prime,\speci}$ corresponding to $w$ and $w' \in W$, respectively. 
Let $\tld{w}$ and $\tld{w}' \in \tld{W}_1^+$ be elements with images $w$ and $w' \in W$, respectively.
Then by Lemma \ref{lemma:import}, there exist $\delta, \delta' \in \Omega$ such that $\tld{w}(\rhobar,\tau_g) = \delta w_0  t_{(e-1)\eta_0} \tld{w}$ and $\tld{w}(\rhobar,\tau_g')  = \delta' w_0  t_{(e-1)\eta_0} \tld{w}'$. 
By Corollary \ref{cor:isom} and the fact that $\tld{w}^*(\tau_g') \phz(\cI) \tld{w}^*(\tau_g')^{-1} \subset \cI$, 
\[
\cI  \tld{w}(\rhobar, \tau_g)^* \cI  \tld{w}^*(\tau_g)\cap \cI   \tld{w}(\rhobar, \tau'_g)^* \cI \tld{w}^*(\tau_g') \neq \emptyset,
\]
or equivalently by taking transposes,
\begin{equation}\label{eqn:cellint}
\tld{w}(\tau_g)\delta  \cI^{\mathrm{op}} w_0 t_{(e-1)\eta_0}\tld{w} \cI^{\mathrm{op}}/ \cI^{\mathrm{op}} \cap \tld{w}(\tau_g')\delta' \cI^{\mathrm{op}} w_0 t_{(e-1)\eta_0}\tld{w}' \cI^{\mathrm{op}}/ \cI^{\mathrm{op}} \neq \emptyset,
\end{equation}
where $\cI^{\mathrm{op}}$ is the opposite Iwahori group scheme. %

To simplify notation, let $\tld{s}$ and $\tld{s}'$ be $\tld{w}(\tau_g)\delta$ and $\tld{w}(\tau_g')\delta'$, respectively.
Then we have $\tld{w}(\rhobar^\speci) = \tld{s}w_0 t_{(e-1)\eta_0} \tld{w}$ and $\tld{w}(\rhobar^{\prime,\speci}) = \tld{s}'w_0 t_{(e-1)\eta_0} \tld{w}'$.
Let $s$, $s'$, $w(\rhobar^\speci)$, and $w(\rhobar^{\prime,\speci}) \in W$ be the images of $\tld{s}$, $\tld{s}'$, $\tld{w}(\rhobar)$, and $\tld{w}(\rhobar^{\prime,\speci})$, respectively.
The equality $w(\rhobar^\speci)w^{-1} = \theta_\rhobar^\zeta(\sigma,\rhobar^\speci) = \theta_\rhobar^\zeta(\sigma',\rhobar^{\prime,\speci}) = w(\rhobar^{\prime,\speci}) (w')^{-1}$ implies that $s = s'$.

The previous paragraph and (\ref{eqn:cellint}) and imply that there exists $\nu\in X^*(T)$ such that
\[
t_\nu \cI^{\mathrm{op}} w_0 t_{(e-1)\eta_0}\tld{w} \cI^{\mathrm{op}}/ \cI^{\mathrm{op}} \cap \cI^{\mathrm{op}}w_0 t_{(e-1)\eta_0}\tld{w}' \cI^{\mathrm{op}}/ \cI^{\mathrm{op}} \neq \emptyset.
\]
Both $t_\nu  \cI^{\mathrm{op}} w_0 t_{(e-1)\eta_0}\tld{w}  \cI^{\mathrm{op}}/  \cI^{\mathrm{op}}$ and $ \cI^{\mathrm{op}} w_0 t_{(e-1)\eta_0}\tld{w}'  \cI^{\mathrm{op}}/ \cI^{\mathrm{op}}$ are stable under the left action of $T$.
There is a $\bG_m$-subgroup which contracts $t_\nu  \cI^{\mathrm{op}} w_0 t_{(e-1)\eta_0}\tld{w} \cI^{\mathrm{op}}/  \cI^{\mathrm{op}}$ to $t_\nu w_0 t_{(e-1)\eta_0}\tld{w}$. %
So $t_\nu w_0 t_{(e-1)\eta_0}\tld{w}$ is in the closure of $  \cI^{\mathrm{op}}w_0 t_{(e-1)\eta_0}\tld{w}'  \cI^{\mathrm{op}}/ \cI^{\mathrm{op}}$, or equivalently $t_\nu w_0 t_{(e-1)\eta_0}\tld{w} \leq w_0 t_{(e-1)\eta_0}\tld{w}'$.
Symmetrically, $t_{-\nu}w_0t_{(e-1)\eta_0}\tld{w}' \leq w_0t_{(e-1)\eta_0}\tld{w}$.
Lemma \ref{lemma:doubleclosure} implies that $\tld{w}' = t_\nu \tld{w}$ and that $\nu\in X^0(T)$.
In particular, we have $w = w'$ so that $\tau = \tau'$, $\sigma \cong \sigma'$, and $\rhobar^\speci = \rhobar^{\prime,\speci}$.
\end{proof}

\subsection{Extremal weights} 
\label{subsec:extr:wt}
In this section, we define extremal weights and use them to give a tameness criterion for Galois representations.

\begin{defn}[Extremal weights] Let $\rhobar$ be a $(2 e(n-1) + 1)$-generic  representation of $G_K$.  Define $W_\obv(\rhobar)$ to be the set of Serre weights $\sigma$ such that there exists some $\rhobar^\speci$ so that $(\sigma,\rhobar^\speci) \in SP(\rhobar)$.
\end{defn}

\begin{prop} \label{prop:obvforss}
If $\rhobar$ is semisimple and $(2e(n-1)+1)$-generic, then $W_\obv(\rhobar)$ agrees with the set $W_\obv(\rhobar|_I)$ from Definition \ref{defn:obv}.
\end{prop}
\begin{proof}  We first note that if $\rhobar$ is semisimple, then it only specializes to $\rhobar|_{I_K}$ by Proposition \ref{prop:sskisin}.  

Fix now a $(2e(n-1)+1)$-generic lowest alcove presentation of $\rhobar|_{I_K}$.   
For each $w \in W^{\cJ}$, there is a unique type $\tau$ (with compatible lowest alcove presentation) such that $\tld{w}(\rhobar|_{I_K}, \tau) = \tld{w}(\rhobar, \tau) = t_{w{-1}(e \eta_0)}$.   
 Let $\tld{w}\in \tld{W}^{+,\cJ}_1$ be the unique element whose projection in $W^{\cJ}$ is $w$.
 This type realizes the specialization to the pair $\rhobar|_{I_K}$ and  $\sigma =  F_{(\tld{w}, \tld{w}(\tau)\tld{w}^{-1}\tld{w}_h^{-1}(0))}$ (see \eqref{eq:weight}).  
Using that $\tld{w}(\tau) = \tld{w}(\rhobar|_{I_K}) t_{e \eta_0}$, we see that $\sigma$ is the extremal weight of $\rhobar|_{I_K}$ corresponding ot $w$. 
\end{proof}

\begin{prop}\label{prop:tamecrit}
Assume that $\rhobar$ is $3e(n-1)$-generic.
The following are equivalent:
\begin{enumerate}
\item 
\label{it:tamecrit:1}
$\rhobar$ is semisimple; and
\item 
\label{it:tamecrit:2}
$\# W_\obv(\rhobar) = \# W^{\cJ}$.
\end{enumerate}
\end{prop}
\begin{proof}
Proposition \ref{prop:obvforss} gives \eqref{it:tamecrit:1} implies \eqref{it:tamecrit:2}.   
Next, assume that $\# W_\obv(\rhobar) = \# W^{\cJ}$.  By the injectivity of $\theta_{\rhobar}$ (Proposition \ref{prop:inj}), $\rhobar$ has a unique extremal specialization, call it $\rhobar^{\speci}$, and furthermore, $W_\obv(\rhobar) = W_\obv(\rhobar^{\speci})$.  

Let $w \in W^{\cJ}$.   Let $\sigma$ (resp.~$\sigma'$) be the extremal weight associated to $w$ (resp.~$w w_0$).   We show that if $\rhobar$ specializes to both $(\sigma, \rhobar^{\speci})$ and $(\sigma', \rhobar^{\speci})$   then $\rhobar$ is semisimple and $\rhobar|_{I_K} \cong \rhobar^{\speci}$.    Let $\tau$ and $\tau'$ be the types realizing these specialization in shape $\tld{z} = t_{w^{-1} (e \eta_0)}$ and $\tld{z}' = t_{w_0 w^{-1} (e \eta_0)}$ with corresponding Breuil--Kisin module $\fM$ and $\fM'$.   

By Proposition \ref{prop:modpform}, there exists eigenbases $\beta$ and $\beta'$ respectively such that 
\[
A^{(j)}_{\fM, \beta} = D_j U_j \tld{z}_j,  \quad A^{(j)}_{\fM', \beta'} = D'_j U'_j \tld{z}'_j
\]
where $D_j, D'_j \in T(\F)$, $U_{j} \in \tld{z}_j N_{\tld{z}_j}(\F) \tld{z}_j^{-1}$, and $U'_{j}  =  \tld{z}'_j N'_{\tld{z}'_j}(\F) (\tld{z}'_j)^{-1}$.  By definition of $N_{\tld{z}}$ (\cite[Definition 4.2.9]{MLM}), we have 
\[
U_j, U'_j \in L^{--} \cG^{(j)}_{\F}(\F)
\]
where $L^{--} \cG^{(j)}_{\F}$ denotes the negative loop group for $L\cG^{(j)}_{\F}$ (in particular, its $\F$ points consist of matrices $A\in\GL_n(\F[1/v])$ which are lower unipotent modulo $1/v$).

By Corollary \ref{cor:isom}, there exists $(I^{(j)}) \in \cI(\F)^{\cJ}$ such that
\begin{equation} \label{eq:a1}
 D_j U_j \tld{z}_j \tld{w}^*(\tau) = I^{(j)} D'_j U'_j \tld{z}'_j \tld{w}^*(\tau') (\phz(I^{(j-1)}))^{-1}
\end{equation}
By scaling $\beta'$ by an element of $T(\F)^{\cJ}$ if necessary, we can arrange that $(I^{(j)}) \in \cI_1(\F)^{\cJ}$.    Since both $\tau$ and $\tau'$ realize the same specialization, $\tld{z}_j \tld{w}^*(\tau) = \tld{z}'_j \tld{w}^*(\tau') = \tld{w}^*(\rhobar^{\speci})$ and so \eqref{eq:a1} becomes
 \begin{equation} \label{eq:a2}
 D_j U_j \tld{w}^*(\rhobar^{\speci}) = I^{(j)} D'_j U'_j \tld{w}^*(\rhobar^{\speci}) (\phz(I^{(j-1)}))^{-1}.
\end{equation}

By Lemma \ref{lem:Istraight},
there exists $(X_j) \in \cI_1(\F)^{\cJ}$ such that $ D_j U_j = X_j D'_j U'_j$ for all $j \in \cJ$.  Thus, $U_j (U'_j)^{-1} \in \Iw(\F) \cap L^{--} \cG^{(j)}_{\F}(\F)$  and so $U_j = U'_j$.   Finally, since $t_{w^{-1}(e(\eta_0))} (A_0) $ and $t_{w_0 w^{-1}(e(\eta_0))} (A_0)$ are in opposite Weyl chambers, $\tld{z}_j^{-1} U_j \tld{z}_j$ and $(\tld{z}'_j)^{-1} U'_j \tld{z}'_j$ are in opposite unipotents by \cite[Corollary 4.2.15]{MLM}.
Thus,  $U_j$ and $U'_j$ are the identity ($\tld{z}_j, \tld{z}'_j$ are both translations).    Since   $A^{(j)}_{\fM, \beta} = D_j \tld{z}_j$ for all $j \in \cJ$, it follows that $T^*_{dd}(\fM)$ is semisimple (see \cite[Proposition 5.5.2]{MLM} for example).     
\end{proof}

\subsection{Maximally ordinary weights}\label{sec:mord}

In this section, we show that the set $W_\obv(\rhobar)$ contains the set of \emph{maximally ordinary weights}.
We further show that the set of maximally ordinary weights is nonempty, so that in particular, the set $W_\obv(\rhobar)$ is nonempty.
When $\rhobar$ is an iterated extension of characters, the set of maximally ordinary weights is the set of \emph{ordinary weights}. %

\begin{lemma}\label{lemma:extension}
Suppose that $\rhobar: G_K \ra \GL_n(\F)$ is $(e(n-1)+2)$-generic and an extension of $\rhobar_2$ by $\rhobar_1$.
For $i = 1$ and $2$, let $n_i$ be the dimension of $\rhobar_i$.
Suppose that $\rhobar_i$ has a potentially crystalline lift $\rho_i : G_K \ra \GL_n(\cO_E)$ of tame inertial type $\tau_i$ and parallel Hodge--Tate weights $(n_1+n_2-1,\ldots,n_2)$ (resp.~$(n_2-1,\ldots,0)$) if $i=1$ (resp.~$i=2$).
Then $\rhobar$ has a lift $\rho$ which is an extension of $\rho_2$ by $\rho_1$ and is potentially crystalline of type $(\tau,\eta)$ where $\tau = \tau_1\oplus \tau_2$.
\end{lemma}
\begin{proof}
Note that by genericity, both $\rhobar$ and $\tau$ are at least $2$-generic, in particular are cyclotomic-free (\cite[Lemma 7.2.9]{MLM}).
By genericity, $\Ext^2_{G_K}(\rho_2,\rho_1)$ is zero.
So the natural reduction map $\Ext^1_{G_K}(\rho_2,\rho_1) \ra \Ext^1_{G_K}(\rhobar_2,\rhobar_1)$ is surjective.
We conclude that there exists a lift $\rho: G_K \ra \GL_n(\cO_E)$ of $\rhobar$ which is an extension of $\rho_2$ by $\rho_1$.

Let $\rho_{i,E}$ be $\rho_i \otimes_{\cO_E} E$.
Then the containment $H^1_g(G_K,\rho_{2,E}^\vee \otimes_E \rho_{1,E}) \subset H^1(G_K,\rho_{2,E}^\vee \otimes_E \rho_{1,E})$ is an equality for dimension reasons.
Indeed, by the local Euler characteristic formula and Tate duality, we have that $h^1(G_K,\rho_{2,E}^\vee \otimes_E \rho_{1,E}) = \dim_E \rho_{2,E}^\vee \otimes_E \rho_{1,E}$.
On the other hand, $h^1_g(G_K,\rho_{2,E}^\vee \otimes_E \rho_{1,E}) = \dim_E D_{\mathrm{dR}}(\rho_{2,E}^\vee \otimes_E \rho_{1,E})/D_{\mathrm{dR}}(\rho_{2,E}^\vee \otimes_E \rho_{1,E})^+$.
Since the Hodge--Tate weights of $\rho_2$ are strictly less than those of $\rho_1$, this latter expression is $\dim_E \rho_{2,E}^\vee \otimes_E \rho_{1,E}$ as well.
We conclude that $\rho$ is an $\cO_E$-lattice in a potentially semistable representation.
Moreover, $\rho$ has parallel Hodge--Tate weights $\eta$.

There is an exact sequence of smooth $I_K$-representations 
\[
0 \ra D_{\mathrm{pst}}(\rho_{1,E}) \ra D_{\mathrm{pst}}(\rho_E) \ra D_{\mathrm{pst}}(\rho_{2,E}) \ra 0.
\]
We conclude that $D_{\mathrm{pst}}(\rho_E) \cong \tau = \tau_1 \oplus \tau_2$.
Moreover, by genericity, $\Hom_{I_K}(\tau_2,\tau_1(-1)) = 0$ and so $\rho$ must be potentially crystalline.
\end{proof}

Let $P^\vee\subset \GL_n$ be a parabolic subgroup with Levi quotient $M^\vee$.
Then $M^\vee \cong \prod_{i=1}^k M^\vee_i$ where $M^\vee_i \cong \GL_{n_i}$ and $\sum_{i=1}^k n_i = n$.
Let $N_i$ be $\sum_{j=i+1}^k n_j$.
We index these dimensions so that for all $1 \leq i \leq k$, $P^\vee$ has a quotient $P^\vee_i$ which is isomorphic to a parabolic subgroup of $\GL_{N_i}$ with Levi quotient $\prod_{j=i+1}^k M^\vee_j$.
In other words, if $P$ is block upper diagonal, then starting from the top left, the $i$-th block has size $n_i$.

\begin{thm}\label{thm:extension}
Let $\rhobar: G_K \ra \GL_n(\F)$ be $(2e(n-1)+1)$-generic.
Suppose that $\rhobar$ factors through $P^\vee(\F)$ for a parabolic subgroup $P^\vee \subset \GL_n$ as above.
Let $M^\vee$, $M^\vee_i$, and $N_i$ be as above.
Suppose that the associated representations $\rhobar_i: G_K \ra M^\vee_i(\F)$ are semisimple. 
For each $i$, let $\rho_i$ be a potentially crystalline lift of type $\tau_i$ and parallel Hodge--Tate weights $(n_i+N_i -1,\ldots,N_i)$ where $\tld{w}(\rhobar_i(-N_i),\tau_i)$ is extremal (here, $\rhobar_i(-N_i)$ means twist of $\rhobar_i$ by the $(-N_i)$-th power of the cyclotomic character).
Then $\rhobar$ has a potentially diagonalizable lift $\rho$ (in the sense of \cite[\S 1.4]{BLGGT}) of type $(\tau,\eta)$ where $\tau = \oplus_{i=1}^k \tau_i$.
The corresponding specialization is $\oplus_{i=1}^k \rhobar_i$.
In particular, the semisimplification $\oplus_{i=1}^k \rhobar_i$ is an extremal specialization of $\rhobar$.
\end{thm}
\begin{proof}
By iterated application of Lemma \ref{lemma:extension}, we obtain a potentially crystalline lift $\rho$ of type $(\tau,\eta)$, which is an iterated extension of potentially crystalline lifts (of type $\tau_i$ and parallel Hodge--Tate weights $(n_i+N_i -1,\ldots,N_i)$) of the representations $\rhobar_i$.
In particular, the semisimplification of $\rho$ is $\oplus_{i=1}^k \rho_i$.
Then by the argument of proof of \cite[Corollary 3.4.11]{LLL} (replacing the reference to Proposition 3.4.8 in \emph{loc.~cit}.~with Proposition \ref{prop:gauge:basis} below, and noting that the semisimple Kisin module produced as in \emph{loc.~cit}.~has Hodge--Tate weights exactly $\eta$), after restriction to a finite index subgroup the semisimplification of $\rho$ is a direct sum of characters.
By \cite[Lemma 1.4.3(1)]{BLGGT}, $\rho$ is potentially diagonalizable.

Since $\tld{w}(\rhobar_i(-N_i),\tau_i)$ is extremal for all $i = 1,\ldots, k$, so is $\tld{w}(\oplus_{i=1}^k \rhobar_i,\tau)$ by an easy computation. 
Since the semisimplification of $\rhobar$ is $\oplus_{i=1}^k \rhobar_i$, we deduce that $\tld{w}(\rhobar,\tau)$ is this same shape by Proposition \ref{prop:weak:sc}.
Thus $\tau$ exhibits the specialization $\oplus_{i=1}^k \rhobar_i$ of $\rho$.
\end{proof}

Suppose that $\rhobar$ is as in Theorem \ref{thm:extension} with $P^\vee$, $M^\vee$, and $N_i$ as before.
Let $P\subset \GL_n$ be the dual parabolic subgroup. 
Let $U$ be the unipotent radical of $P$.
For each $i$ let $\sigma_i \in W_\obv(\rhobar_i)$.
Let $\sigma$ be the unique Serre weight such that 
\[
\sigma^U \cong \boxtimes_{i=1}^k \sigma_i (-N_i).
\]
We call a Serre weight constructed in this way \emph{maximally ordinary}.
Let $W_{\mord}(\rhobar)$ be the set of maximally ordinary Serre weights.
Since we can always find $P^\vee$ as in Theorem \ref{thm:extension}, $W_{\mord}(\rhobar)$ is nonempty. 
If $\rhobar$ is semisimple, then we can take $P^\vee$ to be $\GL_n$ so that $W_{\mord}(\rhobar) = W_\obv(\rhobar)$.
Taking $P^\vee$ to be a minimal parabolic when $\rhobar^\semis$ is a direct sum of characters, we see that ordinary weights are maximally ordinary.

\begin{prop}\label{prop:obvord}
There is an inclusion $W_{\mord}(\rhobar) \subset W_\obv(\rhobar)$.
\end{prop}
\begin{proof}
Let $\rhobar$ and $\rhobar_i$ be as in Theorem \ref{thm:extension}. 
Suppose that $\sigma \in W_{\mord}(\rhobar)$.
For each $i$, let $\tau_i$ be the tame type such that $W^?(\rhobar_i,\tau_i) = \{\sigma_i\}$.
Then if we let $\tau$ be $\oplus_{i=1}^k \tau_i$, then $\tau$ exhibits a specialization of $\rhobar$ to $\rhobar^\semis$.
Moreover, one can check that $W^?(\rhobar^\semis,\tau) = \{\sigma\}$ so that $(\sigma,\rhobar^\semis) \in SP(\rhobar)$.
This shows that $\sigma \in W_\obv(\rhobar)$.
\end{proof}

\subsection{Connections to Emerton--Gee stacks}
\label{sec:EGs}

This section is a series of remarks explaining how the notions of extremal weight and specialization can be interpreted geometrically on the stack of mod $p$ Galois representations $\cX_n$ introduced by Emerton--Gee \cite{EGstack}.  When $K/\Qp$ is unramified, everything can be proved using the techniques of \cite{MLM}.  The ramified case requires extending \cite{MLM} which will be the subject of future work.

First, we briefly recall what we need from \cite{EGstack}.    
In \cite[Theorem 6.5.1]{EGstack}, Emerton and Gee describe a parametrization of the irreducible components of the underlying reduced stack $\cX_{n,\red}$ of the moduli of $(\phz,\Gamma)$-modules $\cX_n$ by Serre weights of $\GL_n(\cO_K)$.  Let $\sigma=F(\kappa)$ be a Serre weight of $\GL_n(\cO_K)$ with $\kappa = (\kappa_j) \in X_1(T)^{\cJ}$.   We use the normalization as in \cite{MLM} where $\cC_{\sigma}\defeq \cX^{\sigma^\vee\otimes \det^{n-1}}_{EG,n,\red}$.

If $\kappa$ is $1$-deep, then $\cC_{\sigma}$ is uniquely characterized by the fact that has a Zariski open subset consisting of $\rhobar$ of the form  
\[ \rhobar \cong \begin{pmatrix}  \chi_1 &* &\cdots & * \\
0&\chi_2& \cdots& *\\
\vdots &&\ddots&\vdots\\
0&\cdots &0 & \chi_n
\end{pmatrix}
\]
where $\chi_{i}|_{I_K}= \overline{\eps}^i \prod_{j\in \cJ} \ovl{\omega}_{K,\sigma_j}^{(\kappa_j)_i}$ and $\rhobar$ admits a unique $G_K$-stable flag.

\begin{rmk} \label{rmk:geomspec}
\begin{enumerate}
\item Let $\rhobar^{\speci}$ be a sufficiently generic tame $\F$-type.  If $\sigma$ is an extremal weight of $\rhobar^{\speci}$ as in Definition \ref{defn:obv}, then there is a Zariski open subset $\cC^{\rhobar^{\speci}}_{\sigma} \subset \cC_{\sigma}$ such that $\rhobar$ specializes to the pair $(\sigma, \rhobar^{\speci})$ if and only if $\rhobar \in   \cC^{\rhobar^{\speci}}_{\sigma}$.   The Zariski open can be constructed via the generalization of the diagram in \cite[Theorem 7.4.2]{MLM} to the ramified setting.
\item Let $\sigma$ be a sufficiently generic Serre weight.   Then there are $(n!)^{\cJ}$ sufficiently generic tame $\F$-types which have $\sigma$ as an extremal for $\rhobar$ corresponding to $w$  weight.    
Thus, the union $\cC^{\obv}_{\sigma} = \bigcup \cC^{\rhobar^{\speci}}_{\sigma}$ where $\rhobar^{\speci}$ ranges over all such types is a Zariski open subset of $\cC_{\sigma}$ consisting exactly of the $\rhobar$ which have $\sigma$ as an extremal weight.     
(One can check that $\cC^{\obv}_{\sigma} = \cC_{\sigma}$ only when $K/\Qp$ is unramified and $\sigma$ is Fontaine--Laffaille.)
\end{enumerate}
\end{rmk}

\begin{rmk}
As has been introduced in other settings (\cite{GL3Wild}), there is a natural set of Serre weights that can be associated to an arbitrary $\rhobar:G_K \rightarrow \GL_n(\F)$, the \emph{geometric} weights, 
\[
W^{g}(\rhobar) = \{ \sigma \mid \rhobar \in \cC_{\sigma}(\F) \}. 
\]
Remark \ref{rmk:geomspec} says that $W_{\obv}(\rhobar) \subset W^{g}(\rhobar)$.    Generally speaking the set of geometric weights will be larger.
\end{rmk}

\subsection{The extremal locus}
\label{sub:extr:loc}

In this section, we discuss the relationship between $W_{\obv}(\rhobar)$ and $W^{g}(\rhobar)$ when $K/\Q_p$ is unramified. 
This gives, in this setting, an alternative to the proof of the existence of extremal weights in \S \ref{sec:mord}. 
The main result of this section will also be used to construct global lifts in \S \ref{sec:TW}.%

Let $K/\Q_p$ be unramified. 
Let $(\tld{w}_1,\omega)$ be a lowest alcove presentation for a Serre weight $\sigma$ compatible with $\zeta \in X^*(\un{Z})$. 
Recall from \cite[Definition 4.6.1]{MLM} that $C^\zeta_\sigma$ is the closure of 
\[
(\cI\backslash \cI (w_0\tld{w}_1)^* \cI (t_\omega)^*)^{\nabla_0} 
\] 
inside $\Fl^{\nabla_0}_{\cJ}$ (see also \S \ref{sec:Modp:mon}).
We define $C^\zeta_{\sigma,\obv}$ to be the (Zariski) open subset 
\[
\cup_{w\in \un{W}} (\cI\backslash \cI (w_0\tld{w}_1)^* \cI (t_\omega w)^*)^{\nabla_0} \subset C^\zeta_\sigma. 
\]

Assume that $\sigma$ is $(3n-1)$-deep. 
Then by \cite[Remark 7.4.3(2)]{MLM}, we have a local model diagram for $\cC_\sigma$ (the irreducible component of $\cX_{n,\red}$ corresponding to $\sigma$, cf.~\S \ref{sec:EGs}) and $C^\zeta_\sigma$.
We then let $\cC_{\sigma,\obv} \subset \cC_\sigma$ be the Zariski open set of $\cC_\sigma$ corresponding to $C^\zeta_{\sigma,\obv}\subset C^\zeta_{\sigma}$.
(The definition of $\cC_{\sigma,\obv}$ does not depend on the lowest alcove presentation of $\sigma$.) 

\begin{prop}\label{prop:obvlocus}
Let $\rhobar$ be $(2(n-1)+1)$-generic.
If $\sigma$ is $(3n-1)$-deep and $\rhobar\in \cC_\sigma$, then $\rhobar \in \cC_{\sigma,\obv}$ if and only $\sigma \in W_{\obv}(\rhobar)$. 
\end{prop}
\begin{proof}
We fix a lowest alcove presentation $(\tld{w}_1,\omega)$ for $\sigma$ compatible with $\zeta$.
Let $x\in C_\sigma^\zeta$ correspond to $\rhobar\in \cC_\sigma$ in the local model diagram \cite[Theorem 7.4.2]{MLM}. 
If $\sigma \in W_{\obv}(\rhobar)$, then let $\tau$ be a tame inertial type exhibiting the extremal weight $\sigma$. 
Then 
\begin{align*}
x \in & (\cI\backslash \cI t_{(w_1)^{-1}(\eta)}^* \cI \tld{w}(\tau)^*)^{\nabla_0} \\
=& (\cI\backslash \cI (w_0\tld{w}_1)^*\cI (\tld{w}(\tau)(\tld{w}_h \tld{w}_1)^{-1})^*)^{\nabla_0} \\
=& (\cI\backslash \cI (w_0\tld{w}_1)^*\cI (t_\omega w(\tau)(w_0 w_1)^{-1})^*)^{\nabla_0} \\
\subset & C^\zeta_{\sigma,\obv}
\end{align*}
where $\tld{w}(\tau)$ is defined with respect to the lowest alcove presentation of $\tau$ compatible with $\zeta$. 

Conversely, suppose that $\rhobar \in \cC_{\sigma,\obv}$. 
Let $w \in \un{W}$ be such that $x\in (\cI\backslash \cI (w_0\tld{w}_1)^* \cI (t_\omega w)^*)^{\nabla_0}$. 
Then we let $\tau$ be such that $\tld{w}(\tau) = t_\omega w \tld{w}_h \tld{w}_1$. 
The above calculation shows that $\tau$ exhibits $\sigma$ as an extremal weight of $\rhobar$. 
\end{proof}

\begin{prop}\label{prop:obvstrat}
Assume that $\sigma$ is $(4n-2)$-deep.
There is an inclusion
\[
\cC_\sigma \subset \underset{\sigma \textrm{ covers } \sigma'}{\cup} \cC_{\sigma',\obv}.
\]
\end{prop}
\begin{proof}
We choose a $(4n-2)$-deep lowest alcove presentation $(\tld{w}_1,\omega)$ of $\sigma$ and will show that 
\[
C_\sigma^\zeta \subset \underset{\sigma \textrm{ covers } \sigma'}{\cup} C_{\sigma',\obv}^\zeta. 
\]
Since the elements of $\tld{\un{W}}$ less than or equal to $w_0\tld{w}_1$ are exactly those of the form $s \tld{w}$ for some $s\in \un{W}$ and $\tld{w} \in \tld{W}^+$ with $\tld{w} \uparrow \tld{w}_1$ (see the proof of Lemma \ref{lemma:bruhatup}), \cite[Proposition 2.8]{iwahori-matsumoto} gives
\[
C_\sigma^\zeta \subset \overline{\cI\backslash \cI (w_0\tld{w}_1)^* \cI t_\omega^*}^{\nabla_0} = \underset{s \in \un{W}}{\cup} \underset{\substack{\tld{w} \in \tld{\un{W}}\\ \tld{w} \uparrow \tld{w}_1}}{\cup} (\cI\backslash \cI (s\tld{w})^* \cI t_\omega^*)^{\nabla_0}.
\]
We will show that $(\cI\backslash \cI (s\tld{w})^* \cI t_\omega^*)^{\nabla_0} \subset C_{\sigma',\obv}^\zeta$ for some $\sigma'$ which $\sigma$ covers. 

Since $(w_0 s^{-1})s\tld{w}$ is a reduced factorization by Lemma \ref{lemma:minrep}, 
\begin{align*}
(\cI\backslash \cI (s\tld{w})^* \cI t_\omega^*)^{\nabla_0} 
&\subset (\cI\backslash \cI (s\tld{w})^* \cI (w_0 s^{-1})^* \cI ((w_0 s^{-1})^{-1})^* t_\omega^*)^{\nabla_0} \\
& = (\cI\backslash \cI (w_0\tld{w})^* \cI (t_\omega s w_0^{-1})^*)^{\nabla_0}.
\end{align*}
To further analyze this, let $\tld{w} = t_\nu \tld{w}_1'$ where $\nu \in X^*(\un{T})$ is dominant and $\tld{w}_1' \in \tld{\un{W}}^+_1$. 
Then $t_{w_0 (\nu)} w_0 \tld{w}_1'$ is a reduced expression for $w_0\tld{w}$ by \cite[Lemma 4.1.9]{LLL}, from which we deduce as before that 
\begin{align*}
(\cI\backslash \cI (w_0\tld{w}_1')^* \cI (t_{\omega+s(\nu)} s w_0^{-1})^*)^{\nabla_0} &= (\cI\backslash \cI (w_0\tld{w}_1')^* \cI (t_\omega s w_0^{-1}t_{w_0(\nu)})^*)^{\nabla_0} \\
&\subset (\cI\backslash \cI (w_0\tld{w})^* \cI (t_\omega s w_0^{-1})^*)^{\nabla_0}. 
\end{align*}
On the other hand, these are irreducible varieties of the same dimension by \cite[Theorem 4.2.4]{MLM} and thus must be equal. 
Letting $\sigma'$ be the Serre weight with lowest alcove presentation $(\tld{w}_1',\omega+s(\nu))$, we have $(\cI\backslash \cI (w_0\tld{w}_1')^* \cI (t_{\omega+s(\nu)} s w_0^{-1})^*)^{\nabla_0} \subset C_{\sigma',\obv}^\zeta$. 
(Note that $\sigma'$ is $(3n-1)$-deep, hence $\mathcal{C}_{\sigma',\obv}$ is defined.)
It suffices to show that $\sigma$ covers $\sigma'$, or by \cite[Proposition 2.3.12(ii)]{MLM} that $t_{\un{W}(\nu)}\tld{w}_1' \uparrow \tld{w}_1$. 
However, we have $t_{\un{W}(\nu)}\tld{w}_1' \uparrow \tld{w} \uparrow \tld{w}_1$ where the first inequality follows from \cite[II.6.5(3)]{RAGS}.
\end{proof}

\begin{prop}\label{prop:obvintersect}
Let $K/\Q_p$ be a finite unramified extension and $\rhobar: G_K \ra \GL_n(\F)$ be a Galois representation. 
Let $\tau$ be a $(5n-1)$-generic tame inertial $L$-parameter. 
Then the following are equivalent. 
\begin{enumerate}
\item \label{item:nonzerodefring} $R_{\rhobar}^\tau$ is nonzero;
\item \label{item:taugeometric}  $\rhobar$ is $4n$-generic and $W^g(\rhobar) \cap \JH(\ovl{\sigma}(\tau)) \neq \emptyset$; and 
\item \label{item:tauobv} $\rhobar$ is $4n$-generic and $W_{\obv}(\rhobar) \cap \JH(\ovl{\sigma}(\tau)) \neq \emptyset$. 
\end{enumerate}
\end{prop}
\begin{proof}
\eqref{item:nonzerodefring} and \eqref{item:taugeometric} are equivalent by \cite[Theorem 7.4.2(1)]{MLM}. 
Since $W_{\obv}(\rhobar) \subset W^g(\rhobar)$, \eqref{item:tauobv} implies \eqref{item:taugeometric}. 
For the converse, suppose that $\rhobar \in \cC_\sigma$ for some $\sigma \in \JH(\ovl{\sigma}(\tau))$. 
Proposition \ref{prop:obvstrat} implies that $\rhobar \in \cC_{\sigma',\obv}$ for some $\sigma'$ which $\sigma$ covers. 
Then $\sigma' \in W_{\obv}(\rhobar)$ by Proposition \ref{prop:obvlocus} and $\sigma' \in \JH(\ovl{\sigma}(\tau))$ by the definition of covering. 
(Note that Propositions \ref{prop:obvlocus}, \ref{prop:obvstrat} apply by the genericity assumption on $\tau$.)
\end{proof}

\clearpage{}%
\clearpage{}%
\section{Some potentially crystalline deformation rings}
\label{sec:PCDR}

The aim of this section is to compute  potentially crystalline deformation rings for a certain class of shapes, namely those related to the subgroup $W_{a,\alpha}\subseteq \tld{\un{W}}$ defined in \ref{sec:comb:weyl}.
We follow the general procedure appearing in \cite{LLL}, improved in \cite{MLM}.

\subsection{The main result on Galois deformation rings}
For a mod $p$ Galois representation $\rhobar$, we write $R^{\eta,\tau}_{\rhobar}$ (resp.~$R^{\leq\eta,\tau}_{\rhobar}$) for the framed universal deformation ring of $\rhobar$ of tame inertial type $\tau$ for $I_K$ over $E$ and parallel Hodge--Tate weights $\eta$ (resp.~$\leq \eta$).
The main result is the following:

\begin{thm}\label{thm:FSM}
Let $\tau$ be a $\max\{(3n-7)e- (n-2), (2n-3)e\}$-generic tame inertial type. 
Suppose that $\tld{w}(\rhobar,\tau)$ is $\tld{w}^{-1} t_{e\eta_0} \tld{w}_\alpha \tld{w}$ for some $\tld{w} \in \un{\tld{W}}^{+}_1$, some $\alpha\in \Delta^\cJ$, and $\tld{w}_{\alpha_j} \in W_{a,\alpha_j}$.

Then $R_\rhobar^{ \eta,\tau}=R_\rhobar^{\leq \eta,\tau}$ is either zero or is a normal domain. 
Furthermore:
\begin{itemize}
\item If $\tld{w}_{\alpha_j}$ is $\id$ or $t_{-e\alpha_j}$ for each $j$, $R^{\eta,\tau}_\rhobar$ is formally smooth over $\cO$.
\item In general, $\Spec \ovl{R}^{\leq \eta,\tau}_\rhobar$ is reduced with $2^{m}$ geometrically irreducible components of the same dimension, where $m=\#\{j\in \cJ \mid \tld{w}_{\alpha_j}\neq \id, t_{-e\alpha_j}\}$.
\end{itemize} 
\end{thm}
\begin{rmk} A key ingredient in our proof of Theorem \ref{thm:FSM} is the fact that the local model (in the sense of \cite{MLM}) of our Galois deformation ring has a Levi reduction property: namely, it is formally smooth over a similar local model attached to a Levi subgroup of $\GL_n$. This turns out to be a general phenomenon whenever the shape $\tld{w}(\rhobar,\tau)^*$ is suitably ``decomposable'', which may be of independent interest. In the specific case of Theorem \ref{thm:FSM}, the Levi subgroup we can reduce to is $\GL_2\times \GL_1^{n-2}$, which is why we have very precise control on the relevant local models, and hence the Galois deformation rings. 
\end{rmk}

\subsection{Gauge bases and parabolic structures}
\label{sub:par:strct}
For each $j\in \cJ$, we set $E_j=\sigma_j(E(v))\in \cO[v]$. Let $R$ be an $\cO$-algebra. We have the usual notion of degrees on $R[v]$, which is submultiplicative $\deg(ab)\leq \deg(a) + \deg(b)$, with equality if either $a$ or $b$ are monic (but not in general). The notion of degree and being monic extends to elements of $R[v,E_j^{-1}]$. The set of elements of degree $\leq 0$ form a subring of $R[v,E_j^{-1}]_{\leq 0}$. This subring contains the set of elements $R[v,E_j^{-1}]_{\leq -1}=R[v,E_j^{-1}]_{< 0}$ of degree $<0$ as an ideal, and another ideal given by $vR[v,E_j^{-1}]_{<0}$. More generally, the set $R[v,E_j^{-1}]_{\leq d}$ of elements of degree $\leq d$ form an $R[v,E_j^{-1}]_{\leq 0}$-module.

Concretely, the elements of $R[v,E_j^{-1}]_{\leq 0}$ are exactly those of the form $\frac{P}{E_j^m}$ with $P\in R[v]$ such that $\deg P\leq me$, with the extra condition $v\mid P$ (for some choice of fractions with $m$ sufficiently large) for elements of $vR[v,E_j^{-1}]_{<0}$, and the extra condition $\deg P< me$ for elements of $R[v,E_j^{-1}]_{< 0}$.
Finally, note that for an element $a$ represented by $\frac{P}{E_j^m}$ with $P(v)\in R[v]$, the $\cO$-algebra generated by the coefficients of $P$ is independent of the choice of representing fraction. 

Let $R$ be a Noetherian $\cO$-algebra. We define
\begin{align*}
 L \cG^{(j)}(R)&\defeq \{A \in \GL_n(R[v]^{\wedge_{E_j}}[\frac{1}{E_j}]), A \textrm{ is upper triangular mod } v\};\\
 L^+ \cM^{(j)}(R) & \defeq \{A \in \mathrm{Mat}_n(R[v]^{\wedge_{E_j}}), A \textrm{ is upper triangular mod } v\};
\end{align*}
For $\tld{z}=z t_\nu\in \tld{W}^\vee$ such that $e\mid ||\nu||$, define $\cU(\tld{z})^{\det,\leq h}(R)$ to be the collection of $  A\in   L \cG^{(j)}(R) $ such that 
\begin{itemize}
\item For $1\leq i, k\leq n$, 
\[A_{ik}=v^{\delta_{i>k}}\frac{P}{E_j^h}\]
with $P\in R[v]$ such that  $\deg P \leq he+\nu_k-\delta_{i>k}-\delta_{i<z(k)}$. Furthermore, this is an equality when $i=z(k)$, in which case $P$ is monic. In particular, $A\in \frac{1}{E_j^h}L^+\cM^{(j)}(R)$.
\item $\det A=\det(z) E_j^{\frac{||\nu||}{e}}$.
\end{itemize}

If $R$ is furthermore $\cO$-flat, then for such $A$ we have 
\[A_{ik}^{-1}=v^{\delta_{i>k}}\frac{Q}{E_j^H}\]
with $Q\in R[v]$ and $H$ sufficiently large, such that $\deg Q \leq He-\nu_i-\delta_{i>k}-\delta_{z(i)<k}$
(The condition that $R$ is $\cO$-flat is used to show that divisibility by $v$ in $R[v,E_j^{-1}]$ is equivalent to evaluating to $0$ at $v=0$, and hence the numerators of all representing fractions have $0$ constant terms).

For each $j\in \cJ$, we define $U^{[a,b]}(\tld{z}_j)\subset \cU^{\det, \leq -a}(\tld{z}_j)$ to be the subfunctor consisting of $A$ such that $E_j^b A^{-1} \in L^+\cM^{(j)}(R)\cap L\cG^{(j)}(R)$ and $E_j^{-a}A\in L^+\cM^{(j)}(R)\cap L\cG^{(j)}(R)$. This is clearly representable by a finite type affine $\cO$-scheme, with a set of generators given by the coefficients of the entries of $A^{(j)}$.  
Note that this depends on $j$, a choice that is implicit in the symbol $\tld{z}_j$.

If $\tld{z}=(\tld{z}_j)_j\in \tld{W}^{\vee,\cJ}$, we set $U^{[a,b]}(\tld{z})=\prod U^{[a,b]}(\tld{z}_j)$.
We have the following definition:
\begin{defn}
Let $(R,\fm)$ be a complete local Noetherian $\cO$-algebra and assume that $\fM\in Y^{[0,n-1],\tau}(R)$ such that $\fM\otimes_{R}R/\fm$ has shape $\tld{z}$ with respect to $\tau$.
An eigenbasis $\beta$ for $\fM$ is said to be a \emph{gauge basis} if $A^{(j)}_{\fM,\beta}\in T^\vee(R)U^{[0,n-1]}(\tld{z}_j)(R)$ for all $j\in \cJ$.
\end{defn}

\begin{prop}
\label{prop:gauge:basis}
Assume that $\tau$ admits a $(e(n-1)+1)$-deep lowest alcove presentation. Suppose $R$ is a complete local Noetherian $\cO$-algebra and let $\fM\in Y^{[0,n-1],\tau}(R)$ such that $\ovl{\fM}\in Y^{[0,n-1],\tau}(\F)$ has shape $\tld{z}$ with respect to $\tau$.
Then $\fM$ has a gauge basis.
Moreover the set of gauge basis for $\fM$ is a torsor for the natural action of $T^{\vee,\cJ}(R)$.
\end{prop}
\begin{proof}
The proof of \cite[Proposition 5.2.7]{MLM} generalizes verbatim by replacing the reference to Proposition 5.1.8 in \emph{loc.~cit.}~by Remark \ref{rmk:cng:basis} above, and noting that the statement of Lemma 5.1.10 in \emph{loc.~cit}.~holds true in our setting. Note the proof in \emph{loc.~cit}.~in fact proves a more general statement where $R$ is only assumed to be merely $p$-adically complete.
\end{proof}

Suppose we are given a gauge basis $\ovl{\beta}$ for $\ovl{\fM}\in Y^{[0,n-1],\tau}(\F)$ with shape $\tld{z}$ and write
\[
A^{(j)}_{\ovl{\fM},\ovl{\beta}}=\ovl{D}^{(j)}\ovl{U}^{(j)}
\]
where $\ovl{D}^{(j)}\in T^\vee(\F)$,  $\ovl{U}^{(j)}\in U^{[0,n-1]}(\tld{z}_j)(\F)$.

If $R$ is a complete local Noetherian $\cO$-algebra, and $\fM\in Y^{[0,n-1],\tau}(R)$ is such that $\fM\otimes_R\F\cong \ovl{\fM}$, then the set of gauge basis for $\fM$ lifting $\ovl{\beta}$ is a torsor under the natural action of $\ker\big(T^{\vee,\cJ}(R)\onto T^{\vee,\cJ}(\F)\big)$.
Thus, the functor representing deformations $(\fM,\beta)$ of the pair $(\ovl{\fM},\ovl{\beta})$ is representable by the completion of $T^{\vee}U^{[0,n-1]}(\tld{z}_j)$ at the point corresponding to $(\ovl{D}^{(j)}\ovl{U}^{(j)})$, and it is formally smooth over the completion of $Y^{[0,n-1],\tau}$ at $\ovl{\fM}$. The subfunctor classifying deformations $(\fM,\beta)$ such that $\fM$ furthermore belongs to $Y^{\leq \eta,\tau}$ correspond to the completion of the closed subscheme $T^\vee U(\tld{z},\leq \eta)$ of $T^{\vee} U^{[0,n-1]}(\tld{z})$ characterized by:
\begin{itemize}
\item $T^\vee U(\tld{z},\leq \eta)$ is $\cO$-flat and reduced.
\item The elementary divisors of $(A^{(j)})\in T^\vee U(\tld{z},\leq \eta)(R)\subset \prod L\cG^{(j)}(R)$ are bounded by $E_j^{(n-1,\cdots, 0)}$, i.e.~for each $1\leq k\leq n$, each $k\times k$ minors of $A^{(j)}$ (which belong to $R[v]$) are divisible by $E_j^{\frac{(k-1)k}{2}}$ (in $R[v]$). 
\end{itemize}

\begin{rmk}
\label{rmk:twisted:LM}
Let 
\[ L \cG^{+,(j)}(R) \defeq \{A \in \GL_n(R[v]^{\wedge_{E_j}}), A \textrm{ is upper triangular mod } v\}\]
a twisted positive loop group. Then $\Gr_{\cG}^{(j)}=L \cG^{+,(j)}\backslash L \cG^{(j)}$ is a twisted affine Grassmannian. Then the generic fiber $\Gr_{\cG,E}^{(j)}\cong (\Gr_{\GL_n,E})^e$ identifies with the product of $e$ copies of the affine Grassmannian for the split group $\GL_n$, while the special fiber $\Gr_{\cG,\F}^{(j)}\cong \Fl$ identifies with the affine flag variety. The Pappas--Zhu local model $M_{j}(\leq\mkern-4mu\eta)$ for $\Res_{\cO_K \otimes_{W(k), \sigma_j} \cO/\cO} \GL_n$ as defined in \cite{LevinLM} is the Zariski closure of the open Schubert variety for the cocharacter $(n-1,n-2,\dots,1,0)$ for each copy of $\Gr_{\GL_n,E}$.  %
In this setup, the scheme $U(\tld{z}_j,\leq\mkern-4mu\eta)$ identifies with an (possibly empty) open affine subscheme of $M_j(\leq\mkern-4mu\eta)$, cf.~the discussion preceding \cite[Theorem 5.3.3]{MLM}. In particular, if non-empty, $U(\tld{z}_j,\leq\mkern-4mu\eta)$ has dimension $e\sum_{\beta>0} \langle \eta_0, \beta^{\vee} \rangle=e\frac{(n-1)n(n+1)}{6}$.
\end{rmk}
The following Proposition shows that in certain cases, any element of $T^\vee U(\tld{z},\leq \eta)$ automatically acquires a parabolic structure.  In Propositions \ref{prop:parabolic_general} and \ref{prop:Levi_reduction}, we work with fixed $j \in \cJ$  and drop the subscript for notational ease. 
\begin{prop}\label{prop:parabolic_general} Let $w\in W^{\vee}$, $r+s=n$ and $\tld{z}=zt_{\nu}=\begin{pmatrix} \tld{z}_t & 0 \\ 0 &\tld{z}_b \end{pmatrix}\in \tld{W}^{\vee}$ with block sizes $r,s$. 

Let $w=w_Mw^M$ be the factorization so that $w^M$ has minimal length and $w_M= (w_t,w_b)\in W(M)=W(\GL_r)\times W(\GL_s)$ (where $M$ is the standard Levi for the partition $r+s=n$). Assume that
\begin{itemize}
\item $w_{b}^{-1}\tld{z}_{b} w_{b}$ has elementary divisors bounded by $v^{e(s-1,\cdots 0)}$.
\item $v^{-es}w_{t}^{-1}\tld{z}_{t} w_{t}$ has elementary divisors bounded by $v^{e(r-1,\cdots 0)}$. 
\end{itemize}
Suppose $R$ is a $\cO$-flat algebra and $A\in T^{\vee}U(w^{-1}\tld{z}w,\leq \eta)(R)$. Then $A=Dw^{-1}Pw$ with $D\in T^{\vee}(R)$ and 
\[P=\begin{pmatrix} M_{t} & 0 \\ X & M_{b} \end{pmatrix}\]
is parabolic with diagonal block sizes $r,s$, and furthermore:
\begin{enumerate}
\item $M_{t}\in E_j^sw_{t}U(t_{-se(1,\cdots,1)}w_{t}^{-1}\tld{z}_{t} w_{t})w_{t}^{-1}$ and has elementary divisors bounded by $E_j^{(n-1,\cdots s)}$.
\item $M_{b}\in w_{b}U(w_{b}^{-1}\tld{z}_{b} w_{b})w_{b}^{-1}$ and has elementary divisors bounded by $E_j^{(r-1,\cdots 0)}$.
\item $(XM_t^{-1})_{ik} \in v^{\delta_{w^{-1}(i)>w^{-1}(k)}}R[v,E_j^{-1}]\cap R[v,E_j^{-1}]_{\leq -\delta_{w^{-1}(i)<w^{-1}(k)}}$.
\item \label{item:fh_integrality}$(M_b^{-1}X)_{ik}\in v^{\delta_{w^{-1}(i)>w^{-1}(k)}}R[v]\cap R[v,E_j^{-1}]_{\leq \nu_k-\nu_i-\delta_{w^{-1}z(i)<w^{-1}z(k)}}$.
(In the last two items, we interpret the indices to run over the rows and columns of $X$ as a submatrix of $P$, i.e.~$r+1 \leq i\leq n$, $1\leq k\leq r$.) 
\end{enumerate}
\end{prop}
\begin{proof} We write $A=Dw^{-1}Pw$ so that $P\in wU(w^{-1}\tld{z}w,\leq \eta)w^{-1}$. This means that $P$ has entries in $R[v]$, with the degree bounds
\[P_{ik}\in v^{\delta_{w^{-1}(i)>w^{-1}(k)}}R[v]\cap R[v,E_j^{-1}]_{\leq\nu_k-\delta_{w^{-1}(i)<w^{-1}z(k)}}.\]
and that the leading coefficient of $P_{iz(k)}$ are $1$. We call the corresponding entry the pivot entries.

Write $P=\begin{pmatrix}M_t & Y \\ X & M_b \end{pmatrix}$.
We first show that $Y=0$. The degree bounds on $P$ imply that when expanding $\det M_b$, there is a unique maximal degree term, which is given by the product of the top degree terms in the pivot entries in $M_t$
(one can see this by noting that this is a combinatorial statement on the degree bounds which can be checked over rings $S$ where $p=0$, where it reduces to the fact that $M_bv^{-\nu_b}z_b^{-1}$ is conjugate to a matrix in with coefficients in $S[v^{-1}]$ which is upper triangular unipotent mod $v^{-1}S[v^{-1}]$).
This shows $\det M_b=\det z_{b}E_j^{\frac{(s-1)s}{2}}$.
Now 
\[(YM^{-1}_b)_{ik}=\sum_l Y_{il}(M^{-1}_b)_{lk}.  \]
We observe
\begin{itemize}
\item $Y_{il}\in v^{\delta_{w^{-1}(i)>w^{-1}(l)}}$, $(M^{-1}_b)_{lk}\in v^{\delta_{w^{-1}(l)>w^{-1}(k)}}$. Hence $(YM^{-1}_b)_{ik}$ is divisible by $v^{\delta_{w^{-1}(i)>w^{-1}(k)}}$ in $R[v,E_j^{-1}]$.
\item $Y_{il}\in R[v,E_j^{-1}]_{\leq \nu_l-\delta_{w^{-1}(i)<w^{-1}z(l)}}$, $(M^{-1}_b)_{lk}\in R[v,E_j^{-1}]_{\leq -\nu_l-\delta_{w^{-1}z(l)<w^{-1}(k)}}$. Hence $(YM^{-1}_b)_{ik}\in R[v,E_j^{-1}]_{\leq -\delta_{w^{-1}(i)<w^{-1}(k)}}$.
\end{itemize}
However, the elementary divisor conditions together with the degree bounds imply that the minor formed by replacing one row of $M_b$ with one row of $Y$ belongs to $E_j^{\frac{(s-1)s}{2}}R$, hence Cramer's rule shows that the entries of $YM^{-1}_b$ are in $R$. Since by the above, these entries also belong to $v^{\delta_{w^{-1}(i)<w^{-1}(k)}}R[v,E_j^{-1}]\cap R[v,E_j^{-1}]_{\leq -\delta_{w^{-1}(i)<w^{-1}(k)}}$, they must be all $0$.

Thus, we see that $P$ has the desired parabolic structure. The first two items immediately follow from the degree bounds on $P$ and the elementary divisor conditions. The third and fourth items follow from the same argument used above in showing $Y=0$. 
\end{proof}
By applying Proposition \ref{prop:parabolic_general} to the universal case, we get%
\begin{prop}\label{prop:Levi_reduction} Assume the setting of Proposition \ref{prop:parabolic_general}. Let $R^{\univ}=\cO(U(w^{-1}\tld{z}w,\leq \eta))$, so that the universal $A^{\univ}\in U(\tld{z},\leq \eta))$ factors as
\[A^{\univ}=D^{\univ}w^{-1}\begin{pmatrix} M_t^{\univ}  &0 \\ X^{\univ} & M_b^{\univ} \end{pmatrix} w.\]
Then the map $A^{\univ}\mapsto (\frac{1}{E_j^s}w_t^{-1}M_t^{\univ}w_t,w_b^{-1}M_bw_b)$ exhibits
$U(w^{-1}\tld{z}w,\leq \eta)$ as an affine space over $U(t_{-se(1,\cdots ,1)}w^{-1}_t\tld{z}_tw_t,\leq (r-1,\cdots 0))\times U(w^{-1}_b\tld{z}_bw_b,\leq (s-1,\cdots 0))$, whose coordinates are the coefficients of the entries of $(M_b^{\univ})^{-1}X^{\univ}$ (which are subject to the degree bounds dictated by Proposition \ref{prop:parabolic_general}). 
\end{prop}
\begin{proof} The fact that we get a map follows from Proposition \ref{prop:parabolic_general}, which clearly induces a closed immersion from $U(w^{-1}\tld{z}w,\leq \eta)$ into the appropriate affine space over $U(t_{-se(1,\cdots ,1)}w^{-1}_t\tld{z}_tw_t,\leq (r-1,\cdots 0))\times U(w^{-1}_b\tld{z}_bw_b,\leq (s-1,\cdots 0))$. To see this injection is an isomorphism, observe that if we set $Z$ to be a matrix subject to the degree bounds of Proposition \ref{prop:parabolic_general}(\ref{item:fh_integrality}) and whose coefficients are free variables, then
\[w^{-1}\begin{pmatrix} M_t^{\univ} & 0 \\ M_b^{\univ}Z & M^{\univ}_b\end{pmatrix}w=w^{-1}\begin{pmatrix} M^{\univ}_t & 0 \\ 0 & M^{\univ}_b \end{pmatrix}\begin{pmatrix} 1 & 0 \\ Z &0 \end{pmatrix}w\]
satisfies the necessary elementary divisors and degree bounds characterizing $U(w^{-1}\tld{z}w,\leq \eta)$. 
\end{proof}

\subsection{Interlude: $\GL_2$ Pappas--Zhu models}
We specialize the previous section to $n=2$. Thus, $M_j(t_{(1,0)})$ is a Pappas--Zhu local model for the Weil restricted group $\Res_{\cO_K \otimes_{W(k), \sigma_j} \cO/\cO} \GL_2$, the (minuscule) cocharacter $(t_{(1,0)},\cdots t_{(1,0)})\in (\Z^2)^{e}$, and Iwahori level structure. 
The following summarizes the known geometric properties of $M_j(t_{(1,0)})$  (see Theorem A in \cite{PR2} or Theorem 2.3.3 and 2.3.5 in \cite{LevinLM}): %
\begin{prop} 
\begin{enumerate}
\item $M_j(t_{(1,0)})_E\cong (\mathbb{P}^1_E)^e$.
\item $M_j(t_{(1,0)})_\F$ is (geometrically) reduced, and identifies with the reduced union of $S(t_{(e,0)})\cup S(t_{(0,e)})$ of $\Fl=\Gr^{(j)}_\F$. Each of its irreducible components are (geometrically) normal.
\end{enumerate}
In particular, $M_j(t_{(1,0)})$ is a normal domain, whose special fiber has two irreducible components. Further more any $x\in S^0(t_{(e,0)})\cup S^0(t_{(0,e)})$ belongs to the regular locus of $M_j(t_{(1,0)})$.
\end{prop}  
Note that the reducedness of the special fiber and the geometric normality of its irreducible component are preserved under taking products. 

We note that the $e(1,0)$-admissible elements are exactly $t_{(e-k,k)}$, $0\leq k \leq e$ and $t_{(e-k,k)}s_\alpha$ with $0<k\leq e$.
\begin{cor}\label{cor:GL_2model} Let $\tld{z}_j$ be $(e,0)$-admissible. Then $U(\tld{z}_j,\leq \eta)$ is a normal domain, and it is formally smooth over $\cO$ if $\tld{z}_j\in \{t_{(e,0)},t_{(0,e)}\}$.
Otherwise, its special fiber has two (geometrically) normal irreducible components. 
\end{cor}
We deduce the following combinatorial property about the admissible set from our geometric considerations:
\begin{cor}\label{cor:at_most_two} Let $\tld{w}_j= w_j^{-1} W_{a,\alpha_j} t_{e\eta_0} w_j \cap \Adm(e\eta_0)$ for some simple root $\alpha_j$ and $w_j\in W$. Then there are at most two $\sigma\in W$ such that $\tld{w}_j\leq t_{\sigma^{-1}(e\eta_0)}$.
\end{cor}
\begin{proof} Set $\tld{z}_j=\tld{w}_j^*$. By Proposition \ref{prop:Levi_reduction}, $U(\tld{z}_j,\leq \eta)$ is an affine space over an affine scheme of the form as in Corollary \ref{cor:GL_2model}. In particular, $U(\tld{z}_j,\leq \eta)_\F$ has at most two irreducible components. On the other hand, this is an open neighborhood of $\tld{z}_j$ in the special fiber $M_j(\leq \eta)_\F$ of a Pappas--Zhu model. We conclude from the fact that $M_j(\leq \eta)_\F=\cup_{\sigma\in W} S(t_{\sigma(e\eta_0)})$.
\end{proof}
\subsection{Analysis of the monodromy condition}
\label{sub:analysis:MC}
Suppose $\rhobar$ admits a Breuil--Kisin module $\ovl{\fM}\in Y^{[0,n-1],\tau}(\F)$ of type $\tau$, with shape $\tld{z}$ and a gauge basis $\ovl{\beta}$.
To analyze the potentially crystalline deformation ring $R^{\leq \eta,\tau}_{\rhobar}$, we need to recall its relationship with the finite height deformation ring $R^{\tau,\ovl{\beta}}_{\ovl{\fM}}$, as in \cite[\S 3,4]{LLL} and \cite[\S 7.1]{MLM}. One has a diagram (cf \cite[Diagram (3.16)]{LLL}, \cite[Proposition 7.2.3]{MLM})
\begin{equation} 
\label{diag:main}
\xymatrix{  &  \mathrm{Spf} R^{\tau,\ovl{\beta},\Box,\nabla}_{\ovl{\fM}} \ar[r]^{f.s.} \ar@{^{(}->}[d]&  \mathrm{Spf} R^{\tau,\ovl{\beta},\nabla}_{\ovl{\fM}}\ar@{^{(}->}[d] \\
 \mathrm{Spf} R^{\tau,\ovl{\beta},\Box}_{\ovl{\fM},\rhobar}\ar[d]^{f.s.} \ar@{^{(}->}[r]\ar[ur]^{\cong} &\mathrm{Spf} R^{\tau,\ovl{\beta},\Box}_{\ovl{\fM}}\ar[r]^{f.s.} &\mathrm{Spf} R^{\tau,\ovl{\beta}}_{\ovl{\fM}} \\
 \mathrm{Spf} R^{\leq \eta,\tau}_\rhobar \\&&  }
\end{equation}
where
\begin{itemize}
\item $R^{\leq \eta,\tau}_\rhobar$ is the framed potentially crystalline deformation ring representing Galois deformations $\rho$ with Hodge-Tate weights $\leq \eta$ and inertial type $\tau$. Note that it is either zero, or is $\cO$-flat, reduced and of Krull dimension $n^2+1+\frac{n(n-1)}{2}[K:\Q_p]$.%
\item $R^{\tau,\ovl{\beta}}_{\ovl{\fM}}$  represents deformations $(\fM,\beta)$ of $(\ovl{\fM},\ovl{\beta})$ where $\fM$ belongs to $Y^{\leq \eta,\tau}$ and $\beta$ is a gauge basis of $\fM$. 
\item $R^{\tau,\ovl{\beta},\Box}_{\ovl{\fM},\rhobar}$ represents potentially crystalline Galois deformations $\rho$ of type $(\leq \eta,\tau)$, together with a gauge basis $\beta$ of its (unique) Breuil--Kisin module $\fM$ in $Y^{\leq \eta,\tau}$. It is formally smooth over $R^{\leq \eta,\tau}$ of relative dimension $nf$.
\item $R^{\tau,\ovl{\beta},\Box}_{\ovl{\fM}}$ represents a deformation $(\fM,\beta)$ of $(\ovl{\fM},\ovl{\beta})$ as above together with a framing basis of the $G_{K_\infty}$-representation associated to $\fM$. This is formally smooth over $R^{\tau,\ovl{\beta}}_{\ovl{\fM}}$ of relative dimension $n^2$.
\item $R^{\tau,\ovl{\beta},\nabla}_{\ovl{\fM}}$ (resp. $R^{\tau,\ovl{\beta},\Box, \nabla}_{\ovl{\fM}}$) is the $\cO$-flat reduced quotient of $R^{\tau,\ovl{\beta}}_{\ovl{\fM}}$ (resp. $R^{\tau,\ovl{\beta},\Box}_{\ovl{\fM}}$) cut out by imposing the monodromy condition on the universal Breuil--Kisin module after inverting $p$.
\end{itemize}

We elaborate on the monodromy condition on the universal Breuil--Kisin module on $R^{\tau,\ovl{\beta}}_{\ovl{\fM}}$. Recall that $E(v)$ is the Eisenstein polynomial of a chosen uniformizer of $K$ over $K_0$, and that $e'=p^{f'}-1=p^{fr}-1$.
Recall from \cite[\S 7.1]{MLM} the ring $\cO^{\rig}\defeq \cO^{\rig}_{K',R^{\tau,\ovl{\beta}}_{\ovl{\fM}}}$, endowed with a canonical derivation $N_\nabla=-u'\lambda\frac{d}{du'}$ (where $\lambda=\prod_{i=0}^\infty \frac{\phz^i(E((u')^{e'}))}{E(0)}$ is constructed out of $E(v)=E((u')^{e'})$ instead of $v+p$), and the module $\fM^{\univ, \rig}\defeq \fM^{\univ}\otimes_{R^{\tau,\ovl{\beta}}_{\ovl{\fM}}} \cO^{\rig}$, such that $\fM^{\univ, \rig}[1/\lambda]$ is endowed with a canonical derivation $N_{\fM^{\univ, \rig}}$ over $N_\nabla$ (cf.~\cite[Proposition 7.1.3(1)]{MLM}). Then the \emph{monodromy condition} alluded to above is the condition that $N_{\fM^{\univ, \rig}}$ preserves $\fM^{\univ, \rig}$.

We now choose a lowest alcove presentation $\tau\cong \tau(s,\mu+\eta_0)$.
Recall from \S \ref{subsubsec:TIT} that attached to $(s,\mu)$ we have the data $s'_{\mathrm{or}, j'}\in W$, $\mathbf{a}^{\prime \, (j')}\in \Z^{n}$. We write $A^{(j')}$ for the matrices constructed out of the universal Breuil--Kisin module and its universal gauge basis over $R^{\tau,\ovl{\beta}}_{\ovl{\fM}}$ or $R^{\tau,\ovl{\beta},\Box}_{\ovl{\fM}}$.
We get the following control of the monodromy condition:
\begin{prop}\label{prop:monodromy_control} Assume $\tau(s,\mu+\eta_0)$ is an $m$-deep lowest alcove presentation of $\tau$. If the monodromy condition holds, the for each $j'$, $0\leq t< n-2$ and $\pi$ a root of $E_j$, the result of the operator $(\frac{d}{dv})^t|_{v=\pi}$ acting on
\[\left(e'v \frac{d}{dv} A^{(j')} + [A^{(j')},\mathrm{Diag}((s'_{\mathrm{or}, j'})^{-1}(\mathbf{a}^{\prime \, (j')}))] \right) (A^{(j')})^{-1}E_j^{n-1}\]
belongs to $p^{\frac{m+1-(n-2)e-t}{e}}R$.

\end{prop}

\begin{proof} This is a straightforward generalization of the computation in \cite[Proposition 7.1.10]{MLM}, with the following changes: $h$ in \emph{loc.cit.} becomes $n-1$, occurrences of $p$ (outside any evaluation at $v=-p$) becomes $E(0)$, occurrences of $(v+p)^h(A^{(j')})^{-1}$ becomes $E_{j'}^{n-1}(A^{(j')})^{-1}$, occurrences of $|_{v=-p}$ becomes $|_{v=\pi}$. Note that $E_j(0)\in p\cO^\times$, and $E_{j'}=E_j$ depends only on $j$ mod $f$.
More specifically, the computation in \emph{loc.cit.} expresses the monodromy condition as 
\[\phz(\lambda)^{n-1}\left(e'v \frac{d}{dv} A^{(j')} + [A^{(j')},\mathrm{Diag}((s'_{\mathrm{or}, j'})^{-1}(\mathbf{a}^{\prime \, (j')}))] \right) (A^{(j')})^{-1}E_j^{n-1}+Err\]
has zeroes of order $n-2$ along the roots of $E_j$, for an appropriate error term $Err$. It follows that the operator $(\frac{d}{dv})^t|_{v=\pi}$ annihilates this expression, for $0\leq t<n-2$ and $\pi$ a root of $E_j$. 

The error term $Err$ has the form
\[\sum_{i=1}^\infty \phz^{i+1}(\lambda)^{n-1}Z_i^{(j')}\]
with $Z_i^{(j')}\in \frac{1}{p^{i(n-2)}}v^{1+m\frac{p^i-1}{p-1}}\mathrm{Mat}_n(R[\![v]\!])$. We conclude from the analysis of the effect of $(\frac{d}{dv})^t|_{v=\pi}$ on the error term $Err$ as in \emph{loc.cit.} (except that we use the differential operator $\frac{d}{dv}$ as opposed to $v\frac{d}{dv}$), noting that in our current situation
\begin{itemize}
\item $(\frac{d}{dv})^t|_{v=\pi}\phz^k(\lambda)\in p^{1+\frac{p-t}{e}}\cO+ p^{p-\frac{t}{e}}\cO$ for any $t,k\geq 1$.
\item If $F\in v^M \mathrm{Mat}_n(R[\![v]\!])$ then $(\frac{d}{dv})^t|_{v=\pi}F\in p^{\frac{M-t}{e}} \mathrm{Mat}_n(R[\![v]\!])$.
\end{itemize} 
\end{proof}

\begin{lemma}\label{lem:error_control} Let $R$ be a $p$-flat $\cO$-algebra. Let $N,k$ be non-negative integers and $F\in R[v]$. Assume that $N<p$ and 
 that $(\frac{d}{dv})^t|_{v=\pi}(F)\in p^kR$ for $0\leq t<N$ and $\pi$ is any root of $E_j$.
Write $F=E_j^Nq+r$ where $q,r\in R[v]$ such that $\deg r<Ne$ (this uniquely determines $q$, $r$)
Then $r\in p^{k-(2N-1)(1-\frac{1}{e})}[v]$
\end{lemma}
\begin{proof} Our hypothesis implies $(\frac{d}{dv})^t|_{v=\pi}(r)\in p^kR$ for $t<N$ and $E_j(\pi)=0$. We decompose $r=\sum_{t=0}^{N-1}E_j^t r_t$ with $\deg r_i<e$. Then the reduction to $R/p^k$ of the coefficients of $r_0$ form an element in the kernel space of the Vandermonde matrix on the roots of $E_j$. It follows that $r_0\in p^{k-\frac{e-1}{e}}R[v]$.
For $t\geq 0$, Then $t!(E'_j)^t(\pi)r_t(\pi)$ differs from $(\frac{d}{dv})^t|_{v=\pi}(r)$ by a polynomial in the coefficients of $r_{t'}$ for $t'<t$. This implies $r_t\in p^{k-(2t+1)(1-\frac{1}{e})}R[v]$ by induction on $t$.
\end{proof}
The following Lemma studies the effect of the approximation of the monodromy condition under the presence of a suitable parabolic structure: 
\begin{lemma}\label{lem:parabolic_monodromy} Let $R$ be a Noetherian $\cO$-algebra, $N$,$r,s$ non-negative integers such that $r+s=n$. Let $\kappa=\begin{pmatrix} \kappa_{t} & 0 \\ 0 & \kappa_{b} \end{pmatrix}\in X^*(T)\otimes \cO$ viewed as a constant diagonal matrix, $w\in W^{\vee}$, and $\tld{z}=zt_{\nu}=\begin{pmatrix} \tld{z}_{t} & 0 \\ 0 & \tld{z}_{b}\end{pmatrix}$ (with block sizes $r, s$). %
Suppose we are also given $P=\begin{pmatrix} A & 0 \\ C & D \end{pmatrix}\in \mathrm{Mat}_n(R[v])$, a block lower triangular matrix corresponding to the partition $r+s=n$ satisfying
\begin{equation} \label{eqn:parabolic_monodromy}
(v\frac{d}{dv} P - [P,\kappa]) P^{-1}\in \frac{1}{E_j} \mathrm{Mat}_n(R[v])
\end{equation}
Assume the following
\begin{enumerate}
\item $\begin{pmatrix} A & 0 \\ 0 & D \end{pmatrix}\in wU(w^{-1}\tld{z}w) w^{-1}$.
\item For $\beta\in \Phi$, the $\beta$-th entry of $CA^{-1}$ (inserted inside $\mathrm{Mat}_n$ at the same position as $C$) belong to $v^{\delta_{w^{-1}(\beta)<0}}R[v,E_j^{-1}]_{<0}$.
\item $E_j^NP^{-1}\in \mathrm{Mat}_n(R[v])$.
\item $\langle \tld{z}(0)-z(\kappa), \beta^{\vee}\rangle+k\in \cO^{\times}$ for all $k\in \{0,\cdots,-Ne\}$. 
\end{enumerate}
Let $\cO_P$ be the $\cO$-algebra generated by the coefficients of the entries of $P$, and $\cO_{A,D}$ be the $\cO$-algebra generated by the coefficients of the entries of $A, D$. Then $\cO_{P}$ is generated over $\cO_{A,D}$ by at most $ers$ elements.
\end{lemma}
\begin{proof}
In this proof only, we abbreviate $\delta_{\beta}=\delta_{w^{-1}(\beta)<0}$, to avoid cluttering notation.

Our hypothesis on $\begin{pmatrix} A & 0 \\ 0 & D \end{pmatrix}$ implies
\[w^{-1}\begin{pmatrix} (v\frac{d}{dv} A - [A,\kappa_{t}]) A^{-1} & 0 \\ 0 &(v\frac{d}{dv} D - [D,\kappa_{b}]) D^{-1}  \end{pmatrix} w\in \mathrm{Mat}_n(R[v,E_j^{-1}]_{\leq 0}),\]
and whose entries above the diagonal are in $R[v,E_j^{-1}]_{<0}$ and whose entries on and below the diagonal are in $vR[v,E_j^{-1}]_{<0}$. Furthermore, modulo $R[v,E_j^{-1}]_{<0}$, the diagonal part is exactly $\frac{vE'_j}{E_j}\Ad(z)\frac{\nu}{e} +(1-\Ad(z))(\kappa)$.

Now
\[(v\frac{d}{dv} P - [P,\kappa]) P^{-1}= \begin{pmatrix} v\frac{d}{dv} A - [A,\kappa_{t}] & 0 \\ v\frac{d}{dv} C - C\kappa_{t}+\kappa_{b}C& v\frac{d}{dv} D - [D,\kappa_{b}] \end{pmatrix}\begin{pmatrix} A^{-1} & 0 \\ -D^{-1}CA^{-1}& D^{-1}\end{pmatrix}\]
Set $B=CA^{-1}$, then the bottom left block of the above expression is 
\begin{align*}
&(v\frac{d}{dv} (BA) - BA\kappa_{t}+\kappa_{b}BA)A^{-1} -(v\frac{d}{dv} D - [D,\kappa_{b}])D^{-1}B \\
=&  v\frac{d}{dv} B - B\kappa_{t}+\kappa_{b}B+B(v\frac{d}{dv} A - [A,\kappa_t])A^{-1} -(v\frac{d}{dv} D - [D,\kappa_{b}])D^{-1}B 
\end{align*}
We abbreviate $\nabla B= v\frac{d}{dv} B - B\kappa_{t}+\kappa_{b}B$, $\nabla A= (v\frac{d}{dv} A - [A,\kappa])A^{-1} $ and $\nabla D=(v\frac{d}{dv} D - [D,\kappa_{b}])D^{-1}$.

In what follows, we label the entries of various matrices of size smaller than $n\times n$ using roots/indices of the $n\times n$ matrix $P$, by interpreting such matrices as one of the non-trivial block of $P$ corresponding to its size.    
We observe:
\begin{itemize}
\item
$E_j\nabla(A)=A_{e-1}+\cdots A_0$, where $A_{i,\beta}=v^{\delta_{\beta}+i}a_{i,\beta}$, $A_{i,ll}=v^{1+i}a_{i,ll}$ with $a_{i,\beta},a_{i,ll}\in R$, for all $\beta$ and $l$ such that the relevant entry exists in $A$.
\item
$E_j\nabla(D)=D_{e-1}+\cdots D_0$, where $D_{i,\beta}=v^{\delta_{\beta}+i}d_{i,\beta}$, $D_{i,ll}=v^{1+i}d_{i,ll}$ with $d_{i,\beta},a_{d,ll}\in R$, for all $\beta$ and $l$ such that the relevant entry exists in $D$.
\item The matrices $A_{lead}$, $D_{lead}$ obtained by extracting the degree $e$ coefficients of $E_j\nabla(A)$, $E_j\nabla(D)$ satisfy 
\[w^{-1} \begin{pmatrix} A_{lead} &0 \\ 0& D_{lead}\end{pmatrix}w \]
is lower triangular, with diagonal entries $\Ad(w^{-1}) (z(\nu)+\kappa-z(\kappa))$. %
\item $B=\frac{1}{E_j^N}(B_0+B_{-1}+\cdots)$ where $B_{i,\beta}=b_{i,\beta}v^{\delta_\beta+Ne-1+i}$ with $b_{i,\beta}\in R$, and $b_{i,\beta}=0$ if $i<-Ne$.
\end{itemize}

Condition (\ref{eqn:parabolic_monodromy}) means
\[\nabla B +B(\nabla A)-(\nabla D )B\in \frac{1}{E_j} \mathrm{Mat}_n(R[v]).\]
Using $v\frac{d}{dv}(\frac{F}{E_j^N})=-N\frac{vE'_j}{E_j^{N+1}}F+\frac{v}{E_j^N}\frac{dF}{dv}$, clearing denominators in the above expression yields
\begin{equation}\label{eqn:key}
-NvE'_j(E_j^NB)+vE_j\frac{d}{dv}(E_j^NB)-E_j(E_j^NB)\kappa_{t}+E_j\kappa_{b}(E_j^NB)+(E_j^NB)(E_j\nabla A)-(E_j\nabla D)(E_j^NB)=E_j^{N}X
\end{equation}
for some $X\in \mathrm{Mat}_n(R[v])$.

 The observations on the degree ranges of $A_i,D_i,B_i$ show that for each relevant $\beta\in \Phi$, $X_{\beta}=v^{\delta_\beta}\sum_{i\geq 0} x_{i,\beta}v^i$ (recall that an element of $R[v][\frac{1}{E_j}]$ is divisible by $v$ if and only if its evaluation at $v=0$ is $0$, a condition that makes sense because $R\subset R[\frac{1}{p}]$).

The degree $Ne+e-1+i+\delta_{\beta}$ part of the $\beta$-th entry of equation (\ref{eqn:key}) reads
\begin{align}\label{eqn:key:1}
&-Neb_{i,\beta}+(\delta_{\beta}+Ne-1+i)b_{i,\beta}+\langle \kappa, \beta^\vee \rangle b_{i,\beta}+O(>i,\beta)+\\
&
\qquad+\sum_{\beta=\beta'+\beta''}\sum_{k,l}b_{k,\beta'}a_{l,\beta''}+\sum_{\beta=\gamma'+\gamma''}\sum_{k',l'}d_{l',\gamma'}b_{k',\gamma''}
=X_{e-1+i,\beta}+O(>e-1+i,\beta)\nonumber
\end{align}
where
\begin{itemize}
\item
The symbol $O(>i,\beta)$ (resp. $O(>e-1+i,\beta)$) stands for a polynomial with $\cO$-coefficients in $b_{i',\beta}$ (resp. $X_{e-1+i',\beta}$) for $i'>i$. 
\item
The decompositions $\beta=\beta'+\beta''$ runs over decompositions in $\Phi$, with the added possibility that $\beta''=0$, in which case $a_{l,\beta''}$ is interpreted as the unique diagonal term $a_{l,tt}$ that contributes to the $\beta$-entry of the matrix product. A similar remark applies to $\beta=\gamma'+\gamma''$.
\item The pairs $k,l$ and $k',l'$ are constrained by 
\[Ne-1+k+l+\delta_{\beta'}+\delta_{\beta''}=Ne+e-1+i+\delta_{\beta}\]
\[Ne-1+k'+l'+\delta_{\gamma'}+\delta_{\gamma''}=Ne+e-1+i+\delta_{\beta}.\]
In particular, we learn that $k\geq i$ (resp. $k'\geq i$), with equality if and only if $l=e-1$ and $\delta_{\beta}+1=\delta_{\beta'}+\delta_{\beta''}$ (resp. $l'=e-1$ and $\delta_{\beta}+1=\delta_{\gamma'}+\delta_{\gamma''}$).
Also observe that when $k=i$ the product $b_{k,\beta'}a_{l,\beta''}$ (resp.~$d_{l',\gamma''}b_{k',\gamma'}$) is zero as soon as $\delta_{\beta''}=0$ (resp.~$\delta_{\gamma''}=0$). 
\end{itemize}
Let $\cO_{A,D,B-top}$ be the $\cO$-algebra generated by the coefficients of $A,D$ and $B_i$ for $i\geq 1-e$. 
The above observation implies that $X_\beta v^{-\delta_\beta}$ has degree $\leq e-1$, and each of its coefficients belong to $\cO_{A,D,B-top}$.

We now show that the coefficient of each entry of $B_i$ belongs to $\cO_{A,D,B-top}$ by downward induction on $i$. 
The claim clearly holds for $i\geq 1-e$. 
Suppose it holds up to $i+1$. 
Let $B^+_i$ be the matrix given by $B^+_{i,\beta}=\delta_\beta B_{i,\beta}$. 
It follows (using $\delta_{\beta'}+\delta_{\beta''}=1+\delta_\beta=2$ if and only if $\delta_{\beta'}=\delta_{\beta''}=1$) from the above facts that
\[iB^+_i+(\kappa_{b}-D_{lead} )B^+_i-B^+_i(\kappa_{t}-A_{lead})\in M_{s\times r}(\cO_{A,D,B-top}[v])\]

As in Proposition \ref{prop:parabolic_general}, the element $w\in W(\GL_n)$ induces an element $(w_{t},w_{b})\in W(\GL_r)\times W(\GL_s)$. We then have
$\Ad(w_{t}^{-1})(A_{lead})$, $\Ad(w_{b}^{-1})(D_{lead})$ are lower triangular. Thus Lemma \ref{lem:invertible_operator} below applies, and shows $B^{+}_i\in M_{s\times r}(\cO_{A,D,B-top}[v])$.
Now set $B^-_i=B_i-B^+_i$. 
Using what we just proved, we also get
\[(i-1)B^+_i+(\kappa_{b}-D_{lead} )B^+_i-B^+_i(\kappa_{t}-A_{lead})\in M_{s\times r}(\cO_{A,D,B-top}[v])\]
and the same argument shows $B^-_i\in M_{s\times r}(\cO_{A,D,B-top}[v])$. This finishes the inductive step.

Finally, since $C=BA$, $\cO_P$ also belongs to $\cO_{A,D,B-top}$.
\end{proof}
\begin{lemma}\label{lem:invertible_operator} Let $R$ be a ring with a subring $S$, $r+s=n$, $w_1\in W(\GL_r), w_2\in W(\GL_s)$. Suppose we are given $A_1\in M_{r}(S)$, $A_2\in M_r(S)$, $B\in M_{s\times r}(R)$ such that 
\begin{itemize}
\item $\Ad(w_i^{-1})(A_i)$ is lower triangular for $i=1,2$.
\item If $s_1$ is a diagonal entry of $A_1$ and $s_2$ is a diagonal entry of $A_2$, then $s_1-s_2\in S^{\times}$.
\item $BA_1-A_2B \in M_{s\times r}(S)$.
\end{itemize}
Then $B\in M_{s\times r}(S)$.
\end{lemma}
\begin{proof}
Replacing $B$ by $ w_2Bw_1^{-1}$, we may assume $w_1=1$, $w_2=1$. In this case, looking at the $(k,l)$-th entry of $BA_1-A_2B$ shows that 
$(s_1-s_2)B_{kl}$ belong to the subalgebra generated by $S$ and $B_{k'l'}$ with $k'-l'< k-l$, where $s_1, s_2$ are suitable diagonal entries of $A_1,A_2$. We conclude induction on $k-l$ that $B_{kl}\in S$. 
\end{proof}
\begin{rmk}\label{rmk:error_term} 
Suppose that in the setting of Lemma \ref{lem:parabolic_monodromy}, we don't have equation (\ref{eqn:parabolic_monodromy}) exactly but only an approximately: for $0\leq t<N$ and $\pi$ a root of $E_j$, the operator $(\frac{d}{dv})^t(E_j^{N+1}\cdot)|_{v=\pi}$ hitting on the matrix in (\ref{eqn:parabolic_monodromy}) belongs to $p^k\mathrm{Mat}_n(R)$. Then the proof shows that the conclusion of Lemma \ref{lem:parabolic_monodromy} also holds approximately:
there is an $\cO$-subalgebra $S$ of $R$ generated over $\cO_{A,D}$ by at most $ers$ elements such that $\cO_P\subset S+p^{k-(2N-1)(1-\frac{1}{e})}R$. This follows from Lemma \ref{lem:error_control}.

\end{rmk}
\subsection{Proof of Theorem \ref{thm:FSM}}
\begin{proof}

We recall the setting of Theorem \ref{thm:FSM}. 
We are given $\tau$, a tame inertial type over $E$, together with a fixed lowest alcove presentation $(s,\mu)$ for it, such that $\mu$ is $\max\{(3n-7)(e-1)+2n-6,(2n-3)e\}$-deep. 
Furthermore, $\tld{w}(\rhobar,\tau)=(\tld{w}_j^{-1}t_{e\eta}\tld{w}_{\alpha_j}\tld{w}_j)_j$ for some simple root $\alpha_j$ for each $j\in \cJ$.

We assume $R^{\leq \eta,\tau}_\rhobar \neq 0$, otherwise there is nothing to prove. In particular we obtain $\ovl{\fM}\in Y^{\leq \eta,\tau}(\F)$ such that $T^*_{dd}(\ovl{\fM})\cong \rhobar|_{{G_{K_\infty}}}$. Then $\ovl{\fM}$ has shape $w^{-1}\tld{z}w=\tld{w}(\rhobar,\tau)^*$, {where $w\in W^{\cJ}$, $\tld{z}\in \tld{W}^{\cJ}$ are as in the statement of Proposition \ref{prop:parabolic_general}.}

We need to analyze $R^{\tau,\ovl{\beta},\nabla}_{\ovl{\fM}}$ in the context of diagram (\ref{diag:main}).

 We first observe that for each $j$, $\tld{z}_j$ has a block diagonal structure
\[\tld{z}_j=\begin{pmatrix} \tld{z}_{j,t} & 0 & 0 \\ 0 & \tld{z}_{\alpha_j} & 0\\ 0 & 0 & \tld{z}_{j,b}\end{pmatrix}\]
with sizes $r$, $2$, $s$ where 
\begin{itemize}
\item $\tld{z}_{j,t}=t_{e(n-1,\cdots, s+2)}$
\item $\tld{z}_{j,b}=t_{e(s-1,\cdots 0)}$.
\item $v^{-es}\tld{z}_{\alpha_j}$ has elementary divisors bounded by $v^{(e,0)}$.
\end{itemize}
In particular, we are in a position to repeatedly apply Proposition \ref{prop:parabolic_general} to $R^{\tau,\ovl{\beta}}_{\ovl{\fM}}$ and each $A^{(j)}_{\fM}$ for the universal Breuil--Kisin module $\fM$. 
This gives
\[A^{(j)}_{\fM}=D^{(j)}w_j^{-1} \begin{pmatrix}P^{(j)} \end{pmatrix} w_j\]
where $P^{(j)}$ is block lower triangular, whose Levi blocks from top to bottom are $E_j^{n-1},\cdots E_j^{s+2},M_{\alpha_j},E_j^{s-1},\cdots, 1$ {(in particular, this defines $M_{\alpha_j}$ as the $2$ by $2$ block of $P^{(j)}$)}. 
Furthermore, the entries of $D^{(j)}$, $P^{(j)}$ over all $j$ topologically generate $R^{\tau,\ovl{\beta}}_{\ovl{\fM}}$. 

Set $x_j=\ovl{M}_{\alpha_j}$ be the reduction of $M_{\alpha_j}$ modulo the maximal ideal.
Then $x_j$ can be naturally interpreted as an element of $M_j(t_{(1,0)})(\F)$.
By Proposition \ref{prop:Levi_reduction}, we get an $\cO$-algebra map $\cO^{\wedge}_{M_j(1,0),x_j}\ra R^{\tau,\ovl{\beta}}_{\ovl{\fM}}$ sending generators of $\cO^{\wedge}_{M_j(1,0),x_j}$ to the corresponding coefficients of the entries of $\ovl{M}_{\alpha_j}$.
Set $x=(x_j)\in M_{\cJ}(t_{(1,0)})=\prod_j M_j(t_{(1,0)})$.
Thus $R^{\tau,\ovl{\beta}}_{\ovl{\fM}}$ acquires an $\widehat{\bigotimes}_{\cO}\cO^{\wedge}_{M_j(1,0),x_j}=\cO^{\wedge}_{M_\cJ(t_{(1,0)}),x}$-algebra structure, and the image of $\cO^{\wedge}_{M_\cJ(t_{(1,0)}),x}$ coincides with the topological subalgebra generated by the coefficients of the entries of $(M_{\alpha_j})$ for all possible $j$.

Repeated applications of the approximate version of Lemma \ref{lem:parabolic_monodromy} to $R^{\tau,\ovl{\beta},\nabla}$ as in Remark \ref{rmk:error_term} (with the control of the monodromy condition obtained by combining Proposition \ref{prop:monodromy_control} and Lemma \ref{lem:error_control}) show that $R^{\tau,\ovl{\beta},\nabla}_{\ovl{\fM}}$ is topologically generated over $\cO^{\wedge}_{M_{\cJ}(1,0),x}$ by $fn+e\sum_j \dim N_{-\alpha_j} = fn+ (\frac{n(n-1)}{2}-1)[K:\Qp]$ elements. But since we assumed $R^{\leq \eta,\tau}\neq 0$, $\dim R^{\tau,\ovl{\beta},\nabla}_{\ovl{\fM}}=1+fn+\frac{n(n-1)}{2}[K:\Qp]= \dim \cO^{\wedge}_{M_{\cJ}(1,0),x}+fn+ (\frac{n(n-1)}{2}-1)[K:\Qp]$. Since $\cO^{\wedge}_{M_{\cJ}(1,0),x}$ is an integral domain (being the completion of an excellent normal scheme), the equality of dimension can only happen if $R^{\tau,\ovl{\beta},\nabla}_{\ovl{\fM}}$ is a power series ring over $\cO^{\wedge}_{M_{\cJ}(1,0),x}$ in the correct number of variables. 
All the assertions of Theorem \ref{thm:FSM} now follows from properties of the $M_{\cJ}(t_{(1,0)})$ which follows form Corollary \ref{cor:GL_2model}.
\end{proof}\clearpage{}%
\clearpage{}%
\section{The main results}
\label{sec:main:mod}

In this section, we prove our main results on the weight part of Serre's conjecture. 
We start with an axiomatic setup before defining the relevant spaces of automorphic forms in \S \ref{sec:TW}. 

Recall from \S \ref{subsub:Lp} that given an $\F$-valued $L$-homomorphism $\rhobar: G_{\Q_p} \ra{}^L \un{G}(\F)$ (resp.~a tame inertial $L$-parameter $\tau: I_{\Q_p} \ra\un{G}^\vee(E)$) we have a corresponding collection $(\rhobar_v)_{v\in S_p}$ of continuous Galois representations $\rhobar_v:G_{F^+_v}\ra\GL_n(\F)$ (resp.~a corresponding collection $(\tau_v)_{v\in S_p}$ of tame inertial types $\tau_v:I_{F^+_v}\ra\GL_n(E)$).

\subsection{Weight elimination}

\begin{thm}\label{thm:WE}
Let $\rhobar: G_{\Q_p} \ra{}^L \un{G}(\F)$ be a $3e(n-1)$-generic $\F$-valued $L$-homomorphism.
Let $\rhobar^\speci$ be a specialization of $\rhobar$ with a compatible $\max\{2,e\}(n-1)$-generic lowest alcove presentation.
Assume that we have a set $W_{\textnormal{elim}}(\rhobar)$ of $3(n-1)$-deep Serre weights satisfying the following local-global compatibility axiom:
\begin{enumerate}[(i)]
\item
\label{eq:elim} 
for any tame inertial $L$-parameter $\tau$, $\JH(\ovl{\sigma}(\tau)) \cap W_{\textnormal{elim}}(\rhobar) \neq \emptyset$ implies that $\rhobar$ has a potentially crystalline lift of type $(\tau,\eta)$.
\end{enumerate}
Then $W_{\textnormal{elim}}(\rhobar) \subset W^?(\rhobar^\speci)$.
\end{thm}
\begin{proof}
Suppose that $F(\lambda) \in W_{\textrm{elim}}(\rhobar)$.
Choose the tame inertial $L$-parameter $\tau$ with $F(\lambda) \in \JH(\ovl{\sigma}(\tau))$ constructed in Proposition \ref{prop:combWE}.
By Proposition \ref{prop:CL:specialfiber} and Theorem \ref{thm:semicont}, $\tld{w}(\rhobar^{\speci},\tau)\in \Adm(e\eta_0)$, and we conclude by Proposition \ref{prop:combWE}.
\end{proof}

\begin{rmk}
If $e\geq 2$ the hypothesis on $\rhobar^{\speci}$ follows from the hypothesis on $\rhobar$.
\end{rmk}

\subsection{Patching functors}\label{sec:patchfunc}

We recall weak patching functors.
Let
\[
R_{\rhobar} \defeq \widehat{\bigotimes}_{v\in S_p,\cO} R_{\rhobar_v}^\square,
\]
and let $R^p$ be a nonzero complete local Noetherian equidimensional flat $\cO$-algebra with residue field $\F$ such that each irreducible component of $\Spec R^p$ and of $\Spec \overline{R}^p$ is geometrically irreducible. 
(The latter hypothesis can be guaranteed after passing to a finite extension of the coefficient field $E$.) 
We let $R_\infty\defeq R_{\rhobar} \widehat{\otimes}_{\cO} R^p$ and suppress the dependence on $R^p$ below.
We let $R_{\rhobar}^{\eta,\tau}$ be
\[
 \widehat{\bigotimes}_{v\in S_p,\cO} R_{\rhobar_{v}}^{\eta_{v},\tau_{v}}%
\] and define $R_\infty(\tau)\defeq R_\infty \otimes_{R_{\rhobar}} R_{\rhobar}^{\eta,\tau}$.
We write $X_\infty$, $X_\infty(\tau)$, and $\ovl{X}_\infty(\tau)$
for $\Spec R_\infty$, $\Spec R_\infty(\tau)$, and $\Spec \ovl{R}_\infty(\tau)$ respectively, denote by $\Mod(X_\infty)$ the category of coherent sheaves over $X_\infty$, and let $\Rep_{\cO}(\GL_n(\cO_p))$ be the category of topological $\cO[\GL_n(\cO_p)]$-modules which are finitely generated over $\cO$. 
We say that an $\ovl{E}$-point of $\Spec R_{\rhobar}$ is \emph{potentially diagonalizable} if for each $v \in S_p$, the corresponding Galois representation $G_{F^+_{v}} \ra \GL_n(\ovl{E})$ is potentially diagonalizable in the sense of \cite[\S 1.4]{BLGGT}.
We say that an $\ovl{E}$-point of $X_\infty$ is potentially diagonalizable if its image in $\Spec R_{\rhobar}$ is.

\begin{defn}\label{minimalpatching}
A \emph{weak patching functor} for an $L$-homomorphism $\rhobar: G_{\Q_p} \ra{}^L \un{G}(\F)$ is a nonzero covariant exact functor $M_\infty:\Rep_{\cO}(\GL_n(\cO_p))\ra \Coh(X_{\infty})$ satisfying the following: if $\tau$ is an inertial $L$-parameter and $\sigma^\circ(\tau)$ is an $\cO$-lattice in $\sigma(\tau)$ then
\begin{enumerate}
\item 
\label{support}
$M_\infty(\sigma^\circ(\tau))$ is a maximal Cohen--Macaulay sheaf on $X_\infty(\tau)$; 
\item 
\label{dimd}
for all $\sigma \in \JH(\ovl{\sigma}^\circ(\tau))$, $M_\infty(\sigma)$ is a maximal Cohen--Macaulay sheaf on $\ovl{X}_\infty(\tau)$ (or is $0$); and
\item
\label{item:pd} if there is an inertial $L$-parameter $\tau_{0}$ such that $\supp M_\infty(\sigma(\tau_{0})^\circ)$ contains a potentially diagonalizable $\overline{E}$-point, then for any inertial $L$-parameter $\tau$, $\supp M_\infty(\sigma(\tau)^\circ)$ contains all potentially diagonalizable $\overline{E}$-points. 
\end{enumerate}
We say that a weak patching functor $M_\infty$ is \emph{minimal} if $R^p$ is formally smooth over $\cO$ and whenever $\tau$ is an inertial $L$-parameter, $M_\infty(\sigma^\circ(\tau))[p^{-1}]$, which is locally free over (the regular scheme) $\Spec R_\infty(\tau)[p^{-1}]$, has rank at most one on each connected component.
\end{defn}

\begin{defn}
We say that a weak patching functor $M_\infty$ is \emph{potentially diagonalizable} if there exists $\tau_0$ as in Definition \ref{minimalpatching}(\ref{item:pd}). 
\end{defn}

\subsection{Cycles from patching functors}\label{sec:cycles}

We recall some notation from \cite[\S 2.2]{EG}.
Let $\mathcal{X}$ be an equidimensional Noetherian scheme of dimension $d$. 
Let $\mathcal{Z}(\mathcal{X})$ be the free abelian group generated by integral subschemes of $\mathcal{X}$ of \emph{maximal dimension} $d$.
If $\mathcal{M}$ is a coherent sheaf on $\mathcal{X}$ with finite-dimensional support, then we can define $Z(\mathcal{M}) \in \mathcal{Z}(\mathcal{X})$ to be $Z_d(\mathcal{M})$ which is defined as in \emph{loc.~cit.}

Now suppose that $\mathcal{X}$ is a $p$-flat equidimensional Noetherian scheme over $\cO$.
Then $\mathcal{X}[p^{-1}] \defeq \mathcal{X} \otimes_{\cO} E$ and $\ovl{\mathcal{X}} \defeq \mathcal{X} \otimes_{\cO} \F$ are equidimensional Noetherian schemes, and there is a natural reduction map $\red: \cZ(\mathcal{X}[p^{-1}]) \ra \cZ(\ovl{\mathcal{X}})$.
Moreover, if we let $\mathcal{M}[p^{-1}] \defeq \mathcal{M} \otimes_{\cO} E$ and $\ovl{\mathcal{M}} \defeq \mathcal{M} \otimes_{\cO} \F$ be the corresponding sheaves on $\mathcal{X}[p^{-1}]$ and $\ovl{\mathcal{X}}$, respectively, we have the following fact.

\begin{prop}
If $\mathcal{M}$ is an $\cO$-flat coherent sheaf over $\mathcal{X}$ with finite-dimensional support, then $\red(Z(\mathcal{M}[p^{-1}])) = Z(\ovl{\mathcal{M}})$.
\end{prop}

We introduce notation for completed products of cycles. 
Suppose that $R$ and $S$ are equidimensional complete local Noetherian flat $\cO$-algebras. 
If $\ovl{Z}_1$ and $\ovl{Z}_2$ are geometrically integral subschemes of $\Spec R\otimes_{\cO} \F$ and $\Spec S \otimes_{\cO}\F$ corresponding to prime ideals $\ovl{\mathfrak{p}}$ and $\ovl{\mathfrak{q}}$, respectively, then we denote by $\ovl{Z}_1\times \ovl{Z}_2$ the subscheme 
\[
\Spec (R\otimes_{\cO} \F)/\ovl{\mathfrak{p}}\widehat{\otimes}_{\F} (S\otimes_{\cO} \F)/\ovl{\mathfrak{q}} \subset \Spec (R\otimes_{\cO} \F) \widehat{\otimes}_{\F} (S \otimes_{\cO}\F) 
\]
which is geometrically integral by \cite[Lemma 3.3(4)]{BLGHT2}. 
Similarly, if ${Z}_1$ and ${Z}_2$ are geometrically integral subschemes of $\Spec R[p^{-1}]$ and $\Spec S[p^{-1}]$ corresponding to prime ideals $\mathfrak{p}$ and $\mathfrak{q}$, then we denote by $Z_1\times Z_2$ the subscheme 
\[
\Spec (R/(\mathfrak{p}\cap R) \widehat{\otimes}_{\cO} S/(\mathfrak{q}\cap S))[p^{-1}] \subset  \Spec R \widehat{\otimes}_{\cO} S[p^{-1}] 
\]
which is geometrically integral by \cite[Lemma 3.3(3)]{BLGHT2}. 

We now specialize to some schemes in our patching axioms. 
Let $\rhobar$ be an $L$-homomorphism over $\F$. 
Fix a finite set $\mathcal{T}$ of inertial $L$-parameters such that:
\begin{enumerate}[(i)]
\setcounter{enumi}{1}
\item
\label{cond:1}
for all $\tau\in \mathcal{T}$ the irreducible components of $\Spec R_{\rhobar}^\tau$ and $\Spec \ovl{R}_{\rhobar}^\tau$ are geometrically integral. 
\end{enumerate}
Let $\Spec R_{\rhobar}^\mathcal{T}$ be the reduced union $\cup_{\tau\in \mathcal{T}}\Spec R_{\rhobar}^\tau$. 
Let $M_\infty$ be a weak patching functor for $\rhobar$. 
We write $R_\infty(\mathcal{T})$ for $R_\infty \widehat{\otimes}_{R_{\rhobar}} R_{\rhobar}^\mathcal{T} \cong R^p \widehat{\otimes}_{\cO} R_{\rhobar}^\mathcal{T}$ and $X_\infty(\mathcal{T})$ for $\Spec R_\infty(\mathcal{T})$. 
Recall that by assumption, the irreducible components of $\Spec R^p[p^{-1}]$ and $\Spec \ovl{R}^p$ are geometrically irreducible. 
Every irreducible cycle $Z \in \cZ(\cX_\infty(\mathcal{T})[p^{-1}])$ is of the form $Z^p \times Z_p$ for geometrically irreducible cycles $Z^p \in \cZ(\Spec R^p[p^{-1}])$ and $Z_p \in \cZ(\Spec R_{\rhobar}^\mathcal{T})$ by \cite[Lemma 3.3(5)]{BLGHT2}. 
Similarly, every irreducible cycle $\ovl{Z} \in \cZ(\ovl{\cX}_\infty(\mathcal{T}))$ is of the form $\ovl{Z}^p \times \ovl{Z}_p$ for geometrically irreducible cycles $\ovl{Z}^p \in \cZ(\Spec \ovl{R}^p)$ and $\ovl{Z}_p \in \cZ(\Spec \ovl{R}_{\rhobar}^\mathcal{T})$ by \cite[Lemma 3.3(6)]{BLGHT2}.

Let $e: \cZ(\Spec \ovl{R}^p) \ra \Z$ be the homomorphism that sends the cycle of an integral subscheme to $1$.
We define the maps 
\begin{align*}
\pr: \cZ(X_\infty(\mathcal{T})[p^{-1}]) &\ra \cZ(\Spec R_{\rhobar}^\mathcal{T}[p^{-1}])\\
Z^p \times Z_p &\mapsto e(\red(Z^p)) Z_p
\end{align*}
and
\begin{align*}
\ovl{\pr}: \cZ(\ovl{X}_\infty(T)) &\ra \cZ(\Spec \ovl{R}_{\rhobar}^\mathcal{T})\\
\ovl{Z}^p \times \ovl{Z}_p &\mapsto \ovl{Z}_p.
\end{align*}
We have that $\red \circ \pr = \ovl{\pr} \circ \red: \cZ(X_\infty(\mathcal{T})[p^{-1}]) \ra \cZ(\Spec \ovl{R}_{\rhobar}^\mathcal{T})$ (using that $\red(Z^p \times Z_p) = \red(Z_p) \times \red(Z^p)$), from which we immediately obtain the following corollary.

\begin{cor}\label{cor:pcomp}
If $\tau \in T$, then the image of the composition 
\[
\cZ(X_\infty(\tau)[p^{-1}]) \overset{\red}{\ra} \cZ(\ovl{X}_\infty(\tau)) \overset{\ovl{\pr}}{\ra} \cZ(\Spec \ovl{R}_{\rhobar}^{\tau})
\]
is contained in $\red(\cZ(\Spec R_{\rhobar}^{\tau}[p^{-1}]))$.
\end{cor}

Depending on context, we denote either $\pr \circ Z$ or $\ovl{\pr} \circ Z$ by $Z_\mathfrak{p}$.

\subsection{Weight elimination and modularity of extremal weights}
\label{subsec:WE:MOD}

Let $\rhobar$ be a $1$-generic $L$-homomorphism, with a lowest alcove presentation for it.
Fix a weak patching functor $M_\infty$ for $\rhobar$.
Let $W_{M_\infty}(\rhobar)$ be the set of $3(n-1)$-deep Serre weights $\sigma$ such that $M_\infty(\sigma)$ is nonzero.

\begin{prop}\label{prop:PWE}
The set $W_{M_\infty}(\rhobar)$ satisfies the condition (\ref{eq:elim}) for $\rhobar$.
\end{prop}
\begin{proof}
Suppose that $\sigma \in \JH(\ovl{\sigma}(\tau)) \cap W_{M_\infty}(\rhobar)$ for a generic tame inertial $L$-parameter $\tau$.
Then $M_\infty(\sigma)$ is nonzero so that $M_\infty(\sigma(\tau)^\circ)$ is nonzero for any lattice $\sigma(\tau)^\circ \subset \sigma(\tau)$ by exactness.
Definition \ref{minimalpatching}(\ref{support}) implies that $R_\infty(\tau)$, and so $R_{\rhobar}^{\tau}$, is nonzero.
\end{proof}

\begin{defn}
We say that a weak patching functor $M_\infty$ for $\rhobar$ is \emph{extremal} if $W_\obv(\rhobar) \cap W_{M_\infty}(\rhobar)$ is nonempty.
\end{defn}

\begin{thm}\label{thm:obv}
Let $\rhobar$ be $6e(n-1)$-generic.
If a weak patching functor $M_\infty$ is extremal, then $W_\obv(\rhobar) \subset W_{M_\infty}(\rhobar)$, and moreover, the map $\theta_{\rhobar}: SP(\rhobar) \ra W^{\cJ}$ is a bijection.
\end{thm}

The proof of Theorem \ref{thm:obv} requires the following two results.
\begin{lemma}\label{lemma:walk}
Assume that $\rhobar$ is $(2e+\max\{2,e\})(n-1)$-generic.
Suppose that $M_\infty$ is a weak patching functor for $\rhobar$, $(\sigma,\rhobar^\speci) \in SP(\rhobar)$, and $\sigma \in W_{M_\infty}(\rhobar)$.
Assume that $\rhobar^\speci$ has a compatible $4e(n-1)$-generic lowest alcove presentation.
Suppose that $\sigma$ is the extremal weight of $\rhobar^\speci$ corresponding to $w\in W$.
Let $\alpha$ be a simple root.

Using $\rhobar^\speci$, $w$, and $\alpha$, we define as in Proposition \ref{prop:indcomb} (using the above $4e(n-1)$-generic lowest alcove presentation) $\sigma_m$ and $\tau_m$ for $0\leq m \leq 2e-1$ and $0 \leq m \leq 2e$, respectively, so that $\sigma_0 = \sigma$.
Then there exists $0 \leq k \leq 2e-1$ such that $\sigma_m \in W_{M_\infty}(\rhobar)$ if and only if $m \leq k$.
Moreover, $\tau_{k+1}$ exhibits a specialization of $\rhobar$ to $(\sigma_k,\rhobar^{\prime,\speci})$ for some $\F$-valued tame inertial $L$-parameter $\rhobar^{\prime,\speci}$.
\end{lemma}
\begin{proof}
Fix $\cO$-lattices $\sigma(\tau_m)^\circ\subset \sigma(\tau_m)$ for $0 \leq m \leq 2e$ (the choices will not affect the argument below).
Let $0 \leq k\leq 2e-1$ be such that $\sigma_m \in W_{M_\infty}(\rhobar)$ for $0\leq m\leq k$ and either $\sigma_{k+1} \notin W_{M_\infty}(\rhobar)$ or $k = 2e-1$.
That $\sigma_m \in W_{M_\infty}(\rhobar)$ implies that $M_\infty(\sigma(\tau_{m+1})^\circ)$ is nonzero. 
Therefore $R_{\rhobar}^{\tau_{m+1}}$ is nonzero for $0\leq m\leq k$. 
We will first show that $\tld{w}(\rhobar,\tau_{m+1}) \in w^{-1}t_{e\eta_0}W_{a,\alpha} w \cap \Adm^\vee(e\eta_0)$ for $0\leq m\leq k$.  

Fix $m$ with $0\leq m\leq k$. 
Suppose that $\tau_{m+1}$ exhibits the specialization to the $\F$-valued tame inertial $L$-parameter $\rhobar^{\prime,\speci}$ i.e.~that $\tld{w}(\rhobar,\tau_{m+1}) = \tld{w}(\rhobar^{\prime,\speci},\tau_{m+1})$. 
Since $\sigma_m \in W_{M_\infty}(\rhobar) \subset W^?(\rhobar^{\prime,\speci})$ by Theorem \ref{thm:WE} and Proposition \ref{prop:PWE}, Proposition \ref{prop:intersect} implies that $\tld{w}_h \tld{s} \tld{w}(\rhobar^{\prime,\speci},\tau_{m+1}) \leq w_0 t_{(e-1)\eta_0} \tld{s}$ for $s \in \{w,s_\alpha w\}$.
(Note that Theorem \ref{thm:WE} applies to $\rhobar$ and $\rhobar^{\prime,\speci}$, by the genericity assumption on $\rhobar$.)
This implies that $\tld{w}(\rhobar,\tau_{m+1}) = \tld{w}(\rhobar^{\prime,\speci},\tau_{m+1}) \leq t_{s^{-1}(e\eta_0)}$. 
Combining this with the fact that $\tld{w}(\rhobar^\speci,\tau_{m+1})\leq \tld{w}(\rhobar,\tau_{m+1})$ by Theorem \ref{thm:semicont}, we have that $\tld{w}(\rhobar,\tau_{m+1}) \in w^{-1}t_{e\eta_0}W_{a,\alpha} w \cap \Adm^\vee(e\eta_0)$ by Proposition \ref{prop:upset}.

Now Proposition \ref{prop:indcomb} applied to $\rhobar^\speci$, Theorem \ref{thm:WE}, and Proposition \ref{prop:PWE} imply that 
\[
W_{M_\infty}(\rhobar) \cap \JH(\ovl{\sigma}(\tau_{m+1})^\circ) \subset W^?(\rhobar^\speci) \cap \JH(\ovl{\sigma}(\tau_{m+1})^\circ) = \{\sigma_m,\sigma_{m+1}\} 
\]
(or $\{\sigma_m\}$ if $m = 2e-1$). 
We now use notation from \S \ref{sec:cycles} with $\mathcal{T} = \{\tau_m\mid 0 \leq m \leq 2e-1\}$. 
The set of types $\mathcal{T}$ satisfies condition \eqref{cond:1} by the genericity assumption on $\rhobar^{\speci}$ and Theorem \ref{thm:FSM}. 
We continue to fix $m$ with $0\leq m \leq k$. 
Since $\sigma_m$ and $\sigma_{m+1}$ appear as Jordan--H\"older factors of $\ovl{\sigma}(\tau_{m+1})^\circ$ with multiplicity one, exactness of $M_\infty$ gives  
\begin{equation}\label{eqn:cyclesum}
Z_\fp(M_\infty(\ovl{\sigma}(\tau_{m+1})^\circ)) = Z_\fp(M_\infty(\sigma_{m}))+Z_\fp(M_\infty(\sigma_{m+1})), 
\end{equation}
for $0 \leq m \leq k-1$ and $Z_\fp(M_\infty(\ovl{\sigma}(\tau_{k+1})^\circ)) = Z_\fp(M_\infty(\sigma_{k}))$. 
We will use \eqref{eqn:cyclesum} and the previous paragraph to show that $\tld{w}(\rhobar,\tau_{k+1})$ is $t_{s^{-1}(\eta_0)}$ for some $s\in \{w,s_\alpha w\}$.

Let us call a cycle \emph{balanced} if it is a multiple of the sum of two distinct integral subschemes and \emph{unbalanced} if it is supported on at most two integral subschemes with distinct multiplicities. 
In particular, an unbalanced cycle is nonzero. 
For $0 \leq m \leq k-1$, $\tld{w}(\rhobar,\tau_{m+1}) \notin \{t_{w^{-1}(\eta_0)},t_{(s_\alpha w)^{-1}(\eta_0)}\}$ since otherwise 
\[
2 = \# W_{M_\infty}(\rhobar) \cap \JH(\ovl{\sigma}(\tau_{m+1})^\circ) \leq \# W^?(\rhobar^{\prime,\speci}) \cap \JH(\ovl{\sigma}(\tau_{m+1})^\circ) = 1
\] 
by Proposition \ref{prop:PWE}.
Then $Z_\fp(M_\infty(\ovl{\sigma}(\tau_{m+1})^\circ))$ is balanced by Corollary \ref{cor:pcomp} since $R_{\rhobar}^{\tau_{m+1}}[p^{-1}]$ is geometrically irreducible and $Z(\Spec \ovl{R}_{\rhobar}^{\tau_{m+1}})$ is balanced (see Theorem \ref{thm:FSM}).
By \eqref{eqn:cyclesum}, $Z_\fp(M_\infty(\sigma_{m}))$ is balanced (resp.~unbalanced) if and only if $Z_\fp(M_\infty(\sigma_{m+1}))$ is balanced (resp.~unbalanced) for $0 \leq m \leq k-1$.
Since $Z_\fp(M_\infty(\sigma_0)) = Z_\fp(M_\infty(\ovl{\sigma}(\tau_0)^\circ))$ is unbalanced as $R_{\rhobar}^{\tau_0}$ is formally smooth over $\cO$, we conclude that $Z_\fp(M_\infty(\sigma_k)) = Z_\fp(M_\infty(\ovl{\sigma}(\tau_{k+1})^\circ))$ is unbalanced. 
We conclude from the argument above that $\tld{w}(\rhobar,\tau_{k+1})$ is $t_{s^{-1}(\eta_0)}$ for some $s\in \{w,s_\alpha w\}$.
In particular, $\tau_{k+1}$ exhibits a specialization of $\rhobar$ to $(\sigma_k,\rhobar^{\prime,\speci})$ (not necessarily the same $\rhobar^{\prime,\speci}$ from the first paragraph).

By the definition of $\sigma_k$ and using that $\sigma_k$ is an extremal weight of $\rhobar^{\prime,\speci}$, we see that 
\[
\tld{w}(\rhobar,\tau_{k+1}) = \tld{w}(\rhobar^{\prime,\speci},\tau_{k+1}) =
\begin{cases}
t_{w^{-1}(e\eta_0)} & \mbox{ if } k \mbox{ is even}\\
t_{(s_\alpha w)^{-1}(e\eta_0)} & \mbox{ if } k \mbox{ is odd}.
\end{cases}
\]
Then a computation shows that
\[
\tld{w}(\rhobar^{\prime,\speci}) = 
\begin{cases}
\tld{w}(\rhobar^\speci)\tld{w}^{-1} t_{(\frac{k}{2}-e)\alpha}s_\alpha \tld{w} & \mbox{ if } k \mbox{ is even}\\
\tld{w}(\rhobar^\speci)\tld{w}^{-1} t_{(\frac{k+1}{2}-e)\alpha} \tld{w} & \mbox{ if } k \mbox{ is odd}.
\end{cases}
\]
Note that $\tld{w}(\rhobar^{\prime,\speci},\tau_m) \in w^{-1} W_{a,\alpha} t_{e\eta_0} w$ for all $0 \leq m \leq 2e$.
Another computation shows that if $m > k+1$, then $\tld{w}(\rhobar^{\prime,\speci},\tau_m)$ is not listed in Proposition \ref{prop:listshapes}.
This implies that $\tld{w}(\rhobar^{\prime,\speci},\tau_m) \notin \Adm(e\eta_0)$ for $m> k+1$.
Corollary \ref{cor:adm} implies that $W^?(\rhobar^{\prime,\speci},\tau_m) = \emptyset$ for $m>k+1$.
In particular, $\sigma_m \notin W^?(\rhobar^{\prime,\speci})$ for $m>k$.
Theorem \ref{thm:WE} and Proposition \ref{prop:PWE} imply that $\sigma_m \notin W_{M_\infty}(\rhobar)$ for $m>k$.
\end{proof}

\begin{cor}\label{cor:reflect}
Let $(\sigma,\rhobar^\speci) \in SP(\rhobar)$ and $\sigma \in W_{M_\infty}(\rhobar)$ be as in Lemma \ref{lemma:walk}.
Let $\alpha$ be a simple root.
Then there exists $(\sigma',\rhobar^{\prime,\speci})$ such that $\theta_{\rhobar}(\sigma',\rhobar^{\prime,\speci}) = \theta_\rhobar(\sigma,\rhobar^\speci)s_\alpha$.
Moreover, if $\sigma\in W_{M_\infty}(\rhobar)$, then $\sigma' \in W_{M_\infty}(\rhobar)$ as well.
\end{cor}
\begin{proof}
Let $\sigma$, $\sigma_k$, $\rhobar^\speci$, and $\rhobar^{\prime,\speci}$ be as in Lemma \ref{lemma:walk}.
Let $\sigma'$ be $\sigma_k$.
Then $\sigma' \in W_{M_\infty}(\rhobar)$.
It suffices to show that $\theta_{\rhobar}(\sigma',\rhobar^{\prime,\speci}) = \theta_{\rhobar}(\sigma,\rhobar^\speci)s_\alpha$.
We have that
\[
w(\rhobar^{\prime,\speci}) = 
\begin{cases}
w(\rhobar^\speci) w^{-1}s_\alpha w & \mbox{ if } k \mbox{ is even}\\
w(\rhobar^\speci) & \mbox{ if } k \mbox{ is odd}
\end{cases}
\]
and $\sigma_k$ is the extremal weight of $\rhobar^{\prime,\speci}$ corresponding to 
\[
\begin{cases}
w & \mbox{ if } k \mbox{ is even.}\\
s_\alpha w & \mbox{ if } k \mbox{ is odd.}
\end{cases}
\]
We conclude that $\theta_{\rhobar}(\sigma',\rhobar^{\prime,\speci}) = w(\rhobar) w^{-1}s_\alpha = \theta_\rhobar(\sigma,\rhobar^\speci)s_\alpha$.
\end{proof}

\begin{proof}[Proof of Theorem \ref{thm:obv}]
Suppose that $\sigma \in W_{M_\infty}(\rhobar)$ and that $(\sigma,\rhobar^\speci) \in SP(\rhobar)$.
Then using Corollary \ref{cor:reflect} and the fact that simple reflections generate $W$, we see that for each $w\in W$, there is $(\sigma_w,\rhobar_{w}^\speci)\in SP(\rhobar)$ such that $\theta_\rhobar(\sigma_w,\rhobar_{w}^\speci) = w$ and $\sigma_w \in W_{M_\infty}(\rhobar)$.
This first implies that the map $\theta_{\rhobar}$ is surjective and hence an isomorphism by Proposition \ref{prop:inj}.
It also implies that $W_\obv(\rhobar) \subset W_{M_\infty}(\rhobar)$.
\end{proof}

\begin{thm}\label{thm:obvpd}
Let $\rhobar$ be $6e(n-1)$-generic and let $M_\infty$ be a weak patching functor for $\rhobar$. 
The following are equivalent.
\begin{enumerate}
\item \label{item:obv} $M_\infty$ is extremal.
\item \label{item:allobv} $W_\obv(\rhobar)\subset W_{M_\infty}(\rhobar)$.
\item \label{item:mpd} $M_\infty$ is potentially diagonalizable.
\end{enumerate}
\end{thm}
\begin{proof}
(\ref{item:obv}) implies (\ref{item:allobv}) by Theorem \ref{thm:obv}.
We next show that (\ref{item:allobv}) implies (\ref{item:mpd}).
Let $\sigma$ be in $W_\mord(\rhobar)$ so that $\tau$ exhibits the specialization pair $(\sigma,\rhobar^\semis) \in SP(\rhobar)$ as in the proof of Proposition \ref{prop:obvord}.
Then $M_\infty(\sigma(\tau)^\circ)$ is nonzero since $M_\infty(\sigma)$ is.
Since $R_{\rhobar}^{\tau}$ is a domain and $\rhobar$ has a potentially diagonalizable lift of type $(\tau,\eta)$ by Theorem \ref{thm:extension}, $M_\infty$ is potentially diagonalizable.

Finally, we show that (\ref{item:mpd}) implies (\ref{item:obv}).
Again, let $\sigma$ be in $W_\mord(\rhobar)$ so that $\tau$ exhibits the specialization pair $(\sigma,\rhobar^\semis) \in SP(\rhobar)$ as in the proof of Proposition \ref{prop:obvord}.
Then since $M_\infty$ is potentially diagonalizable and $\rhobar$ has a potentially diagonalizable lift of type $(\tau,\eta)$ as before, $M_\infty(\sigma(\tau)^\circ)$ is nonzero.
Since $W_{M_\infty}(\rhobar) \subset W^?(\rhobar^\semis)$ by Theorem \ref{thm:WE} and $W^?(\rhobar^\semis,\tau) = \{\sigma\}$ as in the proof of Proposition \ref{prop:obvord}, $M_\infty(\sigma)$ is nonzero.
Thus $W_\mord(\rhobar) \cap W_{M_\infty}(\rhobar)$ is nonempty.
The result now follows from Proposition \ref{prop:obvord}.
\end{proof}

\begin{rmk}
Theorem \ref{thm:obvpd} generalizes \cite[Theorem 4.3.8]{LLL} to the nonsemisimple case in an abstract setting.
Moreover, the above proof (and \S \ref{sec:TW}) gives a different proof of this theorem. 
(Specifically, the order of implications proved is reversed.) 
Indeed, we do not know whether every extremal lift is potentially diagonalizable when $\rhobar$ is wildly ramified. 
\end{rmk}

\begin{cor}
Suppose that $\cO_p$ is \'etale over $\Z_p$, i.e., $F^+_p$ is a product of unramified extensions of $\Q_p$. 
Let $\rhobar$ be an $L$-homomorphism over $\F$. 
Suppose that $M_\infty$ is a weak patching functor for $\rhobar$ satisfying the equivalent conditions of Theorem \ref{thm:obvpd}.
(In particular, $\rhobar$ is $7(n-1)$-generic.) 
If $\tau$ is an $n$-generic tame inertial $L$-parameter, then $R_{\rhobar}^\tau$ is nonzero if and only if $M_\infty(\sigma^\circ(\tau))$ is nonzero for any $\cO$-lattice $\sigma^\circ(\tau) \subset \sigma(\tau)$. 
\end{cor}
\begin{proof}
If $M_\infty(\sigma^\circ(\tau))$ is nonzero, then $R_\infty(\tau)$, and thus $R_{\rhobar}^\tau$, is nonzero. 
Conversely, if $R_{\rhobar}^\tau$ is nonzero, then $W_{\obv}(\rhobar) \cap \JH(\ovl{\sigma}(\tau)) \neq \emptyset$ by Proposition \ref{prop:obvintersect}. 
Theorem \ref{thm:obvpd}\eqref{item:allobv} and exactness of $M_\infty$ imply that $M_\infty(\sigma^\circ(\tau))$ is nonzero. 
\end{proof}

The following freeness result follows from our previous results and the Diamond--Fujiwara trick. 

\begin{thm}\label{thm:patchcyclic}
Let $M_\infty$ be a minimal weak patching functor for $\rhobar$. 
Suppose that the equivalent conditions of Theorem \ref{thm:obvpd} hold for $M_\infty$ and that $\sigma \in W_\obv(\rhobar)$.
Then $M_\infty(\sigma)$ is free of rank $1$ over its support $($which is formally smooth over $\F)$.
\end{thm}
\begin{proof}
There exists a generic tame inertial $L$-parameter $\tau$ which exhibits the specialization $(\sigma,\rhobar^\speci) \in SP(\rhobar)$ for some $\F$-valued inertial $L$-parameter $\rhobar^\speci$.
By Lemma \ref{lemma:import}, we can assume without loss of generality that $\tld{w}(\rhobar,\tau) = t_{w^{-1}(e\eta_0)}$ for some $w\in W$.
By Theorem \ref{thm:FSM}, $R_{\rhobar}^{\tau}$ is formally smooth over $\cO$, so that $R_\infty(\tau)$ is as well.
Since for any $\cO$-lattice $\sigma^\circ(\tau) \subset \sigma(\tau)$, $M_\infty(\sigma^\circ(\tau))$ is nonzero, finitely generated, and maximally Cohen--Macaulay over $R_\infty(\tau)$, it must be free over $\ovl{R}_\infty(\tau)$ by Serre's theorem on finiteness of projective dimension and the Auslander--Buchsbaum formula. 
Since the generic rank is at most $1$, its rank must be $1$. 
\end{proof}

\subsection{Global results}\label{sec:TW}

In this section, we discuss algebraic automorphic forms on certain definite unitary groups to which the Taylor--Wiles patching construction can be applied to obtain patching functors as in \S \ref{sec:patchfunc}.
This gives a context to which results in the previous section can be applied.

\subsubsection{Algebraic automorphic forms on some definite unitary groups}\label{sec:autform}

Let $F^+/\Q$ be a totally real field not equal to $\Q$, and let $F \subset \ovl{F}^+$ be a CM extension of $F^+$.
We say that a finite place of $F^+$ is \emph{split} (resp.~ramified or inert) if it splits (resp.~ramifies or is inert) in $F$. 
We say that a place of $F$ is \emph{split} (resp.~ramified or inert) if its restriction to $F^+$ is split (resp.~ramified or inert) in $F$. 

Let $G_{/F^+}$ be a reductive group which is an outer form of $\GL_n$ such that 
\begin{itemize}
\item $G_{/F}$ is an inner form of $\GL_n$; 
\item $G_{/F^+}(F^+_v) \cong U_n(\R)$ for all $v|\infty$; and 
\item $G_{/F^+}$ is quasisplit at all inert and ramified finite places. 
\end{itemize}
By \cite[\S 7.1]{EGH}, $G$ admits a reductive model $\cG$ over $\cO_{F^+}[1/N]$, for some $N\in \N$, and an isomorphism
\begin{equation}
\label{iso integral}
\iota:\,\cG_{/\cO_{F}[1/N]} \stackrel{\iota}{\rightarrow}{\GL_n}_{/\cO_{F}[1/N]}
\end{equation}
which specializes to
$
\iota_w:\,\cG(\cO_{F^+_v})\stackrel{\sim}{\rightarrow}\cG(\cO_{F_w})\stackrel{\iota}{\rightarrow}\GL_n(\cO_{F_w})
$
for all split finite places $w$ in $F$ prime to $N$ where $v$ is $w|_{F^+}$ here.
For each split place $v$ of $F^+$, we choose a place $\tld{v}$ of $F$ dividing $v$. 
For a split $v$ prime to $N$, let $\iota_v$ be the composition of $\iota_{\tld{v}}$ and the canonical isomorphism $\GL_n(\cO_{F_{\tld{v}}}) \cong \GL_n(\cO_{F^+_v})$ (suppressing the dependence on the choice of $\tld{v}$).

Let $S_p$ be the set of all places in $F^+$ dividing $p$.
Suppose from now on that all places in $S_p$ are split.
If $U = U_p U^{\infty,p} \leq G(\A_{F^+,p}^{\infty}) \times G(\A_{F^+}^{\infty,p})$ is a compact open subgroup and $W$ is a finite $\cO$-module endowed with a continuous action of $U_\Sigma$ for some finite set $\Sigma$ of finite places of $F^+$, then we define the space of algebraic automorphic forms on $G$ of level $U$ and coefficients in $W$ to be the (finite) $\cO$-module 
\begin{equation}
S(U,W) \defeq \left\{f:\,G(F^{+})\backslash G(\A^{\infty}_{F^{+}})\rightarrow W\,|\, f(gu)=u_{\Sigma}^{-1}f(g)\,\,\forall\,\,g\in G(\A^{\infty}_{F^{+}}), u\in U\right\}.
\end{equation}
We recall that the level $U$ is said to be \emph{sufficiently small} if for all $t \in G(\bA^{\infty}_{F^+})$, the order of the finite group $t^{-1} G(F^+) t \cap U$ is prime to $p$.
If $U$ is sufficiently small, then $S(U,-)$ defines an exact functor from finite $\cO$-modules with a continuous $U_p$-action to finite $\cO$-modules.
From now on we assume that $U$ is sufficiently small.

For a finite place $v$ of $F^+$ prime to $N$, we say that $U$ is \emph{unramified} at $v$ if one has a decomposition $U=\cG(\cO_{F^+_v})U^{v}$. 
Let $S$ be a finite set of finite places in $F^+$ containing all places dividing $pN$, $\Sigma$, and all places at which $U$ is \emph{not} unramified.

Let $\cP_S$ be the set of split finite places $w$ of $F$ such that $w|_{F^+} \notin S$. 
For any subset $\cP\subseteq \cP_S$ of finite complement that is closed under complex conjugation, we write $\bT_{\cP}\defeq\cO[T^{(i)}_w,\,\,w\in\cP,\, 0 \leq i \leq n]$ for the universal Hecke algebra on $\cP$.
The space of algebraic automorphic forms $S(U,W)$ is endowed with an action of $\bT_{\cP}$, where $T_w^{(i)}$ acts by the usual double coset operator
\[
\iota_w^{-1}\left[ \GL_n(\cO_{F_w}) \left(\begin{matrix}
      \varpi_{w}\mathrm{Id}_i & 0 \cr 0 & \mathrm{Id}_{n-i} \end{matrix} \right)
\GL_n(\cO_{F_w}) \right].
\]
Let $\bT_{\cP}(U,W)$ be the image of $\bT_{\cP}$ in $\End_{\cO}(S(U,W))$---it is a finite flat $\cO$-algebra and in particular a complete semilocal ring. 
Enlarging $E$ if necessary, we assume that the residue fields are identified with $\F$. 
If $Q$ is the (finite) set $\{w|_{F^+}:w\in \cP_S \setminus \cP\}$, then we also denote $\bT_{\cP}(U,W)$ by $\bT^Q(U,W)$.
For a maximal ideal $\fm \subset \bT^Q(U,W)$, there is a semisimple Galois representation $\rbar \defeq \rbar_{\fm}: G_{F^+,S} \ra \cG_n(\F)$, where $\cG_n$ is the group scheme over $\Z$ defined in \cite[\S 2.1]{CHT}, uniquely determined by the equation 
\begin{equation}\label{eqn:heckefrob}
\det\left(1-\rbar_{\fm}|_{G_F}(\mathrm{Frob}_w)X\right)=\sum_{j=0}^n (-1)^j(\mathbf{N}_{F/\Q}(w))^{\binom{j}{2}}(T_w^{(j)} \hspace{-2mm}\mod \fm)X^j.
\end{equation}
\begin{defn}
We say that such a Galois representation $\rbar: G_{F^+,S} \ra \cG_n(\F)$ is \emph{automorphic} of level $U$ and coefficients $W$ if $\rbar$ satisfies \eqref{eqn:heckefrob} for a finite subset $Q \subset \cP_S$ closed under complex conjugation and a maximal ideal $\fm\subset \bT^Q(U,W)$. 
In this case, we say that $\fm$ is the maximal ideal (of $\bT^Q(U,W)$ or $\bT_{\cP}$) corresponding to $\rbar$. 

We say that $\rbar$ is \emph{automorphic} if $\rbar$ is automorphic of some level $U$ and some coefficients $W$.
\end{defn}

We now suppose that $\rbar_{\fm}$ is absolutely irreducible. 
Let $\alpha: \bT_{\cP} \onto \bT^Q(U,W)_{\fm}$ be the natural quotient map.
Then there is a Galois representation $r_{\fm} \defeq r(U,W)_{\fm}: G_{F^+,S} \ra \cG_n(\bT^Q(U,W)_{\fm})$ determined by the equations
\[
\det\left(1-r(U,W)_{\fm}|_{G_F}(\mathrm{Frob}_w)X\right)=\sum_{j=0}^n (-1)^j(\mathbf{N}_{F/\Q}(w))^{\binom{j}{2}}\alpha(T_w^{(j)})X^j
\]
for all $w\in \cP$.

For each $v\in S_p$, there is an isomorphism $\iota_v: G_{/F^+_v}\cong G_{/F_{\tld{v}}} \cong \GL_{d_v}(D_{\tld{v}/F_{\tld{v}}})$ for some $d_v \in \N$ and some central division algebra $D_{\tld{v}}$ over $F_{\tld{v}}$ where $\GL_{d_v}(D_{\tld{v}/F_{\tld{v}}})(R) \defeq \GL_{d_v}(D_{\tld{v}} \otimes_{F_{\tld{v}}} R)$. 
We now let $U_v$ be $\iota_v^{-1}(\GL_{d_v}(\cO_{D_{\tld{v}}}))$ and $U_p$ be $\prod_{v\in S_p} U_v = \iota_p^{-1} (\prod_{v\in S_p}\GL_{d_v}(\cO_{D_{\tld{v}}}))$ where $\iota_p$ denotes $\prod_{v\in S_p} \iota_v$. 

\begin{defn}
\label{defn:mod:wght}
Suppose that $U^{\infty,p}$ is such that $U = U_p U^{\infty,p}$ is a sufficiently small compact open subgroup of $G(\bA_{F^+}^\infty)$ and let $\sigma$ be an irreducible representation of $\prod_{v\in S_p}\GL_{d_v}(\cO_{D_{\tld{v}}})$ over $\F$. 

We say that $\rbar$ is \emph{automorphic of weight} $\sigma$ and level $U$ if $\rbar$ is automorphic of level $U$ and coefficients $\sigma^\vee \circ \iota_p$, where $\sigma^\vee$ denotes the $\F$-dual of $\sigma$. 
We say that $\rbar$ is automorphic of weight $\sigma$ or $\sigma$ is a \emph{modular} (\emph{Serre}) \emph{weight} for $\rbar$ if $\rbar$ is automorphic of weight $\sigma$ and some level $U$. 

Let $W(\rbar)$ be the set of modular Serre weights of $\rbar$.
\end{defn}

For each $v$, we fix an embedding $\ovl{F}^+ \into \ovl{F}^+_v$ such that the restriction $F \into \ovl{F}^+_v$ induces the place $\tld{v}$. 
Let $\rbar_v$ be the restriction of $\rbar$ to $G_{F^+_v} \cong G_{F_{\tld{v}}}$, and let $\rbar_p$ be the $L$-homomorphism over $\F$ corresponding to the collection $(\rbar_v)_{v\in S_p}$. 
One expects that $W(\rbar)$ depends only on $\rbar_p$. 

\subsubsection{Minimal level}\label{sec:minlevel}

We now introduce a space of modular forms at \emph{minimal level}. 
Suppose that $F/F^+$, $G$, and $\rbar$ are as before. 
Assume moreover that $F/F^+$ is unramified at all finite places and that $\rbar$ is ramified only at split places. 

We begin with some notation and terminology. 
If $v$ is a split place of $F^+$, then we define the minimally ramified type $\tau_v$ at $v$ (with respect to $\rbar$) to be the inertial type obtained from the restriction to inertia of any minimally ramified lift of $\rbar|_{G_{F^+_v}}$ in the sense of \cite[Definition 2.4.14]{CHT}). 

Let $v_1$ be a split place of $F^+$ away from $p$ such that 
\begin{itemize}
\item $v_1$ does not split completely in $F(\zeta_p)$; and
\item $\rbar|_{G_{F^+_{v_1}}}$ is unramified and $\rbar(\Frob_{F^+_{v_1}})$ has distinct eigenvalues, no two of which have ratio equal to $(\mathbf{N}v_1)^{\pm 1}$. 
\end{itemize}
(It is possible to find such a $v_1$ if $\rbar(G_F)$ contains $\GL_n(\F')$ with $\# \F' > 3n$, see \cite[\S 2.3]{CEGGPS}.) 

Let $U \subset G(\A_{F^+}^\infty)$ be the compact open subgroup $\prod_v U_v$ where $U_v$ is 
\begin{itemize}
\item $\iota_{\tld{v}}^{-1}(\GL_n(\cO_{F_{\tld{v}}}))$ if $v$ is a split place of $F^+$ not equal to $v_1$;
\item the preimage of the upper triangular matrices under the composition 
\[
G(\cO_{F^+_{v_1}}) \overset{\iota_{\tld{v}_1}}{\ra} \GL_n(\cO_{F_{\tld{v}_1}}) \ra \GL_n(k_{\tld{v}_1}) 
\]
if $v = v_1$; and 
\item hyperspecial if $v$ is an inert place. 
\end{itemize}
Then the compact open subgroup $U$ is sufficiently small. 

Let $\Sigma$ be the set of places of $F^+$ away from $p$ where $\rbar$ ramifies. 
Recall that $S$ is a finite set of places of $F^+$ containing all places dividing $pN$, $\Sigma$, and $v_1$. 
For any subset $\cP\subseteq \cP_S$ of finite complement that is closed under complex conjugation, we write $\bT_{\cP}'\defeq\bT_{\cP}[T^{(i)}_{\tld{v}_1},\, 0 \leq i \leq n]$ where $\bT_{\cP}$ is the universal Hecke algebra on $\cP$ as before. 
For a $U_p$-module $V$, $\bT_{\cP}'$ acts on the space 
\[
S(U, (\otimes_{v\in \Sigma} \sigma(\tau_v^\vee)^\circ \circ \iota_v) \otimes V) 
\] 
where the action of $T^{(i)}_{\tld{v}_1}$ is by the double coset operator $U_{v_1} \iota_{\tld{v}_1}^{-1}\left(\begin{matrix}
      \varpi_{w}\mathrm{Id}_i & 0 \cr 0 & \mathrm{Id}_{n-i} \end{matrix} \right) U_{v_1}$. 

Choose an ordering $\delta_1,\ldots,\delta_n$ of the distinct eigenvalues of $\rbar(\Frob_{\tld{v}_1})$ and let $\fm'$ be the maximal ideal of $\bT_{\cP}'$ generated by $\fm \subset \bT_{\cP}$ and the elements $T^{(i)}_{\tld{v}_1} - (\mathbf{N}v_1)^{i(1-i)/2}(\delta_1\cdots\delta_i)$. 
Then the space $S(U, (\otimes_{v\in \Sigma} \sigma(\tau_v^\vee)^\circ \circ \iota_v) \otimes V)_{\fm'}$ is nonzero. 

\subsubsection{$G$ quasisplit at $p$} 

With $G$ as in \S \ref{sec:autform}, we furthermore suppose in this section that $G_{/F^+_v}$ is quasisplit for all $v\in S_p$, i.e., $G_{/F^+_v} \cong \GL_{n/F^+_v}$. 

\begin{defn}
We say that $\rbar$ is \emph{potentially diagonalizably automorphic} if there is a $U$, $W$, $Q$, and a homomorphism $\lambda: \bT^Q(U,W)_{\fm} \ra \ovl{\Q}_p$ such that if $r_\lambda: G_{F^+} \ra \cG(\ovl{\Q}_p)$ is the attached semisimple Galois representation characterized by the equation 
\begin{equation}\label{eqn:rlambda}
\det\left(1-r_\lambda|_{G_F}(\mathrm{Frob}_w)X\right)=\sum_{j=0}^n (-1)^j(\mathbf{N}_{F/\Q}(w))^{\binom{j}{2}}\lambda(T_w^{(j)})X^j,
\end{equation} 
then $r_{\lambda,v}$ is potentially diagonalizable for all $v \in S_p$.
\end{defn}

\begin{lemma}\label{lemma:globalpatch}
Let $U_p$ be as in \S \ref{sec:autform} and suppose that $U = U_pU^{\infty,p} \subset G(\bA_{F^+}^\infty)$ is a sufficiently small compact open subgroup. 
Let $\Sigma$ be a finite set of finite places of $F^+$ away from $p$. 
Let $W$ be a finite $\cO[U_\Sigma]$-module. 

Then there is a patching functor $M_\infty$ such that for any finite $\F$-module $V$ with a continuous $\prod_{v\in S_p} \GL_n(\cO_{F^+_v})$-action, 
\begin{equation}\label{eqn:mtor}
M_\infty(V)/\fm_\infty \cong S(U, W \otimes_{\cO} V^\vee \circ \iota_p)[\fm]^\vee, 
\end{equation}
where $\fm_\infty\subset R_\infty$ denotes the maximal ideal. 
In particular, $M_\infty(V)$ is nonzero if and only if $S(U, W \otimes_{\cO} V^\vee \circ \iota_p)_\fm)$ is nonzero. 

If $\rbar$ is potentially diagonalizably automorphic, then there is an $M_\infty$ as above which is moreover potentially diagonalizable. 

Suppose now that $F/F^+$ is unramified at all finite places and that $\rbar$ is ramified only at split places. 
Let $U$ and $\fm'$ be as in \S \ref{sec:minlevel}. 
If $W$ is $\otimes_{v\in \Sigma} \sigma(\tau_v^\vee)^\circ \circ \iota_v$ where $\tau_v$ is the minimally ramified type with respect to $\rbar$ and $\sigma(\tau_v^\vee)^\circ \subset \sigma(\tau_v^\vee)$ is an $\cO$-lattice, then there is a minimal patching functor $M_\infty$ such that for any finite $\F$-module $V$ as before, 
\begin{equation}\label{eqn:minmtor}
M_\infty(V)/\fm_\infty \cong S(U, W \otimes_{\cO} V^\vee \circ \iota_p)[\fm']^\vee. 
\end{equation}
If $\rbar$ is potentially diagonalizably automorphic, then this minimal $M_\infty$ can be taken to be potentially diagonalizable. 
\end{lemma}
\begin{proof}
Except for Definition \ref{minimalpatching}\eqref{item:pd} and the minimality, this follows from the proof of \cite[Lemma A.1.1]{MLM} using that $\fm_\infty$ is the preimage of $\fm$ in \emph{loc.~cit.}~under the map $R_\infty \onto R_\infty/\mathfrak{a}_\infty$. 

Suppose the existence of $\tau_0$ as in Definition \ref{minimalpatching}\eqref{item:pd}. 
Then by the above, $\rbar$ is potentially diagonalizably automorphic. 
Let $\tau$ be an inertial $L$-parameter and $x$ be a potentially diagonalizable $\ovl{E}$-point of $\Spec R_\infty(\tau)$. 
There is an $\ovl{E}$-point $y$ of $\Spec R_\infty(\tau)/\mathfrak{a}_\infty$ which is on the same irreducible component of $\Spec R_\infty(\tau)$ as $x$ by \cite[Lemma 3.9]{paskunas-2adic}. 
For any $\cO$-lattice $\sigma(\tau)^\circ \subset \sigma(\tau)$, $M_\infty(\sigma(\tau)^\circ)/\mathfrak{a}_\infty$, and thus $M_\infty(\sigma(\tau)^\circ)$, is supported at $y$ by \cite[Theorem 4.3.1]{LLL} and the properties of $\sigma(\tau)$ (see \S \ref{subsub:ILL}). 
Since $M_\infty(\sigma(\tau)^\circ)$ is a maximal Cohen--Macaulay $R_\infty(\tau)$-module, it is supported at $x$ as well. 

The construction of $M_\infty$ in the minimal level case is as in \cite[\S 4]{le} ($n=3$ and $p$ is assumed to be split, but the modifications are simple). 
\end{proof}

\begin{thm}[Modularity of extremal weights] \label{thm:globalobv}
Let $\rbar: G_{F^+} \ra \cG(\F)$ be an automorphic representation such that 
\begin{itemize}
\item $\rbar|_{G_{F(\zeta_p)}}$ is adequate; and
\item $\rbar_p$ is $6e(n-1)$-generic (in particular $p\nmid 2n$). 
\end{itemize}
Then the following are equivalent: 
\begin{enumerate}
\item $W_{\obv}(\rbar_p) \cap W(\rbar_p) \neq \emptyset$; 
\item $W_{\obv}(\rbar_p) \subset W(\rbar_p)$; and 
\item $\rbar$ is potentially diagonalizably automorphic. 
\end{enumerate}
\end{thm}
\begin{proof}
Using Lemma \ref{lemma:globalpatch} with $U' = U$, the result follows from Theorem \ref{thm:obvpd}. 
\end{proof}

\begin{thm}[Automorphic tameness criterion]
\label{thm:aut:tameness}
Let $\sigma_w,\sigma_{w_0w}\in W_{\obv}(\rbar_p^{\semis})$ be the extremal weights of $\rbar_p^{\semis}$ corresponding to $w$ and $w_0w\in W$, respectively.
Suppose that $\sigma_w\in W(\rbar_p)$.
Then the following are equivalent:
\begin{enumerate}
\item $\sigma_{w_0w}\in W(\rbar_p)$; and
\item $\rbar_p=\rbar_p^{\semis}$.
\end{enumerate}
\end{thm}
\begin{proof}
Use Theorem \ref{thm:globalobv} and Proposition \ref{prop:tamecrit}.
\end{proof}
\begin{cor}
Suppose that $F(\lambda)\in W(\rbar_p)$ for $\lambda\in \un{C}_0$ \emph{(}in particular, $F(\lambda)\in W_{\obv}(\rbar_p)$\emph{)}.
Let $(F(\lambda),\rbar_p^{\speci})\in SP(\rbar_p)$ any lift of $F(\lambda)\in W_{\obv}(\rbar_p)$.
Then $\rbar_p$ is semisimple if and only if 
\[
F\bigg(\Big(t_\eta w_0\,w\,\theta_{\rbar_p}^{\zeta}\big((F(\lambda),\rbar_p^{\speci})\big) \,w_0^{-1}\Big)\cdot(e(w_0(\eta)-\eta))+w_0\cdot(\lambda-\eta)\bigg)\in W(\rbar_p)
\] 
where $w\in\un{W}$ is such that $F(\lambda)$ is the obvious weight of $\rbar_p^{\speci}$ corresponding to $w$.
\end{cor}

\begin{thm}[mod $p$ multiplicity one]\label{thm:multone}
Suppose that $F/F^+$ is unramified at all finite places, $G$ is quasisplit at all finite places, and that if $\rbar|_{G_{F^+_v}}$ is ramified for a finite place $v$ of $F^+$, then $v$ splits in $F$. 
Let $U$ be as in \S \ref{sec:minlevel}. 

Let $\Sigma$ be the set of finite places of $F^+$ away from $p$ at which $\rbar$ is ramified. 
For each $v\in \Sigma$, let $\tau_v$ be the minimally ramified inertial type corresponding to $\rbar|_{G_{F^+_v}}: G_{F^+_v} \ra \GL_n(\F)$. 
If $\rbar$ satisfies the equivalent conditions of Theorem \ref{thm:globalobv}, then for each $\sigma \in W_{\obv}(\rbar_p)$, 
\begin{equation} \label{eqn:multone}
S(U,\otimes_{v\in \Sigma}\sigma^\circ(\tau_v^\vee) \circ \iota_\Sigma \otimes_{\cO} \sigma \circ \iota_p)[\fm]
\end{equation}
is one-dimensional over $\F$. 
\end{thm}
\begin{proof}
This follows from \eqref{eqn:minmtor} and Theorem \ref{thm:patchcyclic}. 
\end{proof}

\begin{rmk}
Using Theorem \ref{thm:multone}, one can recover the main results of \cite{enns-ord} (with stronger genericity assumptions) which assert a multiplicity one statement for the \emph{ordinary part} of \eqref{eqn:multone}. 
\end{rmk}

We require the following ``change of type" result. 

\begin{thm}\label{thm:changetype}
Let $F^+$ be a totally real field and $F\subset\ovl{F}^+$ a CM extension where every place of $F^+$ dividing $p$ splits in $F$. 
Suppose further that $\zeta_p \notin F$. 
For each place $v$ of $F^+$ dividing $p$, choose an embedding $\ovl{F}^+ \into \ovl{F}^+_v$. 

Let $\rbar: G_{F^+} \ra \cG(\F)$ be a Galois representation such that $\rbar(G_{F(\zeta_p)})$ is adequate and there is a RACSDC automorphic representation $\Pi$ of $\GL_n(\A_F)$ such that 
\begin{itemize}
\item $\rbar|_{G_F} \cong \rbar_{p,\iota}(\Pi)$; and
\item for each $v|p$, $r_{p,\iota}(\Pi)|_{G_{F^+_v}}$ is potentially diagonalizable. 
\end{itemize}

Let $\Delta$ be a finite set of places in $F$ away from $p$ which split in $F$ such that if $w\in \Delta$, then $\Pi_w$ is supercuspidal. 
For each place $v$ of $F^+$ dividing $p$, suppose that $\rbar_v$ admits a potentially diagonalizable lift which is potentially crystalline of type $(\lambda_v+\eta_v,\tau_v)$. 

Then there exists a RACSDC automorphic representation $\pi$ of $\GL_n(\A_F)$ such that 
\begin{itemize}
\item $\rbar|_{G_F} \cong \rbar_{p,\iota}(\pi)$; 
\item for each $v|p$, $r_{p,\iota}(\pi)|_{G_{F^+_v}}$ is potentially diagonalizable and potentially crystalline of type $(\lambda_v+\eta_v,\tau_v)$; and 
\item for each $w\in \Delta$, $\pi_w$ is supercuspidal. 
\end{itemize}
\end{thm}
\begin{proof}
This follows from \cite[Theorem 4.3.1]{LLL}, which is based on \cite[Theorem 3.1.3]{BLGG}, except for the assertion of supercuspidality. 
However, \cite[Theorem 3.1.3]{BLGG} with $S$ chosen to contain $S_p$ and $\Delta^+ \defeq \{w|_{F^+} \mid w \in \Delta\}$ guarantees that one can choose $\pi$ so that $r_{p,\iota}(\Pi)|_{G_{F_w}} \sim r_{p,\iota}(\pi)|_{G_{F_w}}$ for each $w\in \Delta$. 
In particular, the irreducibility of $\mathrm{WD}(r_{p,\iota}(\Pi)|_{G_{F_w}})|_{W_{F_w}}$ implies the irreducibility of $\mathrm{WD}(r_{p,\iota}(\pi)|_{G_{F_w}})|_{W_{F_w}}$, which implies the desired assertion. 
\end{proof}

\begin{cor}\label{cor:changetype}
Let $F^+$ be a totally real field and $F\subset\ovl{F}^+$ a CM extension where every finite place of $F^+$ is unramified in $F$ and every place dividing $p$ splits. 
Suppose further that $p$ is unramified in $F^+$ and that $\zeta_p \notin F$. 
For each place $v$ of $F^+$ dividing $p$, choose an embedding $\ovl{F}^+ \into \ovl{F}^+_v$. 
Fix a set $\Delta$ of split places in $F$ away from $p$. 

Let $\rbar: G_{F^+} \ra \cG(\F)$ be a Galois representation such that $\rbar(G_{F(\zeta_p)})$ is adequate, $\rbar|_{G_{F^+_v}}$ is $(6n-2)$-generic for all $v|p$, and there is a RACSDC automorphic representation $\Pi$ of $\GL_n(\A_F)$ such that 
\begin{itemize}
\item $\rbar|_{G_F} \cong \rbar_{p,\iota}(\Pi)$; 
\item for each $v|p$, $r_{p,\iota}(\Pi)|_{G_{F^+_v}}$ is potentially diagonalizable; and 
\item for each $w\in \Delta$, $\Pi_w$ is supercuspidal. 
\end{itemize}
For each place $v$ of $F^+$ dividing $p$, let $\tau_v$ be a tame inertial type. 
Then the following are equivalent: 
\begin{enumerate}
\item \label{item:locallift} $R_{\rbar_v}^{\tau_v}$ is nonzero for all places $v$ of $F^+$ dividing $p$; 
\item \label{item:globallift} there is a RACSDC automorphic representation $\pi$ of $\GL_n(\A_F)$ such that 
\begin{itemize}
\item $\rbar|_{G_F} \cong \rbar_{p,\iota}(\pi)$; 
\item $r_{p,\iota}(\pi)|_{G_{F^+_v}}$ is potentially crystalline of type $(\eta_v,\tau_v)$ for all $v|p$; and 
\item for each $w\in \Delta$, $\pi_w$ is supercuspidal. 
\end{itemize}
\end{enumerate}
\end{cor}
\begin{proof}
\eqref{item:globallift} immediately implies \eqref{item:locallift}. 
We now assume \eqref{item:locallift} and show the converse. 
\eqref{item:locallift} in particular implies that $\tau$ is $(5n-1)$-generic, so that Proposition \ref{prop:obvintersect} applies. 
Indeed, $\tau$ is $(5n-4)$-generic by \cite[Proposition 7]{enns}. 
Then $\tau$ is in fact $(5n-1)$-generic by \cite[Theorem 3.2.1]{LLL}. 
Let $\Delta^+$ be the set $\{w|_{F^+} \mid w\in \Delta\}$. 
Recall from the proof of Theorem \ref{thm:globalobv} that for each $v\in S_p$, $\rbar_v$ admits a potentially diagonalizable lift of type $(\eta_v,\tau'_v)$ for some tame inertial type $\tau'_v$. 
Let $\pi$ be the RACSDC automorphic representation of $\GL_n(\A_F)$ guaranteed by Theorem \ref{thm:changetype}. 
\cite[Theorem 5.4]{labesse} 
 implies that for some supercuspidal inertial types $(\tau'_v)_{v\in \Delta^+}$, 
\[
S(U,\underset{v\in S_p}{\otimes}\sigma^\circ(\tau^{\prime\vee}_v) \circ \iota_p \otimes \underset{v\in \Delta^+}{\otimes} \sigma^\circ(\tau^{\prime\vee}_v) \circ \iota_{\Delta^+})_{\fm} \neq 0
\] 
where $\sigma^\circ(\tau^{\prime\vee}_v) \subset \sigma(\tau^{\prime\vee}_v)$ is a $\GL_n(\cO_{F^+_v})$-stable $\cO$-lattice for each $v\in S_p \cup \Delta^+$. 
Let $M_\infty$ be the potentially diagonalizable patching functor guaranteed by Lemma \ref{lemma:globalpatch} with $W \defeq \otimes_{v\in \Delta^+} \sigma^\circ(\tau_v^{\prime\vee}) \circ \iota_{\Delta^+}$. 

Theorem \ref{thm:obvpd} implies that $W_{\obv}(\rbar_p) \subset W_{M_\infty}(\rbar_p)$. 
Properties of $M_\infty$ from Lemma \ref{lemma:globalpatch} imply that 
\[
S(U,\underset{v\in S_p}{\otimes}(\sigma^\vee \circ \iota_p) \otimes \underset{v\in \Delta^+}{\otimes} (\sigma^\circ(\tau^{\prime\vee}_v) \circ \iota_{\Delta^+}))_{\fm} \neq 0
\] 
for any $\sigma \in W_{\obv}(\rbar_p)$. 
Exactness of $S(U,-)_{\fm}$ and Proposition \ref{prop:obvintersect} imply that 
\[
S(U,\underset{v\in S_p}{\otimes}(\sigma^\circ(\tau^{\vee}_v) \circ \iota_p) \otimes \underset{v\in \Delta^+}{\otimes} (\sigma^\circ(\tau^{\prime\vee}_v) \circ \iota_{\Delta^+}))_{\fm} \neq 0. 
\] 
We conclude with an application of \cite[Corollaire 5.3]{labesse}. 
\end{proof}

\subsubsection{G anisotropic mod center at $p$}

With $G$ as in \S \ref{sec:autform}, we furthermore suppose in this section that for all $v\in S_p$, $G_{/F^+_v}$ is anisotropic modulo center, i.e., we have an isomorphism $\iota_v: G_{/F^+_v}\risom D_{\tld{v}/F_{\tld{v}}}^\times$. 
We first recall the set of irreducible $\cO_{D_{\tld{v}}}^\times$-representations over $\F$ (or \emph{Serre weights}). 

Let $\fm_{D_{\tld{v}}}\subset \cO_{D_{\tld{v}}}$ denote the maximal ideal. 
Then $k_{D_{\tld{v}}} \defeq \cO_{D_{\tld{v}}}/\fm_{D_{\tld{v}}}$ is a degree $n$ field extension of the residue field $k_{\tld{v}}$ of $F_{\tld{v}}$. 
We say that a character of $\cO_{D_{\tld{v}}}^\times$ is \emph{tame} if it factors through $k_{D_{\tld{v}}}^\times$. 

Since $1+\fm_{D_{\tld{v}}}$ is a pro-$p$ group (under multiplication), it acts trivially on any irreducible $\cO_{D_{\tld{v}}}^\times$-representation over $\F$. 
Thus any irreducible $\cO_{D_{\tld{v}}}^\times$-representation over $\F$ is a tame $\F$-character. 
Moreover, the $\cO$-Teichm\"uller lift gives a bijection between irreducible $\cO_{D_{\tld{v}}}^\times$-representations over $\F$ and tame $\cO$-valued characters of $\cO_{D_{\tld{v}}}^\times$. 

Given a tame character $\chi_v: \cO_{D_{\tld{v}}}^\times \ra k_{D_{\tld{v}}}^\times \ra \cO^\times$, we define a tame inertial type $\tau(\chi_v)$ as follows. 
Let $K_{\tld{v}}$ be $W(k_{D_{\tld{v}}})[p^{-1}] \otimes_{W_{k_{\tld{v}}}[p^{-1}]} F_{\tld{v}}$ and choose an $F_{\tld{v}}$-linear embedding of $K_{\tld{v}} \into \ovl{F}^+_v$. 
We also denote by $\chi_v$ the character $\cO_{K_{\tld{v}}}^\times \ra k_{D_{\tld{v}}}^\times \overset{\chi_v}{\ra} \cO^\times$. 
Then we let $\tau(\chi_v)$ be $\Ind_{W_{K_{\tld{v}}}}^{W_{F^+_v}} (\tld{\chi}_v \circ \mathrm{Art}_{K_{\tld{v}}}^{-1})|_{I_{F^+_v}}$ for an extension $\tld{\chi}_v: K_{\tld{v}}^\times \ra \cO^\times$ of $\chi_v|_{\cO_{K_{\tld{v}}}^\times}$. 
The tame inertial type $\tau(\chi_v)$ does not depend on the choice of embedding $K_{\tld{v}} \into \ovl{F}^+_v$ or extension $\tld{\chi}_v$. 

\begin{lemma}\label{lemma:LGC}
Let $\lambda: \bT^Q(U,W) \ra \ovl{\Q}_p$ be a homomorphism and $r_\lambda: G_{F^+} \ra \cG(\ovl{\Q}_p)$ be the attached semisimple Galois representation characterized by \eqref{eqn:rlambda}. 

Let $\chi = \otimes_{v\in S_p} \chi_v: \prod_{v\in S_p} \cO_{D_{\tld{v}}}^\times \ra \cO^\times$ be a tame character. 
If $\tau(\chi_v)$ is a regular tame inertial type for all $v\in S_p$, then the following are equivalent: 
\begin{enumerate}
\item \label{item:potcrys} for each $v\in S_p$, $r_\lambda|_{G_{F^+_v}}$ is potentially crystalline of type $(\eta_v,\tau(\chi_v))$; and
\item \label{item:chi} $S(U,\chi^\vee \circ \iota_p)_{\ker(\lambda)} \neq 0$. 
\end{enumerate}
If $\tau(\chi_v)$ is not regular and $S(U,\chi^\vee \circ \iota_p)_{\ker(\lambda)} \neq 0$, then $r_\lambda|_{G_{F^+_v}}$ is potentially semistable of type $(\eta_v,\tau_v)$ with $\tau_v$ not regular. 
\end{lemma}
\begin{proof}
Let $\pi$ be the automorphic representation of $G(\A_{F^+})$ corresponding to $\lambda$. 
First suppose that $\tau(\chi_v)$ is regular for all $v\in S_p$. 
We will show that \eqref{item:potcrys} and \eqref{item:chi} are equivalent to 
\begin{equation} \label{eqn:localtype}
\mathrm{rec}_{F_{\tld{v}}}(\mathrm{JL}(\pi_v))|_{I_{F_{\tld{v}}}}\cong\tau(\chi_v)
\end{equation} 
for all $v\in S_p$.

Choosing a subring of $F_{\tld{v},n} \subset D_{\tld{v}}$ which is a degree $n$ unramified field extension of $F_{\tld{v}}$ for each $v\in S_p$, \eqref{item:chi} is equivalent to the fact that for each $v\in S_p$, $\pi_v$ is isomorphic to $\Ind_{F_{\tld{v},n}^\times(1+\fm_{D_{\tld{v}}})}^{D_{\tld{v}}^\times} \tld{\chi}_v$ for some character $\tld{\chi}_v: F_{\tld{v},n}^\times \ra E^\times$ extending $\chi_v$ (see \cite[\S 1.5]{BH11}). 
This is in turn equivalent to \eqref{eqn:localtype} for all $v\in S_p$ by the main result of \cite{BH11}. 

Let $\Pi$ be the automorphic representation of $\GL_n(\A_F)$ in \cite[Proposition 6.5.1]{HKV}. 
Fixing $v\in S_p$, $|\mathrm{LJ}_{G(F_{\tld{v}})}|\Pi_{\tld{v}} \cong \pi_v$ so that $\Pi_{\tld{v}} \cong \mathrm{JL}(\pi_v)$ (see \cite[\S 3]{Bad08}). 
Then 
\begin{equation}\label{eqn:LGC}
\mathrm{WD}(r_\lambda|_{G_{F^+_v}})^{F\textrm{-ss}}|_{W_{F^+_v}} \cong \mathrm{WD}(r_{p,\iota}(\Pi)|_{G_{F_{\tld{v}}}})^{F\textrm{-ss}}|_{W_{F_{\tld{v}}}} \cong \mathrm{rec}_{F_{\tld{v}}}(\mathrm{JL}(\pi_v)\otimes|\det|^{\frac{1-n}{2}})|_{W_{F_{\tld{v}}}}
\end{equation} 
by \cite[Lemma 6.2.2]{HKV}. 
Since $r_\lambda|_{G_{F^+_v}}$ is potentially semistable of weight $\eta_v$ by \cite[Theorem 2.1.1]{BLGGT}, we conclude that \eqref{item:potcrys} is also of equivalent to \eqref{eqn:localtype} for all $v\in S_p$. 

Now suppose that $\tau(\chi_v)$ is not regular for some $v\in S_p$ and that $S(U,\chi^\vee \circ \iota_p)_{\ker(\lambda)} \neq 0$. 
As before $r_\lambda|_{G_{F^+_v}}$ is potentially semistable of type $(\eta_v,\tau_v)$ for some inertial type $\tau_v$. 
We will show that $\tau_v$ is tame and is not regular. 
Let $\Pi$ be as above. 
Then as before, $\Pi_{\tld{v}} \cong \mathrm{JL}(\pi_v)$ so that \eqref{eqn:LGC} holds. 
Since $\mathrm{rec}_{F_{\tld{v}}}(\mathrm{JL}(\pi_v)\otimes|\det|^{\frac{1-n}{2}})|_{I_{F_{\tld{v}}}}$ is tame and is not regular by \cite[Proposition 6.2.3]{HKV}, we conclude that $\tau_v$ is tame and is not regular. 
\end{proof}

\begin{thm}\label{thm:divalg}
Let $F^+$ be a totally real field and $F\subset\ovl{F}^+$ a CM extension where every finite place of $F^+$ is unramified in $F$ and every place dividing $p$ splits. 
Suppose further that $p$ is unramified in $F^+$ and that $\zeta_p \notin F$. 

Let $\rbar: G_{F^+} \ra \cG(\F)$ be an automorphic Galois representation such that $\rbar(G_{F(\zeta_p)})$ is adequate, $\rbar|_{G_{F^+_v}}$ is $(6n-2)$-generic for all $v|p$, and there is a RACSDC automorphic representation $\Pi$ of $\GL_n(\A_F)$ such that 
\begin{itemize}
\item $\rbar|_{G_F} \cong \rbar_{p,\iota}(\Pi)$; 
\item for each $v|p$, $r_{p,\iota}(\Pi)|_{G_{F^+_v}}$ is potentially diagonalizable; and 
\item for each finite place of $w$ of $F$ for which $G_{/F_w}$ is not quasisplit, $\Pi_w$ is supercuspidal. 
\end{itemize}
Let $\chi:D_{\tld{v}}^\times\rightarrow E^\times$ be a character.
Then $\ovl{\chi} \in W(\rbar)$ if and only if $\rbar|_{G_{F^+_v}}$ has a potentially crystalline lift of type $\tau(\chi_v)$ for every $v\in S_p$. 
\end{thm}
\begin{proof}
Suppose that $\ovl{\chi} \in W(\rbar)$. 
Then $S(U,\chi^\vee \circ \iota_p)_\fm$ is a nonzero finite free $\cO$-module so that there exists $\lambda$ as in Lemma \ref{lemma:LGC} such that $S(U,\chi^\vee \circ \iota_p)_{\ker(\lambda)} \neq 0$.  
We first claim that $\tau(\chi_v)$ is regular for every $v\in S_p$. 
If $\tau(\chi_v)$ is not regular for some $v\in S_p$, then Lemma \ref{lemma:LGC} implies that $\rbar|_{G_{F^+_v}}$ has a potentially semistable lift of type $(\eta_v,\tau_v)$ for some tame inertial type $\tau_v$ which is not regular. 
This leads to a contradiction since $\rbar|_{G_{F^+_v}}$ has no such lift by \cite[Proposition 7]{enns}. 
Now since $\tau(\chi_v)$ is regular for every $v\in S_p$, the existence of desired local lifts follows from Lemma \ref{lemma:LGC}. 

Suppose now that $\rbar|_{G_{F^+_v}}$ has a potentially crystalline lift of type $\tau(\chi_v)$ for every $v\in S_p$. 
Then let $\pi$ be as in Corollary \ref{cor:changetype}. 
(As in the proof of Corollary \ref{cor:changetype}, $\tau(\chi_v)$ is $(5n-1)$-generic for all $v\in S_p$ and thus cuspidal.)
Let $\pi'$ be the base change cuspidal automorphic representation of $G(\A_F)$ guaranteed by \cite[Proposition 6.5.2]{HKV}. 
Then $r_{p,\iota}(\pi) \cong r_{p,\iota}(\pi')$ so that in particular $\rbar_{p,\iota}(\pi') \cong \rbar$ and $r_{p,\iota}(\pi')|_{G_{F^+_v}}$ is potentially crystalline of type $(\eta_v,\tau_v)$ for each $v\in S_p$. 
Taking $\lambda$ in Lemma \ref{lemma:LGC} corresponding to $\pi'$, we have that $S(U,\chi^\vee \circ \iota_p)_{\ker(\lambda)}$, and thus $S(U,\ovl{\chi}^\vee \circ \iota_p)_{\fm}$, is nonzero. 
\end{proof}

\begin{rmk}\label{rmk:BM}
In the setting of Theorem \ref{thm:divalg}, Proposition \ref{prop:obvintersect} implies that for $v|p$, $\rbar|_{G_{F^+_v}}$  has a potentially crystalline lift of type $\tau(\chi_v)$ if and only if $W^g(\rbar|_{G_{F^+_v}}) \cap \JH(\ovl{\sigma}(\tau(\chi_v)))$ is nonempty. 
Since $\ovl{\sigma}(\tau(\chi_v)) = \mathrm{JL}_p(\ovl{\chi}_v)$ with $\mathrm{JL}_p$ as defined in \cite[\S 5]{dottoBM}, Theorem \ref{thm:divalg} implies that $\chi \in W(\rbar_p)$ if and only if $\mathrm{JL}_p(\ovl{\chi}_v) \cap W^g(\rbar|_{G_{F^+_v}})$ is nonempty for all $v|p$. 
For general $G$ as in \S \ref{sec:autform}, we would expect that $\sigma \in W(\rbar_p)$ if and only if $\mathrm{JL}_p(\sigma_v) \cap W^{\mathrm{BM}}(\rbar|_{G_{F^+_v}})$ is nonempty for all $v|p$ when the set $W^{\mathrm{BM}}(\rbar|_{G_{F^+_v}})$ is defined (see \cite[\S 1.6, 8]{MLM}), but we use $W^g(\rbar|_{G_{F^+_v}})$ here because it has been defined in greater generality. 
\end{rmk}

\clearpage{}%

\newpage
\bibliography{Biblio}
\bibliographystyle{amsalpha}

\end{document}